\documentclass[letterpaper,11 pt]{article}
\usepackage{algorithmic,algorithm}
\usepackage{amsthm}
\usepackage {amssymb}
\usepackage{graphicx,caption,tabularx,multicol,multirow,color}
\usepackage {amsmath}
\usepackage{natbib}
\usepackage{rotating}
\usepackage{bbm}
\usepackage{xcolor}
\usepackage{slashbox}
\usepackage{pifont}
\usepackage{caption}
\usepackage{tabularx}
\usepackage{lipsum}
 \bibpunct[, ]{(}{)}{,}{a}{}{,}%
\newtheorem{theorem}{Theorem}
\newtheorem{lemma}{Lemma}
\newtheorem{assumption}{Assumption}
\newtheorem{defn}{Definition}
\newtheorem{thm}{Theorem}
\newtheorem{col}{Corollary}
\newtheorem{pre}{Proposition}
\newtheorem{rem}{Remark}
\usepackage{mathrsfs}
\usepackage[bottom=1.00in, left=1.00in, right=1.00in]{geometry}
\usepackage{graphicx}
\usepackage{hyperref}
\usepackage{array}
\usepackage{float}
\usepackage{comment}
\usepackage{setspace}
\newcommand{\tr}[1]{\textnormal{tr}\!\left(#1\right)}
\newcommand{\EX}{\mathbb{E}} 
\newcommand{\nm}[1]{{\color{black}#1}}
\newcommand{\mn}[1]{{\color{black}#1}}
\newcommand{\ks}[1]{{\color{black}#1}}
\newcommand{\cmark}{\ding{51}}%
\newcommand{\xmark}{\ding{55}}%
\allowdisplaybreaks
\begin{document}
\onehalfspacing
\title{First-order Methods with Convergence Rates for Multi-agent Systems on Semidefinite Matrix Spaces}
\author{Nahidsadat Majlesinasab\thanks{School of Industrial Engineering \& Management, Oklahoma State University, Stillwater, OK 74074, USA,  \texttt{nahid.majlesinasab@okstate.edu};}, Farzad Yousefian\thanks{School of Industrial Engineering \& Management, Oklahoma State University, Stillwater, OK 74074, USA,  \texttt{farzad.yousefian@okstate.edu};}, Mohammad Javad Feizollahi\thanks{J. Mack Robinson College of Business, Georgia State University, Atlanta, GA 30303, USA, \texttt{mfeizollahi@gsu.edu}.}}
\setlength{\textheight}{9.2in} \setlength{\topmargin}{0in}\setlength{\headheight}{0in}\setlength{\headsep}{0in}

\maketitle
\thispagestyle{empty}
\begin{abstract} 
The goal in this paper is to develop first-order methods equipped with convergence rates for multi-agent optimization problems on semidefinite matrix spaces. These problems include cooperative optimization problems and non-cooperative Nash games. Accordingly, first we consider a multi-agent system where the agents cooperatively minimize the summation of their local convex objectives, and second, we consider Cartesian stochastic variational inequality (CSVI) problems with monotone mappings for addressing stochastic Nash games on semidefinite matrix spaces. Despite the recent advancements in first-order methods addressing problems over vector spaces, there seems to be a major gap in the theory of the first-order methods for optimization problems and equilibriums on semidefinite matrix spaces. In particular, to the best of our knowledge, there exists no method with provable convergence rate for solving the two classes of problems under mild assumptions. Most existing methods either rely on strong assumptions, or require a two-loop framework where at each iteration, a projection problem, i.e., a semidefinite optimization problem, needs to be solved. 
Motivated by this gap, in the first part of the paper, we develop a mirror descent incremental subgradient method for minimizing a finite-sum function. We show that the iterates generated by the algorithm converge asymptotically to an optimal solution and derive a non-asymptotic convergence rate. In the second part, we consider semidefinite CSVI problems. We develop a stochastic mirror descent method that only requires monotonicity of the mapping. We show that the iterates generated by the algorithm converge to a solution of the CSVI almost surely. Using a suitably defined gap function, we derive a convergence rate statement. This work appears to be the first that provides a convergence speed guarantee for monotone CSVIs on semidefinite matrix spaces. Our numerical experiments performed on a multiple-input multiple-output multi-cell cellular wireless network support the convergence of the developed method.\footnote{A preliminary version of the second part of this work has been accepted for publication in Proceedings of the 2019 American Control Conference (cf. \cite{majlesinasab2018first}).}
\end{abstract}
\section{Introduction}\label{sec_intro}
This paper addresses multi-agent problems over semidefinite matrix spaces including cooperative multi-agent problems and non-cooperative Nash games. First, we consider cooperative multi-agent problems. Decentralized optimization problems have a wide range of applications arising in data mining and machine learning (\cite{nedic2017distributed}), wireless sensor networks (\cite{durham2012distributed}), control (\cite{ram2009distributed}), and other areas in science and engineering (\cite{xiao2006optimal}) where decentralized processing of information is crucial for security purposes or for real-time decision making. In this paper, we consider the following multi-agent finite-sum optimization problem that involves a network of multiple agents who cooperatively optimize a global objective,
\begin{align}
	\label{eqn:problem10} 
	\underset{X \in \mathcal B}{\text{minimize}\ } & \sum_{i=1}^m f_i( X)
\end{align}
\begin{comment}
\begin{equation} \label{eqn:problem10} 
\begin{split}
&\boxed{\begin{array}{ll}
\hbox {minimize} & \sum_{i=1}^m f_i( X) \cr
\hbox{subject to}
&X \in \mathcal B,
\end{array}
}
\end{split}
\end{equation}
\end{comment}
where $\mathcal B\triangleq\{X \in \mathbb{S}_{ n}:  X\succeq 0$ and $\tr{X}=1\}$, and $f_i:\mathcal B \to \mathbb R$ is a convex function. Note that each agent $i$ is associated with the local objective $f_i(X)$ and all agents cooperatively minimize the network objective $\sum_{i=1}^m f_i( X)$. In decentralized optimization, the agents (players) need to communicate with their adjacent agents to spread the distributed information over the network and \ks{reach a consensus}. \par
In the past two decades, there has been much interest in the development of models and distributed algorithms for multi-agent optimization problems. In particular, incremental gradient/subgradient methods and their accelerated aggregated variants (\cite{nedic2009distributed,lobel2011distributed,shi2015extra,gurbuzbalaban2017convergence}) have been studied where a local gradient/subgradient is evaluated at each step of an iteration. Although each step is inexpensive, these methods usually require a large number of iterations to converge. Each iteration in decentralized optimization requires visiting all agents one by one which may cause a significant delay before a transfer of data begins. In this line of research, distributed proximal gradient methods (\cite{bertsekas2011incremental, bertsekas2015incremental}), and alternating direction method of multipliers (ADMM) (\cite{chang2015multi,makhdoumi2017convergence}) were developed and studied extensively as well. These methods have also been extended to applications where the network has a time-varying topology and/or there is a need to asynchronous implementations (\cite{nedic2011asynchronous,nedic2015distributed}). Multi-agent mirror descent method for decentralized optimization was proposed by (\cite{xi2014distributed}) where a local Bregman divergence at each agent is employed, and an asymptotic convergence result is provided. More recently, \cite{boct2018incremental} proposed an incremental mirror descent method with a stochastic sweeping of the component functions. 
While incremental gradient/subgradient methods and their accelerated aggregated variants are extensively studied in vector spaces, their performance and convergence analysis in matrix spaces have not been studied yet. 

The sparse covariance estimation is a specific application of finite-sum problem which sets a certain number of coefficients in the inverse covariance to zero to improve the stability of covariance matrix estimation.
%(\cite{price1972extension}). The goal is to find a sparse representation of the sample data and to highlight independence relationships between the sample variables. 
%To discover the pattern of zeroes, 
%\cite{d2008first} studied a regularized maximum likelihood problem and solved it using two first-order algorithms with low memory requirements, one based on a first-order method developed by \cite{nesterov2005smooth} with overall complexity of $\mathcal O(n^{4.5}/\epsilon)$, the other based on a block-coordinate gradient method which is shown to converge in practice. 
\cite{lu2010adaptive} developed two first-order methods including the adaptive spectral projected gradient and the adaptive Nesterov's smooth methods to solve the large scale covariance estimation problem. \cite{hsieh2013big} proposed a block coordinate descent (BCD) method with a superlinear convergence rate.
%The proposed methods are capable of solving large scale problems with nearly a half million constraints within a reasonable amount of time.
In conic programming, first-order methods are equipped with duality or penalty strategies (\cite{lan2011primal,necoara2017complexity}) to tackle complicated constraints. A major limitation to the aforementioned methods in addressing Problem \eqref{eqn:problem10} is that either they require a projection step that is computationally costly in the semidefinite space, or they employ Lagrangian relaxation techniques slowing down the convergence speed of the underlying first-order method. Accordingly, in the first part of the paper, we address this gap by developing a matrix mirror descent incremental subgradient (M-MDIS) method to solve finite-sum Problem \eqref{eqn:problem10} where we choose the distance generating function to be defined as the quantum entropy following \cite{tsuda2005matrix}. M-MDIS is a first-order method in the sense that it only requires a gradient-type of update at each iteration. This method is a single-loop algorithm meaning that it provides a closed-form solution for the projected point and hence it does not need to solve a projection problem at each iteration. We prove that M-MDIS method converges to the optimal solution of \eqref{eqn:problem10} asymptotically and derive a non-asymptotic convergence rate of $\mathcal O(1/\sqrt{t})$.
%The aforementioned methods are projection based and do not scale with \nm{the} problem size. A summary of these methods is given in Table \ref{tbl:comp}.

In the second part of the paper, we consider non-cooperative multi-agent systems. In addressing such problems, variational inequalities (VIs) were first introduced in the 1960s. VIs have a wide range of applications arising in engineering, finance, physics and economics (cf. \cite{facchinei2007finite}).
They can be used for formulating various equilibrium problems and analyzing them from the viewpoint of existence and uniqueness of solutions and stability. Particularly, in mathematical programming, VIs address problems such as optimization problems, complementarity problems and systems of nonlinear equations, to name a few 
%The simplest way to see a VI is a generalization of the optimality conditions 
(\cite{scutari2010convex}). 
%where the variable $ X$ is a positive semidefinite matrix. 
\mn{Given a set ${\mathcal X}$ and a mapping $F:\mathbf{\mathcal X} \to \mathbb{R}^{n\times n}$, a VI problem denoted by VI$(\mathbf{\mathcal X}, F)$ seeks a matrix $X^* \in \mathbf{\mathcal X}$ such that $\tr{(  X-  X^*)^TF(  X^*)} \geq 0, \quad \text{for all} ~  X \in {\mathcal X}.$
In addressing non-cooperative Nash games, we consider Cartesian stochastic variational inequality (CSVI) problems where the set $\mathcal X$ is a Cartesian product of some component sets $\mathcal X_i$, i.e., 
\begin{align}
\label{eq:setx}
{\mathcal X}\triangleq\{X \in \mathbb{S}_{n}| X=\text{diag} (X_1,\ldots,X_N),~ X_i \in \mathcal X_i\},  \notag\\
 \text{where}~ \mathcal X_i\triangleq\{X_i \in \mathbb S^+_{n_i}| \tr{X_i}=1\} \quad \text{for all} \quad i=1,\ldots,N.
\end{align}
Hence, we seek a matrix $X^*=\text{diag} (X^*_1,\ldots,X^*_N)$ that solves the following inequality for all $i=1,\ldots,N$:
\begin{align}
\label{eq:VI2}
\tr{(  X_i-  X_i^*)^TF_i(  X^*)} \geq 0, \quad \text{for all} ~  X_i \in {\mathcal X_i}.
\end{align}
In particular, we study VI($\mathbf{\mathcal X}, F$) where $F_i(X)=\EX[\Phi_i(  X,\xi_i(w))]$, i.e., the mapping $F_i$ is the expected value of a stochastic mapping $\Phi_i:{\mathcal X}\times \mathbb{R}^{d_i} \to \mathbb{S}_{n}$ where the vector $\xi_i:\Omega \to \mathbb{R}^{d_i}$ is a random vector associated with a probability space represented by $(\Omega, \mathcal{F},\mathbb{P})$. Here, $\Omega$ denotes the sample space, $\mathcal{F}$ denotes a $\sigma$-algebra on $\Omega$, and $\mathbb{P}$ is the associated \nm{probability measure.} Therefore, $X^* \in \mathbf{\mathcal X}$ solves VI($\mathbf{\mathcal X}, F$) if for all $i=1,\ldots,N$,
\begin{align}
\label{eq:VI}
	\tr{(  X_i-  X_i^*)^T\EX[\Phi_i(  X^*,\xi(w))]} \geq 0,~\text{for all}~X_i \in {\mathcal X_i}.
\end{align}
Throughout, we assume that $\EX[\Phi_i(  X^*,\xi_i(w))]$ is well-defined (i.e., the expectation is finite).} 
%This problem is motivated by multi-user stochastic optimization problems and non-cooperative Nash games which we discuss briefly in the following. 
%\subsection{Existing methods}
There are several challenges in solving CSVIs on semidefinite matrix spaces including presence of uncertainty, the semidefinite solution space and the Cartesian product structure. In what follows, we review some of the methods that address these challenges, and explain their limitations. %More details are provided in Table \ref{tbl:comp}.

Stochastic Approximation (SA) schemes (\cite{robbins1951stochastic}) and their prox generalization (\cite{nemirovski2009robust,majlesinasab2017optimal}) shown to be very successful in solving optimization and VI problems (\cite{jiang2008stochastic}) with uncertainties. %While the convergence analysis of this class of solution methods relies on the monotonicity of the gradient mapping, the extragradient methods (\cite{korpelevich1977extragradient,dang2015convergence,juditsky2011solving}) depend on weaker assumptions, i.e., pseudo-monotone mappings to address VIs.
%with smooth and strongly monotone mappings 
%Then, \cite{koshal2013regularized} and \cite{yousefian2017smoothing} extended that work. 
%by addressing merely monotone stochastic VIs. 
%More recently, a regularized smoothing SA method was developed by \cite{yousefian2017smoothing} which address stochastic VIs with non-Lipschitzian and merely monotone mappings. 
%In recent years, prox generalization of SA methods were developed for solving stochastic optimization problems (\cite{nemirovski2009robust,majlesinasab2017optimal}) and VIs. 
 %proposed a stochastic mirror descent (SMD) method to solve stochastic optimization problems with convex objectives. SMD method which generalizes the SA method employs the Bregman distance function in vector spaces equipped with non-Euclidean norms.
%
%The authors provided the convergence rate analysis under smoothness properties of the problem and the distance generator function. 
%The prox generalization of the extragradient schemes to stochastic settings were developed \nm{by} \cite{juditsky2011solving}. 
%which addresses stochastic VIs with monotone operators. 
%The authors proved that under an averaging scheme, the generated iterates by SMP method converge to a weak solution of the stochastic VI. 
%The almost sure convergence of extragradient algorithms in solving stochastic VIs with pseudo-monotone mappings was studied in \cite{kannan2014pseudomonotone}. 
Averaging techniques first introduced \nm{by} \cite{Polyak92} proved successful in increasing the robustness of the SA method.
% In vector spaces equipped with non-Euclidean norms, \nm{\cite{nemirovski2009robust}} developed the stochastic mirror descent (SMD) method for solving nonsmooth stochastic optimization problems. 
Applying SA schemes to solve semidefinite optimization problems result in a two-loop framework and require projection onto a semidefinite cone at each iteration which increases the computational complexity. 
%Under a window-based averaging scheme, the rate of $\mathcal{O}\left(1/\sqrt{t}\right)$ is established. 
%Generalizations of this optimal averaging technique was developed for stochastic mirror prox methods in \cite{yousefian2016stochastic} for addressing SVIs with merely monotone mappings. 
%They also derived the optimal rate of convergence under a strong pseudo-monotone condition. 

Solving optimization problems with positive semidefinite variables is more challenging than solving problems in vector spaces because of the structure of problem constraints. Matrix exponential learning (MEL) which has strong ties to mirror descent methods is an optimization algorithm applied to positive semidefinite nonlinear problems. The distance generating function applied in MEL is the quantum entropy.
\cite{mertikopoulos2012matrix} proposed an MEL based approach to solve the power allocation problem in multiple-input multiple-output (MIMO) multiple access channels.
The convergence of MEL and its robustness w.r.t. uncertainties are investigated by \cite{mertikopoulos2016learning}.
%The single-user MIMO throughput maximization problem 
%\cite{yu2016dynamic} addressed the single-user MIMO throughput maximization problem which is an optimization problem not a Nash game.
%and proposed an online algorithm with a different projection step simpler than matrix exponentials
Although in the above studies, the problem can be formulated as an optimization problem, some practical cases such as multi-user MIMO maximization problem discussed in Section \ref{sec:motiv} cannot be treated as an optimization problem. Hence, \cite{mertikopoulos2017distributed} proposed an MEL based algorithm to solve $N$-player games under uncertain feedback and proved that it converges to a stable Nash equilibrium assuming that the mapping is strongly stable. However, in most applications including the game \eqref{eq:Rgame} this assumption is not met.
In the VI regime, the focus has been more on addressing stochastic VIs (SVIs) on vector spaces. In particular, CSVIs on matrix spaces which have applications in wireless networks and image retrieval (cf. Section \ref{sec:motiv}) have not been studied yet. 
%Much of \nm{the} interest in the VI regime has also focused on addressing VIs on vector spaces. %, i.e., the case that set of variables is a subset of $\mathbb R^n$. 
%Moreover, in the literature of semidefinite programming, most of the methods address deterministic semidefinite optimization. Yet, there are many stochastic systems such as wireless communication systems that can be modeled as positive semidefnite Nash games. 
%These systems consist of several users who seek to maximize/minimize their own profit/payoff. Hence, the problem can be considered as a game and since the profit/payoff is subject to noise it can be solved as a SVI. 
In addressing these limitations, we consider CSVIs on matrix spaces where the mapping is merely monotone. We develop an averaging matrix stochastic mirror descent (A-M-SMD) method to solve CSVI \eqref{eq:VI}. A-M-SMD is a first-order single-loop algorithm. To drive rate statements and to improve its robustness w.r.t. uncertainties, we employ averaging techniques. In the second part of the paper, we improve the MEL method of \cite{mertikopoulos2017distributed} in the sense that we require an applicable assumption on the mapping since strong stability of the mapping either does not hold in applications, or it is hard to be verified. The originality of this work lies in the convergence and rate analysis under the monotonicity assumption. We establish convergence to a weak solution of the CSVI by introducing an auxiliary sequence. Then, we derive a convergence rate of $\mathcal O(1/\sqrt{t})$ in terms of the expected value of a suitably defined gap function. Our work is amongst the first ones that provide a convergence rate for CSVI on semidefinite matrix spaces. In Table \ref{tbl:comp}, the distinctions between the existing methods and our work are summarized. We apply the A-M-SMD method on a throughput maximization problem in wireless multi-user MIMO networks. Our results show that the A-M-SMD scheme has a robust performance w.r.t. uncertainty and problem parameters and outperforms both non-averaging M-SMD and MEL %(\cite{mertikopoulos2017distributed}) 
methods.
\begin{table*}[ht]
\scriptsize
\caption{Comparison of first-order schemes}
\label{tbl:comp}
	\centering
\begin{tabular}{|c|c|c|c|c|c|c|c|}
	\hline
	 Reference & Problem  & Assumptions & Space & Scheme & 1-loop & Rate \\ 
	\hline
%Dang and Lan \cite{dang2015convergence} & VI & Deterministic & MM,S & Vector &\begin{minipage}{2cm} Non-Euclidean \\ extragradient \end{minipage}&  No & ${\cal O} \left({1}/{\sqrt{t}}\right)$ \\
%	\hline 
			\cite{lan2011primal} & Opt & C,S/NS & Matrix & Primal-dual Nesterov's methods& \xmark & ${\cal O} \left({1}/{t}\right)$  \\
	\hline 	
	\cite{hsieh2013big} & Opt & NS,C & Matrix & BCD & \xmark & superlinear \\
	\hline
%\cite{necoara2017complexity} & Opt & C,S/NS & Vector & Inexact Lagrangian & \xmark & ${\cal O} \left({1}/{t^{1.5}}\right)$  \\
\cite{bertsekas2015incremental} & finite-sum & C,S & Vector & Incremental Aggregated Proximal & \xmark & Linear\\
	\hline 
	\cite{gurbuzbalaban2017convergence} & finite-sum & C,S & Vector & Incremental Aggregated Gradient & \xmark & Linear\\
		\hline 
	\cite{boct2018incremental} & finite-sum & C,NS & Vector & Incremental SMD & \xmark & ${\cal O} \left({1}/{\sqrt{t}}\right)$ \\
	\hline 
	\textbf{Our work} & finite-sum & MM, NS & Matrix & M-MDIS & \cmark & ${\cal O} \left({1}/{\sqrt{t}}\right)$ \\
	\hline 
	\cite{jiang2008stochastic} & SVI  & SM,S & Vector & SA & \xmark & $-$ \\
	\hline 
	\cite{juditsky2011solving} & SVI  & PM,S/NS & Vector & Extragradient SMP  & \xmark & ${\cal O} \left({1}/{t}\right)$ \\
	\hline 
		\cite{mertikopoulos2012matrix} & SOpt & C,S  & Matrix & Exponential Learning & \cmark & $e^{-\alpha t}\nm{(\alpha>0)}$ \\
	%\hline 
%\begin{minipage}{2cm}	Mertikopoulos \& \\ Sandholm \cite{mertikopoulos2016learning} \end{minipage}  & SDP & S & convex & Matrix & Exponential Learning & $-$ & ${\cal O} \left(\frac{1}{\sqrt{t}}\right)$ \\
	\hline  
	\cite{koshal2013regularized} & SVI & MM,S & Vector & Regularized Iterative SA & \xmark & $-$ \\
	\hline  
	\cite{yousefian2017smoothing}  & SVI & MM,NS & Vector & Regularized Smooth SA & \xmark & ${\cal O} \left({1}/{\sqrt{t}}\right)$ \\
	\hline 
	%Auslender \cite{auslender2015exact} & SDP & Deterministic &  nonconvex and smooth $f$ & Matrix & Sequential Linear Cone & $-$ & $$ \\
	%\hline 
	\cite{mertikopoulos2017distributed} & SVI & SL,S & Matrix & Exponential Learning & \cmark &${\cal O} \left({1}/{\lambda t}\right)$   \\
	\hline 
	\cite{yousefian2016stochastic} & CSVI  & PM,S & Vector & Averaging B-SMP & \xmark & ${\cal O} \left({1}/{t}\right)$ \\
	\hline
	\textbf{Our work} & SVI & MM, NS & Matrix & A-M-SMD & \cmark & ${\cal O} \left({1}/{\sqrt{t}}\right)$ \\
		\hline 
\end{tabular}
\begin{minipage}{15cm}%
\vspace{0.08in}
 SM: \textit{strongly monotone mapping},\quad MM: \textit{merely monotone mapping}, \quad PM: \textit{psedue-monotone mapping},\quad C: \textit{convex},\\ \quad SL: \textit{strongly stable mapping}, \quad S: \textit{smooth function}\quad NS: \textit{nonsmooth function},\\ Opt: \textit{optimzation problem}, \quad $\lambda$: \textit{strong stability parameter}
  \end{minipage}%
%\begin{tablenotes}
%\item[(1)] SM: strongly monotone mapping, MM: merely monotone, PM: Psedue-monotone S: smooth function,  NS: nonsmooth, C: convex 
%\end{tablenotes}
\end{table*}
%\textbf{Semidefinite and cone programming:}
\begin{rem}
It should be noted that the accelerated variants of first-order methods such as SVRG (\cite{johnson2013accelerating}), SAGA (\cite{defazio2014saga}) and IAG (\cite{gurbuzbalaban2017convergence}) provide improved rate guarantees for optimization and VI problems (\cite{chen2017accelerated}) on vector spaces. Developing this type of methods for solving finite-sum and CSVI problems on matrix spaces and providing their convergence analysis can be a direction for future research. 
\end{rem}
The paper is organized as follows. Section \ref{sec:motiv} presents the motivation and source problems. In Section \ref{sec:pre}, the von Neumann divergence and its main properties are discussed and some results that are applied in the analysis of the paper are established. In Section \ref{sec:coop}, we address the finite-sum Problem \eqref{eqn:problem10}, outline a matrix mirror descent incremental subgradient method and provide its convergence analysis. In Section \ref{sec:algorithm}, we present an averaging matrix stochastic mirror descent algorithm for solving CSVI \eqref{eq:VI} and analyze its convergence. 
%A distributed self-tuned steplength SA (DSSA) scheme for the Lipschitzian CSVI is provided in Section \ref{}. A distributed locally randomized SA scheme is provided in Section \ref{}. 
We report the numerical experiments in Section \ref{sec:num} and conclude in Section \ref{sec:conclusion}.
\par
\textbf{Notation}: Throughout, $\mathbb S_n$ denotes the set of all $n \times n$ symmetric matrices and $\mathbb S_n^+$ the cone of all positive semidefinite matrices. The mapping $F:\mathcal X \to \mathbb R^{n \times n}$ is called monotone if for any $X,Y \in \mathcal X$, we have $\tr{(  X-  Y)(  F (  X)-  F (  Y))}\geq 0$. The set of solutions to VI($\mathbf{\mathcal X},F$) is denoted by $\text{SOL}({\mathcal X}, \mathit F)$.
We define the set $\mathscr X\triangleq\{X\in \mathbb S_n^+|\tr{X} \leq 1\}$. We let $[A]_{uv}$ \ks{denote the components of matrix $A$. $\mathbb C$ is the set of complex numbers.} The spectral norm of a matrix $A$ being the largest singular value of $A$ is denoted by the norm $\Vert A \Vert_2$. The trace norm of a matrix $A$ being the sum of singular values of the matrix is denoted by $\tr{A}$. Note that spectral and trace norms are dual to each other (\cite{fazel2001rank}). We let $\mathbf A^\dagger$ denote the conjugate transpose of matrix $\mathbf A$.
A square matrix $A$ that is equal to its conjugate transpose is called Hermitian. We let $\mathbb H_n$ denote the set of all $n\times n$ Hermitian matrices.
%\begin{rem} 
%\label{rem:hermitian}
%Note that the sum of any two Hermitian matrices is Hermitian too.
%\end{rem}
\section{Motivation and Source Problems}
\label{sec:motiv}
Our research is motivated by the following problems:
\begin{itemize}
\item [(a)]\textit{Example on cooperative multi-agent problems: distributed sparse estimation of covariance inverse} 

Given a set of samples $\{z_i^j\}_{j=1}^{n_i}$ associated with agent $i$, where $z_i^j\sim \mathcal N(\mu,\Sigma)$, $n_i$ is the sample size of the $i$th agent, $\mu \in \mathbb R^d$ and $\Sigma \in \mathbb R^{d\times d}$ are the mean and covariance matrix of a multivariate Gaussian distribution, respectively. To estimate $\mu$ and $\Sigma$, consider the maximum likelihood estimators (MLE) given by  
	\[\hat\mu,\hat \Sigma=\underset{\mu,\Sigma}{\text{argmax}} \prod_{i=1}^m\prod_{j=1}^{n_i} \frac{1}{\sqrt{(2\pi)^{n_i}\text{det}(\Sigma)}}\text{exp}\left(-\frac{1}{2}(z_i^j-\mu)^T\Sigma^{-1}(z_i^j-\mu) \right).
\]
This equation can then be cast as a distributed inverse covariance estimation problem
	\[	\underset{\Sigma^{-1}\succ 0}{\text{min}}-\sum_{i=1}^m\text{log}\left(\text{det}\Sigma^{-1} \right)+\sum_{i=1}^m \tr{S_i \Sigma^{-1}},
\]
where $S_i\triangleq \frac{1}{n_i} \sum_{j=1}^{n_i}-\frac{1}{2}(z_i^j-\hat\mu_i)^T(z_i^j-\hat\mu_i)$ with $\hat \mu_i \triangleq \frac{1}{n_i}\sum_{j=1}^{n_i}z_i^j$. To induce sparsity, consider adding \ks{a lasso penalty of the form $\lambda\|P\ast\Sigma^{-1}\|_1$} to the likelihood as follows
\begin{align}
\label{eq:estimate}
	\underset{\Sigma^{-1}\succ 0}{\text{min}}-\sum_{i=1}^m\text{log}\left(\text{det}\Sigma^{-1} \right)+\sum_{i=1}^m \tr{S_i \Sigma^{-1}}+\lambda \|P\ast\Sigma^{-1}\|_1,
\end{align}
where $P$ is a suitable matrix with nonnegative elements, $\lambda>0$ is the regularization parameter, and $\ast$ denotes element-wise multiplication. For a matrix $A$, we define $\|A\|_1=\sum_{i,j}|[A]_{ij}|$.
Two common choices for $P$ would be the matrix of all ones or this matrix with zeros on the diagonal to avoid shrinking diagonal elements of $\Sigma$ (\cite{bien2011sparse}).
Problem \eqref{eq:estimate} can be viewed as an instance of Problem \eqref{eqn:problem10}, where we define
$f_i(\Sigma^{-1})=-\text{log}\left(\text{det}\Sigma^{-1} \right)+\tr{S_i \Sigma^{-1}}+\frac{\lambda}{m} \|P\ast\Sigma^{-1}\|_1.$ 
\begin{rem}
We propose M-MDIS algorithm to solve Problem \eqref{eqn:problem10}. It should be noted that the constraint $\tr{X}=1$ makes the analysis more complicated. Our Analysis can be easily extended to the cases similar to the sparse covariance estimation problem where this constraint does not exist. 
\end{rem}
\item [(b)]\textit{Stochastic non-cooperative Nash games:}
In a non-cooperative game, $N$ players (users) with conflicting interests compete to minimize their own payoff function. Suppose each player controls a positive semidefinite matrix variable $X_i \in \mathcal X_i$ where $\mathcal X_i$ denotes the set of all possible actions of player $i$. We let $X_{-i}\triangleq(  X_1,...,  X_{i-1},  X_{i+1},...,  X_N)$ denote the possible actions of other players and $f_i(X_i,X_{-i})$ denote the payoff function of player $i$. Therefore, the following Nash game needs to be solved 
\begin{align}
\label{eqn:problem1}
\underset{X_i \in \mathcal X_i} {\text{minimize}}\quad f_i(X_i,X_{-i}), \quad \text{for all}~ i=1,\cdots,N,
\end{align}
\begin{comment}
\begin{align}
\label{eqn:problem1} 
\underset{X_i \in \mathcal X_i} {\text{minimize}}\quad f_i(X_i,X_{-i}).
\end{align}
\begin{equation} \label{eqn:problem1} 
\begin{split}
&\boxed{\begin{array}{ll}
\hbox {minimize} & f_i(  X_1,...,  X_N) \cr
\hbox{subject to}
&  X_i \in \mathcal X_i.
\end{array}
}
\end{split}
\end{equation}
\end{comment}
which includes $N$ semidefinite optimization problems. A solution $X^*=\left(  X^*_1,\ldots,   X^*_N\right)$ to this game, called a Nash equilibrium, is a feasible action profile such that $f_i(  X_i^*,  X_{-i}^*)\leq f_i(  X_i,  X^*_{-i})$, 
%\begin{align}
	%f_i(  X_i^*,  X_{-i}^*)\leq f_i(  X_i,  X_{-i}^*),
%\end{align}
for all \nm{$X_i \in \mathcal X_i=\{X_i \in \mathbb S^+_{n_i}|~ \tr{X_i}=1\}$}, $i=1,\ldots,N$. Later, in Lemma \ref{lemma:nash}, we prove that the optimality conditions of Nash game \eqref{eqn:problem1} can be formulated as a VI$(\mathbf{\mathcal X}, F)$ where $\mathcal X\triangleq \{  X|  X=\text{diag} (  X_1,\cdots,  X_N), \:   X_i\in \mathcal X_i,~\text{for all}~i=1,\ldots,N\}$ and $F(X)\triangleq \text{diag}(\nabla_{  X_1} f_1(  X),\cdots,\nabla_{  X_N} f_N(  X))$. Next, we discuss one of the applications of Problem \eqref{eqn:problem1} in wireless communication network.

\textit{Wireless Communication Networks:} A wireless network is composed of transmitters and receivers that generate and detect radio signals, respectively. An antenna enables a transmitter to send signals into the space, and enables a receiver to pick up signals from the space. In a multiple-input multiple-output (MIMO) wireless transmission system, multiple antennas are applied in transmitters and receivers in order to improve the performance. In some MIMO systems such as MIMO broadcast channels and MIMO multiple access channels, there are multiple users with mutual interferes. In recent years, MIMO systems under uncertainty have been studied where the state channel information is subject to noise, delays and other imperfections (\cite{mertikopoulos2017distributed}).
 %multiple antennas can be used to transmit and receive the radio signals.
%This system is called multiple-input multiple-output (MIMO) which provides high spectral efficiency in single-user wireless links without interference (\cite{foschini1998limits}). 
%Other MIMO systems include MIMO broadcast channels and MIMO multiple access channels, where there are multiple users that mutually interfere. In these systems users either share the same transmitter or the same receiver. 
%In the literature, some effort has also been devoted to MIMO systems where interfering users have different transmitters and receivers, for example, MIMO cellular systems where the signal received by a user in one cell is interfered by the signal transmitted from users in other cells \cite{scutari2009mimo}. 
%Recently, there has been much interest in MIMO systems under uncertainty where the state channel information is subject to measurement errors, delays or other imperfections (\cite{mertikopoulos2017distributed}).  
Here, our problem of interest is the throughput maximization in multi-user MIMO networks under feedback errors. In this network, $N$ MIMO links (users) compete where each link $i$ represents a pair of transmitter-receiver with $m_i$ antennas at the transmitter and $n_i$ antennas at the receiver. 
%We assume each of these links is a player of the game. 
%In this problem, the players are MIMO links and each link represents a pair of transmitter-receiver with $m_i$ antennas at the transmitter and $n_i$ antennas at the receiver.
Let $\mathbf  x_i \in \mathbb C^{n_i}$ and $\mathbf y_i \in \mathbb C^{m_i}$ denote the signal transmitted from and received by the $i$th link, respectively. The signal model can be described by $\mathbf{y}_i={H}_{ii}\mathbf{x}_i+\sum\nolimits_{j \ne i}{H}_{ji}\mathbf{x}_j+\mathbf{\epsilon}_i$,
\begin{comment}
\begin{align}
\label{eq:baseband}
	\mathbf{y}_i={H}_{ii}\mathbf{x}_i+\sum\nolimits_{j \ne i}{H}_{ji}\mathbf{x}_j+\mathbf{\epsilon}_i,
\end{align}
\end{comment}
where ${H}_{ii} \in \mathbb C^{m_i \times n_i}$ is the direct-channel matrix of link $i$, ${H}_{ji} \in \mathbb C^{m_i \times n_j}$ is the cross-channel matrix between transmitter $j$ and receiver $i$, and $\mathbf {\epsilon}_i \in \mathbb C^{m_i}$ is a zero-mean circularly symmetric complex Gaussian noise vector with the covariance matrix $\mathbf I_{m_i}$ (\cite{mertikopoulos2016learning}). 
 Each transmitter $i$ tries to improve its performance % \in \{1,\ldots,N\}, meaning that
by transmitting at its maximum power level. Hence, the action for each player is the transmit power. However, doing so results in a conflict in the system since the overall interference increases and affects the capability of all involved transmitters.
%doing so increases the overall interference in the system, which in turn, adversely impacts the performance of all involved transmitters and presents a conflict. 
Here, we consider the interference generated by other users as an additive noise. Therefore, $\sum_{j \ne i}{H}_{ji}\mathbf{x}_j$ %in relation \eqref{eq:baseband} 
represents the multi-user interference (MUI) received by the $i$th player and generated by other users. 
Assuming the random vector $\mathbf{x}_i$ follows a complex Guassian distribution, transmitter $i$ controls its input signal covariance matrix $  X_i\triangleq\EX[\mathbf{x}_i\mathbf{x}_i^\dag]$ subject to two constraints: first the signal covariance matrix is positive semidefinite and second each transmitter's maximum transmit power is set to a positive scalar $p$.  
%Assuming single user decoding at the receiver which means treating the interference by all other transmitters as additive noise
Under these assumptions, each user's transmission throughput for a given set of users' covariance matrices $X_1,\ldots,X_N$ is given by
\begin{align}
\label{eq:game}
	R_i(  X_i,  X_{-i})&=\log \det\left(\mathbf I_{m_i}+\sum\nolimits_{j=1}^N H_{ji}  X_j  H_{ji}^\dagger\right)-\log \det(W_{-i}),
\end{align}
 where $W_{-i}=\mathbf I_{m_i}+\sum_{j \ne i}H_{ji}X_j H_{ji}^\dagger$ is the MUI-plus-noise covariance matrix at receiver $i$ (\cite{telatar1999capacity}). Let $\mathcal X_i=\{X_i \in \mathbb C^{n_i\times n_i}:X_i\succeq 0$, $\tr{X_i}=p\}$. The goal is to solve 
%$X=\text{diag}\left(X_1,...,X_N\right)$ and 
\begin{align}
\label{eq:Rgame} 
\underset{X_i \in \mathcal X_i} {\text{maximize}}\quad R_i(  X_i,  X_{-i}), \quad \text{for all}~ i=1,\ldots,N.
\end{align}
\begin{comment}
\begin{equation} \label{eq:Rgame} 
\begin{split}
&\boxed{\begin{array}{ll}
\hbox {maximize} & R_i(  X_i,  X_{-i}) \cr
\hbox{subject to}
&  X_i \in \mathcal X_i,
\end{array}
}
\end{split}
\end{equation}
\end{comment}
\end{itemize}
In section \ref{sec:num}, we present the implementation of our scheme in addressing Problem \eqref{eq:Rgame}.
\par
\section{Preliminaries} 
\label{sec:pre}
Suppose $\omega: \text{dom}(\omega) \to \mathbb R$ is a strictly convex and differentiable function, where $\text{dom}(\omega) \subseteq \mathbb{R}^{n\times n}$, and let $X,Y\in \text{dom}(\omega)$. Then, Bregman divergence between $X$ and $Y$ is defined as 
\nm{
%\begin{align*}
	$D(  X,  Y):=\omega(  X)-\omega(  Y)-\tr{(  X-  Y)\nabla\omega(  Y)^T}.$
%\end{align*}
%and its gradient with respect to the first argument is defined as $\nabla_X D(X,Y)=\nabla\omega(X)-\nabla\omega(Y)$ since $\nabla_  A \tr{  A   B}=  B^T$ \cite{tsuda2005matrix} and $\tr{  A+   B}=\tr{  A}+\tr{  B}$.
In what follows, our choice of $\omega$ is the quantum entropy (\cite{vedral2002role}), 
\begin{equation}
\label{eq:entropy}
\omega(X) \triangleq \left\{
\begin{array}{rl}
\tr{X\log X-X}&~\text{if }\quad X\in \mathcal B,\\
+\infty \quad \quad \quad \quad \quad \quad & \quad \text{otherwise},
\end{array} \right.
\end{equation}
%\begin{align}
%\label{eq:entropy}
%\omega(  X)=\tr{  X\log  X-  X},	
%\end{align}
 %, where $\nabla\omega(  X)=(\log   X)^T$ (see \cite{tsuda2005matrix}, \nm{pg.} 998). 
where $\mathcal B\triangleq\{X \in \mathbb{S}_{ n}:  X\succeq 0$ and $\tr{X}=1\}$. The Bregman divergence corresponding to the quantum entropy is called von Neumann divergence and is given by
%\begin{align*}
%	&D(X,Y)=\tr{X\log X-X\log Y-X+Y},
	%\tr{X\log X-X}-\tr{  Y\log   Y-  Y}-\nonumber\\&\tr{(X-Y)(\log Y)}.
%\end{align*}
%For $X,Y \in \mathbb S_n^+$ with $\tr{X}=\tr{Y}=1$, the Bregman divergence reduces to \cite{tsuda2005matrix}
\begin{align}
	D(X,Y)=\tr{X\log X- X\log Y}
\end{align}
}
(\cite{tsuda2005matrix}). In our analysis, we use the following property of $\omega$.
%Next, we address the strong convexity of quantum entropy \cite{yu2013strong}. 
\begin{lemma} (\cite{yu2013strong})
\label{lm:strong} 
Let $\mathscr X\triangleq\{X\in \mathbb S_n^+|\tr{X} \leq 1\}$. The quantum entropy $\omega : \mathscr X\to \mathbb R$ is strongly convex with modulus 1 under the trace norm.
\begin{comment}
 meaning that for all $X,Z \in \mathbf\mathscr X$,
\begin{align*}
\omega(X)\geq 	\omega(Z)+\tr{\nabla_X \omega(X)(X-Z)}+\frac{1}{2}\Vert X-Z \Vert^2_{\text{tr}}.
\end{align*}
We also have for all $X,Y \in \mathcal X$, $D(X,Y):=\tr{ X\log  X-  X\log  Y} \geq \frac{1}{2}\|X-Y\|^2_\text{tr}$.
\end{comment}
\end{lemma}
Since $ \mathcal B\subset\mathbf{\mathscr X}$, the quantum entropy $\omega:{\mathcal B}\to \mathbb R$ is also strongly convex with modulus 1 under the trace norm.
Next, we derive the conjugate of the quantum entropy and its gradient. 
%\subsection{Convergence and Rate Analysis}
%\label{sec:convergence}
%and 
%\begin{align*}
 %\nabla\omega^*(  {Y})=\text{argmax}\{\tr{  X  Y}-\omega(  X) :   X \in \mathcal X\}.
%\end{align*}
\begin{lemma} [Conjugate of von Neumann entropy]
\label{lemma:conjugatevon}
Let $Y \in \mathbb S_{n}$ and $\omega(X)$ be defined as \eqref{eq:entropy}. Then, we have 
%(see \cite{hiai2014introduction}, page 147)
\begin{equation}
\label{eq:conjugate}
\omega^*(Y)=\log(\tr{\exp(Y+I_{n})}),
 \end{equation}
 \begin{equation}
\nabla\omega^*(Y)=\frac{\exp(Y+I_{n})}{\tr{\exp(Y+I_{n})}}.
\label{eq:gradient-omega-star}
\end{equation}
\end{lemma}
\begin{proof}
Note that $\omega$ is a lower semi-continuous convex function on the linear space of all symmetric matrices. The conjugate of function $\omega$ is defined as 
\begin{align}
	\label{eq:supconjugate}
 &\omega^*({Y})=\sup\{\tr{  D  Y}-\omega(  D) : ~  D \in \mathcal B\}=\sup\{\tr{ D  Y}-\tr{  D \log   D-  D} :  D \in \mathcal B \}\notag\\
&=-\inf\{-\underbrace{\tr{  D(  Y+  I_{n})}+\tr{  D \log   D}}_{\text{Term 1}},~  D \in \mathcal B\}.
\end{align}
%Note that the strict convexity of function $\omega_i$ is well known (\cite{nielsen2002quantum}), therefore the minimization problem \eqref{eq:supconjugate} admits a unique solution. 
%The matrix $  Y+  I$ is a Hermitian matrix according to Remark \ref{rem:hermitian}.
%To find the minimizer, we solve the equation 
The minimizer of the above problem is $\displaystyle  D= \frac{\exp(  Y+  I_{n})}{\tr{\exp(  Y+  I_{n})}}$ which is called the Gibbs state (see \cite{hiai2014introduction}, Example 3.29).
%Using the first-order Karush–Kuhn–Tucker (KKT) conditions for the problem \eqref{eq:supconjugate}, we have,
	%\[-(Y_i+  I_{n_i})+\log X_i+I_{n_i}=\mathbf 0.\]
By plugging it into Term 1, we have \eqref{eq:conjugate}. The relation \eqref{eq:gradient-omega-star} follows by standard matrix analysis and the fact that $\nabla_  {Y}\tr{\exp(  Y)}=\exp(  Y)$ (\cite{athans1965gradient}).  We observe that $\nabla\omega^*(Y)$ is a positive semidefinite matrix with a trace equal to one, implying that $\nabla\omega^*(Y) \in \mathcal B$. 
\end{proof}

\begin{comment}
Let $[A]_{uv}$ denote the elements of matrix $A$. The trace of a product of two matrices can be rewritten as the sum of entry-wise products of elements:
\begin{align}
\label{traceofproduct}
\tr{  A^T   B}=\sum_{u}\sum_{v}[{  A} ]_{uv}[  B]_{uv}
\end{align}
\begin{defn}
\label{def:Monotonicity}
(Monotonicity) Consider a mapping $  F : \mathcal X \rightarrow \mathbb R^{n \times n}$. $  F$ is called a monotone mapping if for any $  X,   Y \in \mathcal X$, we have $\tr{(  X-  Y)(  F (  X)-  F (  Y))}\geq 0$.
\end{defn}
%\section{Motivation}\label{sec_mot}
In the following lemma, we prove that the optimality conditions of a differentiable constrained optimization problem can be formulated as a variational inequality on matrix spaces. 
\end{comment}
Next, we show that the optimality conditions of a matrix constrained optimization problem can be formulated as a VI. The proof can be found in the Appendix. %which is an extension of Prop. 1.1.8 in \cite{bertsekas2009convex}.
\begin{lemma} \label{lemma:optimality}
% \in \mathbb H_{n}
Let $\mathcal B \subseteq \mathbb R^{n \times n}$ be a nonempty closed convex set, and let $f:\mathbb R^{n \times n}\to \mathbb R$ be a differentiable convex function. Consider the optimization problem
\begin{align}
\label{eqn:problem2} 
\underset{\widetilde{X} \in \mathcal B} {\text{minimize}}\quad f( \widetilde{  X}).
\end{align}
\begin{comment}
\begin{equation}  
\begin{split}
&\boxed{\begin{array}{ll}
\hbox {minimize} & f( \widetilde{  X}) \cr
\hbox{subject to}
&\widetilde{  X} \in \mathcal B,
\end{array}
}
\end{split}
\end{equation}
\end{comment}
A matrix $\widetilde{  X}^*$ is optimal to Problem \eqref{eqn:problem2} iff $\widetilde{  X}^* \in \mathcal B$ and $\tr{(  Z-\widetilde {  X}^*)^T \nabla f(\widetilde{  X}^*)} \geq 0$, for all $Z\in \mathcal B$.
\end{lemma}
\nm{The next Lemma shows a set of sufficient conditions under which a Nash equilibrium can be obtained by solving a VI.}
\begin{lemma}  
[Nash equilibrium]
\label{lemma:nash}
Let $\mathcal X_i \subseteq \mathbb S_{n_i}$ be a nonempty closed convex set and $f_i(  X_i,  X_{-i})$ be a differentiable convex function in $  X_i$ for all $i=1,\cdots,N$, where $  X_i \in \mathcal X_i$ and $  X_{-i} \in \prod_{j\ne i} {\mathcal X_j}$. Then, $  X^*\triangleq \text{diag}(  X_1^*,\cdots,  X_N^*)$ is a Nash equilibrium (NE) to game \eqref{eqn:problem1} if and only if $  X^*$ solves VI($\mathbf{\mathcal X}, F$), where
\begin{align}
\label{Fdefinition}
&F(X)\triangleq \text{diag}(\nabla_{  X_1} f_1(  X),\cdots,\nabla_{  X_N} f_N(  X)),	\\
\label{xdefinition}
&\mathcal X\triangleq \{  X|  X=\text{diag} (  X_1,\cdots,  X_N), \:   X_i\in \mathcal X_i,~\text{for all} ~ i\}.
\end{align}
\end{lemma}
\begin{proof}
First, suppose $  X^*$ is an NE to game \eqref{eqn:problem1}. We want to prove that $  X^*$  solves VI($\mathbf{\mathcal X}, F$), i.e, $\tr{(Z-X^*)^T F(  X^*)} \geq 0$, for all $Z\in \mathcal X$.
By optimality conditions of optimization problem $\underset{X_i \in \mathcal X_i} {min}\ f_i(X_i,X_{-i})$ and from Lemma \ref{lemma:optimality}, we know $X^*$ is an NE if and only if $\tr{(Z_i-X_i^*)^T\nabla_{X_i} f_i(X^*)} \geq 0$ for all $Z_i\in \mathcal X_i$ \nm{and all} $i=1,\ldots,N$.
Then, we \nm{obtain} for all $i=1,\cdots,N$ 
\begin{align}
\label{eq:lemmfdefinition-2}
&\tr{(  Z_i-  X_i^*)^T\nabla_{X_i} f_i(  X^*)}=\sum_{u}\sum_{v}[Z_i-  X_i^*]_{uv}[\nabla_{X_i} f_i(  X^*)]_{uv}\geq 0.
\end{align}
Invoking the definition of mapping $F$ given by \eqref{Fdefinition} and \nm{from} \eqref{eq:lemmfdefinition-2}, we have 
%Summing the inequalities \eqref{eq:lemmfdefinition-2} for $i=1,\ldots,N$, we obtain $\sum_{i}\sum_{u}\sum_{v}[\nabla_{  X_i} f_i(  X^*)]_{uv}[(  Z_i-  X_i^*)]_{uv}\geq 0$. 
%\begin{align*}
$\tr{(  Z-  X^*)^T F(  X^*)}=\sum_{i,u,v}[Z_i-  X_i^*]_{uv}[\nabla_{X_i} f_i(  X^*)]_{uv}\geq 0.$
%\end{align*}
From the definition of VI($\mathbf{\mathcal X}, F$) and relation \eqref{eq:VI2}, we conclude that $  X^*\in \text{SOL}({\mathcal X}, \mathit F)$. Conversely, suppose $  X^* \in \text{SOL}({\mathcal X}, \mathit F)$. Then, $\tr{(  Z-  X^*)^T F(  X^*)} \geq 0, \text{for all} \:   Z \in \mathcal X$. Consider a fixed $i \in \{1,\ldots,N\}$ and a matrix $\bar{  Z} \in \mathcal X$ given by \eqref{xdefinition} such that the only difference between $  X^\ast$ and $\bar{  Z}$ is in $i$-th block, i.e. 
\begin{align*}
\bar{  Z}= \text{diag}\left(\left[  X_1^*\right],\ldots,\left[  X_{i-1}^*\right],\left[  Z_i\right],\left[  X_{i+1}^*\right],\ldots,\left[  X_{N}^*\right] \right),
\end{align*}
where $  Z_i$ is an arbitrary matrix in $\mathcal X_i$. Then, we have
\begin{align}
\label{eq:z-x}
\bar{  Z}-  X^*= \text{diag}\left(\mathbf 0_{n_1\times n_1},\ldots,\left[  Z_{i} -  X_i^*\right],\ldots,\mathbf 0_{n_N\times n_N} \right).
\end{align}
Therefore, substituting $\bar{  Z}-  X^*$ by term \eqref{eq:z-x}, we obtain 
\begin{align*}
	\tr{(\bar{  Z}-  X^*)^T F(  X^*)}=\sum_{u}\sum_{v}[(  Z_i-  X_i^*)]_{uv}[\nabla_{  X_i} f_i(  X^*)]_{uv}=\tr{(  Z_i-  X_i^*)^T\nabla_{  X_i} f_i(  X^*)}\geq 0.
\end{align*}
Since $i$ was chosen arbitrarily, $\tr{(  Z_i-  X_i^*)^T\nabla_{  X_i} f_i(  X^*)}\geq 0$ for any $i=1,...,N$. Hence, by applying Lemma \ref{lemma:optimality} we conclude that $  X^*$ is a Nash equilibrium to game \eqref{eqn:problem1}. 
\end{proof}
\begin{comment}
\begin{align*}
F(X)\triangleq \text{diag}(\nabla_{  X_1} f_1(  X),\cdots,\nabla_{  X_N} f_N(  X)),
\end{align*}
\begin{align*}
 \mathcal X\triangleq \{  X|  X=\text{diag} (  X_1,\cdots,  X_N), \:   X_i\in \mathcal X_i,~\text{for all} ~ i\}.
\end{align*}
\end{comment} 
\section{Cooperative multi-agent problems} \label{sec:coop}
Consider the multi-agent optimization Problem \eqref{eqn:problem10} on semidefinite matrix spaces. In this section, we present the mirror descent incremental subgradient method for solving \eqref{eqn:problem10}. Algorithm \ref{alg3} presents the outline of the M-MDIS method. The method maintains two matrices for each agent $i$: primal $U_i$ and dual $Y_i$. The connection between the two matrices is via a function $U_{i}=\nabla \omega^*(Y_{i})$ which projects $Y_{i}$ onto the set $\mathcal B$ defined by \eqref{eq:setx}. At each iteration $t$ and for any agent $i$, first, the subgradient of $f_i$ is calculated at $U_{i-1,t}$, denoted by $\tilde\nabla f_i( U_{i-1,t})$. Next, we update the dual matrix by moving along the subgradient. Here $\eta_t$ is a non-increasing step-size sequence. Then, $Y_{i,t}$ will be projected onto the set $\mathcal B$ using the closed-form solution \eqref{eq:project2}.  It should be noted that the update rule \eqref{eq:project2} is obtained by applying Lemma \ref{lemma:conjugatevon}. Finally, the primal and dual matrices of agent $m$, i.e. $U_{m,t}$ and $Y_{m,t}$ are the input to the next iteration. 
\begin{algorithm}
 \caption{Matrix Mirror Descent Incremental Subgradient (M-MDIS)}
\label{alg3}
\begin{algorithmic}
     \STATE 1: \textbf{initialization}: pick $X_{0} \in \mathcal B$, and $Y_{m,-1} \in \mathbb S_n$ arbitrarily.
			 \STATE 2: \textbf{General step}: for any $t= 0, 1, 2, \cdots$ do the following: 
			\begin{itemize}
				 \item [(a)] $U_{0,t}:=X_{t}$ and $Y_{0,t}:=Y_{m,t-1}$
				\item [(b)] For {i=1,...,$m$} do the following:
				\begin{align}
						 &Y_{i,t}:=  Y_{i-1,t}-\eta_{t}\tilde\nabla f_i( U_{i-1,t}) \\
						\label{eq:project2}
				     &U_{i,t}:= \displaystyle\frac{\exp(Y_{i,t}+\mathbf I_n)}{\tr{\exp(Y_{i,t}+\mathbf I_n)}}
				\end{align}
						 \item [(c)] $X_{t+1}:=U_{m,t}$.
\end{itemize} 
\end{algorithmic}
\end{algorithm}
Next, we state the main assumption and discuss its rationality.
\ks{
\begin{assumption} 
\label{ass:boundsub}
Let the set $\mathcal B\triangleq\{X \in \mathbb{S}_{ n}:  X\succeq 0$ and $\tr{X}=1\}$. The functions $f_i$'s are proper and convex on $\mathcal B$. 
%There exists a constant $L_{f_i}$ for which $\|\tilde \nabla f_i(X)\|_2 \leq L_{f_i}$ for all $\tilde \nabla f_i(X)\in \partial f_i(X)$, and $X\in \mathcal B$.
\end{assumption} 
\begin{rem}[Boundedness of subgradients] 
\label{rem:boundsub}
Under Assumption \ref{ass:boundsub}, the union $\underset{X \in \mathcal B}{\cup}\partial f_i(X)$ is nonempty and bounded (\cite{Beck17}, Theorem 3.16). Therefore, there exists a constant $L_{f_i}$ for which $\|\tilde \nabla f_i(X)\|_2 \leq L_{f_i}$ for all $\tilde \nabla f_i(X)\in \partial f_i(X)$, $X\in \mathcal B$ and for all $i=1,\ldots,m$.
\end{rem}
}
We use the following relation in the convergence analysis, 
\begin{align}
\label{eq:twosided}
Y_{i,t}\triangleq \tilde\nabla\omega(U_{i,t})\in \partial \omega (U_{i,t})\Leftrightarrow U_{i,t} \in \partial \omega^\star (Y_{i,t}).	
\end{align}
It should be noted that the above relation holds because $\omega$ is a closed and convex function (\cite{rockafellar1970convex}).
Since $(A-B)^2 \in \mathbb S_n^+$, we have $ 0 \leq \tr{(A-B)^2}=\tr{A^2}-2\tr{AB}+\tr{B^2}$. Therefore, 
\begin{align}
\label{eq:fenchel2}
2\tr{A^TB} \leq \tr{A^2}+\tr{B^2} \leq (\tr{A})^2+n\|B^2\|_2=(\tr{A})^2+n\|B\|_2^2,
\end{align}
where the last inequality follows by positive semidefinteness of matrix $A$ and the relation $\tr{B} \leq n\|B\|_2$. Next, we prove the convergence of M-MDIS algorithm. 
\begin{thm} [asymptotic convergence] \label{th1}
Consider Problem \eqref{eqn:problem10}. Let Assumption \ref{ass:boundsub} hold. Let $\{X_t\}$ be generated by the M-MDIS method with a positive stepsize sequence $\{\eta_t\}$. 
If $\lim_{T\rightarrow \infty}\frac{\sum_{t=0}^{T-1}\eta_t^2}{\sum_{t=0}^{T-1}\eta_t}=0$, then $f_T^{\min}$ converges to $f^*$ as $T\rightarrow \infty$, where $f_T^{\min}\triangleq\underset{t=0,\cdots,T}{\min} ~ f(X_t)$.
\end{thm}
\begin{proof}
Let $Y \in \cap_{i=1}^m \text{dom}f_i$ be fixed. For every $i=1, \cdots,m$ and every $t\geq 0$ we have
\begin{align*}
D(Y,U_{i,t})&=\omega(Y)-\omega(U_{i,t})-\tr{\tilde\nabla^T \omega(U_{i,t})(Y-U_{i,t})}\\
&=\omega(Y)-\omega(U_{i,t})-\tr{(Y_{i,t})^T(Y-U_{i,t})}\\
&=\omega(Y)-\omega(U_{i,t})-\tr{(Y_{i-1,t}-\eta_t \tilde\nabla f_i( U_{i-1,t}))^T(Y-U_{i,t})}\\
&=\omega(Y)-\omega(U_{i,t})-\tr{(Y_{i-1,t})^T(Y-U_{i,t})}+\eta_t\tr{\tilde\nabla^T f_i( U_{i-1,t})(Y-U_{i,t})}\\
&=\omega(Y)-\omega(U_{i,t})-\tr{\tilde\nabla^T\omega (U_{i-1,t})(Y-U_{i,t})}+\eta_t\tr{\tilde\nabla^T f_i( U_{i-1,t})(Y-U_{i,t})},
\end{align*}
where we used relation \eqref{eq:twosided} in the second and last equality and we applied the update rule of the Algorithm \ref{alg3} in the third equality. By adding and subtracting the term $\omega(U_{i-1,t})+\tilde\nabla^T\omega (U_{i-1,t}) U_{i-1,t}$, we get
\begin{align*}
D(Y,U_{i,t})&=\omega(Y)-\omega(U_{i-1,t})-\tr{\tilde\nabla^T\omega(U_{i-1,t})(Y-U_{i-1,t})}
+\omega(U_{i-1,t})-\omega(U_{i,t})\\&-\tr{\tilde\nabla^T\omega(U_{i-1,t})(U_{i-1,t}-U_{i,t})}+\tr{\eta_t \tilde\nabla^T f_i( U_{i-1,t})(Y-U_{i,t})}\\
&=D(Y,U_{i-1,t})-D(U_{i,t},U_{i-1,t})+\eta_t \tr{\tilde\nabla^T f_i( U_{i-1,t})(Y-U_{i,t})}.
\end{align*}
By adding and subtracting the term $\eta_t \tr{\tilde\nabla^T f_i(U_{i-1,t})U_{i-1,t}}$, we have
\begin{align}
\label{eq:5}
&D(Y,U_{i,t})=D(Y,U_{i-1,t})-D(U_{i,t},U_{i-1,t})+\eta_t \tr{\tilde\nabla^T f_i( U_{i-1,t})(Y-U_{i-1,t})} \nonumber\\
&-\eta_t \tr{\tilde\nabla^T f_i( U_{i-1,t})(U_{i,t}-U_{i-1,t})} \leq D(Y,U_{i-1,t})-D(U_{i,t},U_{i-1,t})\nonumber\\&+ \eta_t \left( f_i(Y)-f_i(U_{i-1,t})\right)+\eta_t \tr{\tilde\nabla^T f_i( U_{i-1,t})(U_{i-1,t}-U_{i,t})},
\end{align}
where we used the definition of subgradient in the last relation. Using relation \eqref{eq:fenchel2},
\begin{align}
\label{eq:decop}
\eta_t \tr{\tilde\nabla^T f_i( U_{i-1,t})(U_{i-1,t}-U_{i,t})}\leq n\eta_t^2 \|{\tilde\nabla^T f_i( U_{i-1,t})}\|_2^2+\frac{1}{4}(\tr{U_{i-1,t}-U_{i,t}})^2.
\end{align}
Plugging \eqref{eq:decop} into \eqref{eq:5}, we get
\begin{align*}
D(Y,U_{i,t})& \leq D(Y,U_{i-1,t})-D(U_{i,t},U_{i-1,t})+ \eta_t (f_i(Y)-f_i(U_{i-1,t}))
\\
&+n\eta_t^2\|{\tilde\nabla^T f_i( U_{i-1,t})}\|_2^2+\frac{1}{4}(\tr{U_{i-1,t}-U_{i,t}})^2.
\end{align*}
Using that $\omega$ is 1-strongly convex, Lemma \ref{lm:strong} and definition of Bregman divergence, we get
\begin{align*}
D(Y,U_{i,t})& \leq D(Y,U_{i-1,t})-D(U_{i,t},U_{i-1,t})+ \eta_t \left(f_i(Y)-f_i(U_{i-1,t})\right)
+n\eta_t^2\|{\tilde\nabla^T f_i( U_{i-1,t})}\|_2^2\\
&+\frac{1}{2} D(U_{i,t},U_{i-1,t})=D(Y,U_{i-1,t})+\eta_t \left(f_i(Y)-f_i(U_{i-1,t})\right)+n\eta_t^2\|{\tilde\nabla^T f_i( U_{i-1,t})}\|_2^2\\
& -\frac{1}{2}D(U_{i,t},U_{i-1,t}).
\end{align*}
By Remark \ref{rem:boundsub}, we have for any $i=1,\cdots, m$ and $t\geq 0$
\begin{align*}
D(Y,U_{i,t})&\leq D(Y,U_{i-1,t})+ \eta_t \left(f_i(Y)-f_i(U_{i-1,t})\right)+n\eta_t^2{L_{f_i}}^2-\frac{1}{2}D(U_{i,t},U_{i-1,t}).
\end{align*}
Summing the above inequality over $i=1,\cdots, m$, we obtain
\begin{align*}
D(Y,U_{m,t})&\leq D(Y,U_{0,t})+ \eta_t \sum_{i=1}^m \left(f_i(Y)-f_i(U_{i-1,t})\right)+n\eta_t^2\sum_{i=1}^m{L_{f_i}}^2-\frac{1}{2}\sum_{i=1}^mD(U_{i,t},U_{i-1,t}).
\end{align*}
Note that $U_{0,t}=X_t$. By adding and subtracting the term $\eta_t f(X_t)$, we have
\begin{align}
\label{eq:12}
D(Y,U_{m,t})&\leq D(Y,X_{t})+ \eta_t \sum_{i=1}^m\left(f_i(Y)-f_i(X_t)\right)+\eta_t \sum_{i=1}^m \left(f_i(X_t)-f_i(U_{i-1,t})\right)\nonumber\\
&+n\eta_t^2\sum_{i=1}^m{L_{f_i}}^2-\frac{1}{2}\sum_{i=1}^mD(U_{i,t},U_{i-1,t}). 
\end{align}
By Remark \ref{rem:boundsub}, we have $f_i$ is continuous over $\mathcal B$ with parameter $L_{f_i}>0$, i.e., $|f_i(A)-f_i(B)| \leq L_{f_i}\|A-B\|_2$. Therefore, we have
\begin{align*}
&\sum_{i=1}^m \left(f_i(X_t)-f_i(U_{i-1,t})\right)=\sum_{i=2}^m \sum_{j=1}^{i-1}\left(f_i(U_{j-1,t})-f_i(U_{j,t})\right)	\leq \sum_{i=2}^m \sum_{j=1}^{i-1} L_{f_i}\|U_{j-1,t}-U_{j,t}\|_2 \\
&\leq \left(\sum_{l=1}^m L_{f_l}\right) \sum_{i=1}^m \|U_{i-1,t}-U_{i,t}\|_2
=\left(\sum_{l=1}^m L_{f_l}\right) \sum_{i=1}^m \|\nabla \omega^*(Y_{i-1,t})-\nabla \omega^*(Y_{i,t})\|_2 \\
&\leq \left(\sum_{l=1}^m L_{f_l}\right) \sum_{i=1}^m \|Y_{i-1,t}-Y_{i,t}\|_2, 
\end{align*}
where the last inequality follows by Lipschitz continuity of $\nabla \omega^*$. Applying the update rule of the Algorithm \ref{alg3}, we have
\begin{align}
\label{eq:11}
\sum_{i=1}^m \left(f_i(X_t)-f_i(U_{i-1,t})\right)&\leq \left(\sum_{l=1}^m L_{f_l}\right) \sum_{i=}^m \|\eta_t\tilde \nabla f_i(U_{i-1,t})\|_2
\leq  \eta_t \left(\sum_{l=1}^m L_{f_l}\right) \left(\sum_{i=1}^m L_{f_i}\right),
\end{align}
where the last inequality follows by Assumption \ref{ass:boundsub}. Plugging \eqref{eq:11} into \eqref{eq:12}, for any $t\geq 0$
\begin{align*}
D(Y,U_{m,t})&\leq D(Y,X_{t})+ \eta_t \sum_{i=1}^m\left(f_i(Y)-f_i(X_t)\right)+\eta_t^2 \left(\sum_{i=1}^m L_{f_i}\right)^2\nonumber\\
&+n\eta_t^2\sum_{i=1}^m{L_{f_i}}^2-\sum_{i=1}^m\frac{1}{2}D(U_{i,t},U_{i-1,t}). 	
\end{align*}
Since $\sum_{i=1}^m{L_{f_i}}^2 \leq \left(\sum_{i=1}^m L_{f_i}\right)^2$, also $U_{m,t}=X_{t+1}$, and $Y_{m,t}=Y_{0,t+1}$, we get for any $t
\geq 0$ that
\begin{align*}
D(Y,X_{t+1})&\leq D(Y,X_{t})+ \eta_t \sum_{i=1}^m\left(f_i(Y)-f_i(X_t)\right)+\eta_t^2 (n+1)\left(\sum_{i=1}^m L_{f_i}\right)^2, 	
\end{align*}
where we used the fact that $D(U_{i,t},U_{i-1,t}) \geq 0$. Let $Y:=X^*$, summing up the inequality from $t=0$ to $T-1$, where $T\geq 1$ and rearranging the terms, we get
\begin{align*}
D(X^*,X_{T})+ \sum_{t=0}^{T-1}\eta_t \left(\sum_{i=1}^m f_i(X_t)-\sum_{i=1}^m f_i(X^*)\right) &\leq D(X^*,X_{0})+(n+1)\sum_{t=0}^{T-1} \eta_t^2 \left(\sum_{i=1}^m L_{f_i}\right)^2.	
\end{align*}
By definition of $f^{\min}_{T-1}$, we have
\begin{align*}
\sum_{t=0}^{T-1}\eta_t \left(f^{\min}_{T-1}-f^*\right)
&\leq
\sum_{t=0}^{T-1}\eta_t \left(\sum_{i=1}^m f_i(X_t)-\sum_{i=1}^m f_i(X^*)\right) 
\end{align*}
Since $D(X^*,X_{T})\geq 0$, we get
\begin{align}
\label{eq:15}
f^{\min}_{T-1}-f^* &\leq\frac{ D(X^*,X_{0})+(n+1)\left(\sum_{i=1}^m L_{f_i}\right)^2\sum_{t=0}^{T-1}\eta_t^2 }{\sum_{t=0}^{T-1}\eta_t}.	
\end{align}
By assumption, $\lim_{T\rightarrow \infty}\frac{\sum_{t=0}^{T-1}\eta_t^2}{\sum_{t=0}^{T-1}\eta_t}=0$ which implies $\sum_{t=0}^{T-1}\eta_t\rightarrow +\infty$. Therefore,
%\begin{align*}
$f^{\min}_{T-1}-f^* \rightarrow 0,$	
%\end{align*} 
i.e., $f^{\min}_{T-1}$ converges to $f^*$ as $T\rightarrow \infty$.
\end{proof}
Next, we present the convergence rate of the M-MDIS scheme.
\begin{lemma} (Rate of convergence) Consider Problem \eqref{eqn:problem10}. Suppose Assumption \ref{ass:boundsub} holds and let the sequence $\{X_t\}$ be generated by Algorithm \ref{alg3}. Given a fixed $T\geq 1$, let $\eta_t$ be a sequence given by 
\begin{align}
\label{eq:con-step}
		\eta_t=\frac{1}{\sum_{i=1}^m L_{f_i}}\sqrt{\frac{D(X^*,X_{0})}{n+1}}\frac{1}{\sqrt{T}}. 
\end{align}
Then, we have
		\begin{align}
\label{eq:con-rate}
	f^{\min}_{T-1}-f^* \leq 2\left(\sum_{i=1}^m L_{f_i}\right) \sqrt{\frac{D(X^*,X_0)(n+1)}{T}}={\cal O}\left(\frac{1}{\sqrt T}\right).		
\end{align}
\end{lemma}
\begin{proof}
Assume that the number of iterations $T$ is fixed and the stepsize is constant, i.e, $\eta_t=\eta$ for all $t \geq 0$, then it follows by \eqref{eq:15} that
\begin{align}
\label{eq:bala}
	f^{\min}_{T-1}-f^* \leq\frac{ D(X^*,X_{0})+(n+1)\left(\sum_{i=1}^m L_{f_i}\right)^2\sum_{t=0}^{T-1}\eta^2 }{\sum_{t=0}^{T-1}\eta}.
\end{align}
Then, by minimizing the right-hand side of the above inequality over $\eta>0$, we obtain the constant stepsize \eqref{eq:con-step} for all $t \geq 0$. By plugging \eqref{eq:con-step} into \eqref{eq:bala}, we obtain the rate of the convergence of \eqref{eq:con-rate} for $T\geq 1$. 
		\begin{comment}
		Now assume the number of iterations $N$ is not fixed beforehand. So, the following stepsizes can be substituted by the the constant stepsizes
	\[\eta_t=\frac{1}{\sum_{i=1}^m L_{f_i}}\sqrt{\frac{D(X^*,X_{0})}{n+1}}\frac{1}{\sqrt{t+1}} \quad \text{for all} \quad t\geq0 \]		
By plugging the above stepsize into \eqref{eq:15} and recalling Lemma \ref{lm:bound-t}, for $N\geq 1$, we obtain
			\[f^{\min}_{N-1}-f^* \leq (2+\log(N)) \left(\sum_{i=1}^m L_{f_i}\right)  \sqrt{\frac{D(X^*,X_0)(n+1)}{N}}.\]
\end{comment}
\end{proof}
\section{Stochastic non-cooperative Nash games}\label{sec:algorithm}
In this section, we present the A-M-SMD scheme for solving CSVI \eqref{eq:VI}. 
\nm{Algorithm \ref{alg} presents the outline of the A-M-SMD method. At each iteration $t$ and for any user $i$, first, using an oracle, a realization of the stochastic mapping $F$ is generated at $X_{t}$, denoted by $\Phi_i(  X_t, \xi_t)$. Next, a matrix $Y_{i,t}$ is updated using \eqref{eq:y}. Here $\eta_t$ is a non-increasing step-size sequence. Then, $Y_{i,t}$ will be projected onto the set $\mathcal X_i$ defined by \eqref{eq:setx} using the closed-form solution \eqref{eq:x}. It should be noted that the update rule \eqref{eq:x} is obtained by applying Lemma \ref{lemma:conjugatevon}. Then the averaged sequence $\overline{X}_{i,t+1}$ is generated using relations $\eqref{eq:gamma}$.}
Next, we state the main assumptions. Let us define the stochastic error at iteration $t$ as
\mn{
\begin{align}
\label{eq:z-definition}
	   Z_{i,t} \triangleq \Phi_i(  X_t, \xi_t)- F_i(  X_t) \quad \text{for all}  \quad t\geq 0, \quad \text{and for all} \quad i=1,\ldots,N.
\end{align}
Let $\mathcal F_t$ denote the history of the algorithm up to time $t$, i.e., $\mathcal F_t=\{  X_0, \xi_0,\ldots,\xi_{t-1}\}$ for $t\geq 1$ and $\mathcal F_0=\{  X_0\}$. 
\begin{assumption} Let the following hold:
\label{ass:boundphi}
\begin{itemize}
\item [(a)] The mapping $F(X)=\EX[\Phi(  X_t,\xi_t)]$ is monotone and continuous over the set $\mathbf{\mathcal X}$. 
\item [(b)] The stochastic mapping $\Phi_i(X_t,\xi_t)$ has a finite mean squared error, i.e, there exist scalars $C_i>0$ such that $\EX[\Vert\Phi_i(  X_t,\xi_t)\Vert^2_2|\mathcal F_t] \leq C_i^2$ for all $i=1,\ldots,N$.
%\nm{ (Under this assumption, the mean squared error of the stochastic noise is bounded.)}
\item [(c)] The stochastic noise $Z_{i,t}$ has a zero mean, i.e., $\EX[Z_{i,t}|\mathcal F_t]=\mathbf{0}$ for all $t\geq 0$ and for all $i=1,\ldots,N$.
\end{itemize}
\end{assumption} 
}
\begin{comment}
\begin{rem}
By monotonicity of $F(X)$ and Definition \ref{def:Monotonicity}, we have
\begin{align}
\label{eq:convex}
	\tr{(  {X}-  {Y})(F(  X)-F(  Y))} \geq 0, \quad \text{for all}~   X,   Y \in \mathbf{\mathcal X}.
\end{align}
\end{rem}
\end{comment}
\begin{comment}
, i.e. 
\begin{align*}
	&\EX[\Vert  Z_t\Vert^2_2|\mathcal F_t]=\EX[\Vert\Phi(  X_t, \xi_t)-F(  X_t)\Vert^2_2|\mathcal F_t]= \\
	&\EX[\Vert\Phi(  X_t, \xi_t)\Vert^2_2|\mathcal F_t]+\Vert F(  X_t)\Vert^2_2-\\
	&2\Vert\EX[\Phi(  X_t, \xi_t)|\mathcal F_t]F(  X_t) \Vert_2\leq C^2+\Vert F(  X_t)\Vert^2_2\\ &-2\Vert F(  X_t)\Vert^2_2= C^2-\Vert F(  X_t)\Vert^2_2 \leq C^2,
\end{align*}
where the first inequality follows from Assumptions 1(a)-(b) and the fact that spectral norm of square of a matrix is equal to the suqare of its spectral norm.
\end{rem}
\end{comment}
%\begin{align}
	%\EX[\Phi_{t}|  X_t]=  F(  X_t),
%\end{align}
\begin{algorithm}
 \caption{Averaging Matrix Stochastic Mirror Descent (A-M-SMD)}
\label{alg}
\begin{algorithmic}
\mn{
     \STATE \textbf{initialization}: Set $  Y_{i,0}:=  I_{n_i}/n_i$, a stepsize $\eta_0 > 0$, $\Gamma_0=\eta_0$, let $  X_{i,0} \in {\mathcal X}_i$ be a random initial matrix, and $\overline{X}_{i,0}= X_{i,0}$. 
      \FOR {$t=0,1,...,T-1$}	
			   \FOR {$i=1,...,N$}	
				\STATE Generate $\xi_t$ as realizations of the random variable $\xi$ and evaluate the mapping $\Phi_i(  X_t, \xi_t)$. Let 
				\begin{align} \label{eq:y}
				&Y_{i,t+1} :=   Y_{i,t} - \eta_{t}  \Phi_i(X_t, \xi_t),\\
				\label{eq:x}
				&X_{i,t+1}:=\displaystyle\frac{\exp(  Y_{i,t+1}+  I_{n_i})}{\tr{\exp(  Y_{i,t+1}+  I_{n_i})}}.
				\end{align}
       % \STATE let $Y_{t+1} :=   Y_t - \eta_{t}  \Phi(X_t, \xi_t)$
         % \STATE let $  X_{t+1}:=\nabla \omega^*(Y_{t+1})=\displaystyle\frac{\exp(  Y_{t+1}+  I_n)}{\tr{\exp(  Y_{t+1}+  I_n)}}$
					\STATE Update $\Gamma_t$ and $\overline{X}_{i,t}$ using the following recursions:
					\begin{align}\label{eq:gamma}
						 &\Gamma_{t+1}:=\Gamma_{t}+\eta_{t+1},~\overline{X}_{i,t+1}:=\frac{\Gamma_t\overline{X}_{i,t}+\eta_{t+1}{X}_{i,t+1}}{\Gamma_{t+1}}.
						%\label{eq:xbar}
						 %\overline{X}_{t+1}:=\frac{\Gamma_t\overline{X}_t+\eta_{t+1}{X}_{t+1}}{\Gamma_{t+1}}
					\end{align}
\ENDFOR 
\ENDFOR 
\STATE Return $\overline{X}_{T}$.
}
%\STATE let $\overline{  X}_{T}:=\sum_{t=0}^{T} \frac{\eta_t  X_t}{\sum_{t=0}^{T} \eta_t}$
\end{algorithmic}
\end{algorithm}
\subsection{Convergence and Rate Analysis}
\label{sec:convergence}
\nm{In this section, our interest lies in analyzing the convergence and deriving a rate statement for the sequence generated by the A-M-SMD method.}
\nm{Note that a solution of VI(${\mathcal X}, F$) is also referred to as a strong solution. The convergence analysis is carried out by a gap function $G$ defined subsequently. The definition of $G$ is closely tied with a weak solution which is a counterpart of a strong solution. Next, we define a weak solution.}
\begin{defn}[Weak solution]\label{def:stable} 
The matrix ${X}^*_w \in \mathcal X$ is called a weak solution to VI($\mathbf{\mathcal X}, F$) if it satisfies 
$\tr{(  {X}-  {X}^*_w)^T F(  X)} \geq 0$, for all $X \in \mathbf{\mathcal X}.$
\end{defn}
We let $\mathcal X^\star_w$ and $\mathcal X^*$ denote the set of weak solutions and strong solutions to VI($\mathbf{\mathcal X}, F$), respectively.
\begin{rem}  Under Assumption \ref{ass:boundphi}(a), when the mapping $F$ is monotone, any strong solution of Problem \eqref{eq:VI} is a weak solution, i.e., $\mathcal X^* \subseteq \mathcal X^\star_w$. 
%It can be obtained by setting $  Y=  X^*_w$ in \eqref{eq:convex} and the definition of strong solution of variational inequality $\tr{F(  X^*_w)(  {X}-  {X}^*_w)} \geq 0, ~ \text{for all}~   X \in \mathbf{\mathcal X}$ (\cite{scutari2010convex}), in other words,
%\begin{align*}
%	\tr{F(  X)(  {X}-  {X}^*_w)} \geq \tr{F(  X^*_w)(  {X}-  {X}^*_w)} \geq 0.
%\end{align*}
From continuity of $F$ in Assumption \ref{ass:boundphi}(a), the converse is also true meaning that a weak solution is a strong solution. Moreover, for a monotone mapping $F$ on a convex compact set e.g., $\mathcal X$, a weak solution always exists (\cite{juditsky2011solving}). 
\end{rem}

\nm{Unlike optimization problems where the objective function provides a metric for distinguishing solutions, there is no immediate analog in VI problems. However, different variants of gap function have been used in the analysis of variational inequalities (cf. Chapter 10 in \cite{facchinei02finite}). Here we use the following gap function associated with a VI problem to derive a convergence rate.}
\begin{defn} [$G$ function] Define the following function $G: \mathcal X \to \mathbb R$ as
\label{def:gap}
\begin{align*}
	G(  {X})= \underset {  Z \in \mathcal X} {\sup}\ \tr{(  {X}-  {Z})^T F(  Z)}, \quad \text{for all}~   {X} \in  \mathcal X.
\end{align*}
\end{defn}
The next lemma provides some properties of the $G$ function.
\begin{lemma}
\label{lm:gap properties2}
The function $G({X})$ given by Definition \ref{def:gap} is a well-defined gap function, i.e, $(i)$ $G(  {X})\geq 0$ for all $  X \in \mathcal X$; $(ii)$ $  X^*_w$ is a weak solution to Problem \eqref{eq:VI} iff $G(  {X}^*_w)=0$.
\end{lemma}
\begin{proof}
$(i)$ For an arbitrary $  X \in \mathcal X$, we have
\begin{align*}
		G(  {X})= \underset {  Z \in \mathcal X} {\sup}\ \tr{(  {X}-  {Z})^T F(  Z)} \geq \tr{(  {X}-  {A})^T F(  A)},
\end{align*}
for all ${A} \in  \mathcal X$. For $A=X$, the above inequality suggests that $G(  {X})\geq \tr{({X}-{X})^TF(  X)}=0$ implying that the function $G(  {X})$ is nonnegative for all $  {X} \in  \mathcal X$. \\
$(ii)$ Assume $  X^*_w$ is a weak solution. By Definition \ref{def:stable}, $\tr{(  {X}^*_w-  {X})^TF(  X)} \leq 0$, for all $X \in \mathbf{\mathcal X}$ which implies 
%\begin{align*}
$G(  {X}^*_w)= \underset {  X \in \mathcal X} {\sup}\ \tr{(  {X}^*_w-  {X})^TF(  X)}	 \leq 0$.
%\end{align*}
On the other hand, from Lemma \ref{lm:gap properties2}$(i)$, we get $G(  {X}^*_w) \geq 0$. We conclude that $G(  {X}^*_w) = 0$ for any weak solution ${X}^*_w$. 
Conversely, assume that there exists an ${X}$ such that $G(  {X}) = 0$. Therefore, $\underset {  Z \in \mathcal X} {\sup}\ \tr{(X-Z)^TF(  Z)}=0$ which implies $\tr{(  {Z}-  {X})^TF(  Z)} \geq 0$ for all $  Z \in \mathcal X$. Therefore, ${X}$ is a weak solution. 	
   \end{proof}
%By definition of $  Z_t$, for a fixed nonnegative integer $t$, we have
%\begin{align*}
%\EX[  Z_t|\mathcal F_t]= \EX[\Phi(  X_t, \xi_t)|\mathcal F_t]- F(  X_t)=F(  X_t)-F(  X_t)=0.
%\end{align*}
The proof of the following lemma can be found in Appendix.
\begin{lemma} 
\label{lemma:average}
%Consider problem \eqref{eq:VI} and let Assumption \ref{ass:boundphi} hold. Let $  X_t$ be generated by Algorithm \ref{alg}. 
Assume the sequence $\eta_t$ is non-increasing and the sequence $\overline{X}_{i,t}$ is given by the recursive rule \eqref{eq:gamma} where $\Gamma_0=\eta_0$ and $\overline{X}_{i,0}={X}_{i,0}$. Then, 
\begin{align}
\label{eq:avg}	
\overline{X}_{i,t}=\sum_{k=0}^{t} \left(\frac{\eta_k}{\sum_{k'=0}^{t} \eta_{k'}}\right)  X_{i,k}\quad  \text{for any} ~ t\geq 0.
\end{align}
%\begin{align}
%\label{eq:avg}
%\end{align}
\end{lemma}
\mn{
Throughout, we use the notion of Fenchel coupling (\cite{mertikopoulos2016learning2}):
\begin{equation}
\label{eq:fenchel}
H_i(  {Q_i},  {Y_i})\triangleq\omega_i(  {Q_i})+\omega_i^*(  {Y_i})-\tr{{Q_i}^T  {Y_i}},
\end{equation}
 which provides a proximity measure between $  {Q_i}$ and $\nabla \omega_i^*(  {Y_i})$ and is equal to the associated Bregman divergence between $  {Q}$ and $\nabla \omega_i^*(  {Y_i})$. \nm{We also make use of the following Lemma which is proved in Appendix.}  
%Proposition A.2
\begin{lemma} (\cite{mertikopoulos2017distributed})
\label{pre:smoothstrong}
Let $\mathcal X_i$ be given by \eqref{eq:setx}. For all matrices $X_i\in \mathcal X_i$ and for all ${Y_i},{Z_i} \in \mathbb S_{n_i}$, the following holds
\begin{equation}  
\label{eq:smoothstrong}
H_i(  {X_i},  {Y_i}+  {Z_i})\leq H_i(  {X_i},  {Y_i})+ \tr{  {Z_i}^T(\nabla \omega_i^*(  {Y_i})-  {X_i})}+\Vert   Z_i\Vert^2_2.
\end{equation}
\end{lemma}
Next, we develop an error bound for the G function given by Definition \ref{def:gap}. 
\begin{lemma}
\label{lemma:convergence}
Consider Problem \eqref{eq:VI}. Let $  X_i \in \mathcal X_i$ and the sequence $\{\overline{  X}_t\}$ be generated by A-M-SMD algorithm. Suppose Assumption \ref{ass:boundphi} holds. Then, for any $T \geq 1$, 
\begin{align}  
\label{eq:bound-1}
 \EX[G(\overline{  X}_T)] &\leq \frac{2}{\sum_{t=0}^{T-1}\eta_t}\left(\sum_{i=1}^{N}\log(n_i+1)+\sum\nolimits_{t=0}^{T-1}\eta_t^{2} \sum_{i=1}^{N}C_i^2\right).
%\EX[Gap(\overline{  X}_i)] &\leq \frac{ 2m+\sum_{t=i}^{T}\eta_t^2 (C^2+A^2)}{\sum_{t=i}^{T}\eta_t }.
%\EX[Gap(\overline{  X}_T)] &\leq \frac{1}{\sum_{t=0}^{T-1}\eta_t }\left(m+H(  {X},  U_{0})+\sum_{t=0}^{T-1}\eta_t^2 (C^2+A^2)\right).
%(\sum_{t=0}^{T-1}\eta_t) \EX[Gap(\overline{  X}_T)] &\leq& m+4M^2+\sum_{t=0}^{T-1}\eta_t^2 (\sqrt{rN} C^2+A^2).
\end{align}
\end{lemma}
}
\begin{proof}
\mn{
From the definition of $Z_{i,t}$ in relation \eqref{eq:z-definition}, the recursion in the A-M-SMD algorithm can be stated as
\begin{equation}  
\label{eq:lem-Error bounds-1}
  {Y}_{i,t+1}=  {Y}_{i,t}-\eta_t (  F_i(X_t)+  Z_{i,t}).
\end{equation}
Consider \eqref{eq:smoothstrong}. From Algorithm \ref{alg} and \eqref{eq:gradient-omega-star}, we have $X_{i,t}=\nabla\omega_i^*(Y_{i,t})$. Let $Y_i:=Y_{i,t}$ and $Z_i:=-\eta_t (F_i(X_t)+ Z_{i,t})$. From \eqref{eq:lem-Error bounds-1}, we obtain
\begin{align*}  
&H_i(  {X_i},  Y_{i,t+1})\leq H_i(  {X_i},  Y_{i,t})-\eta_t \tr{(  X_{i,t}-  X_i)^T(  F_i(  X_t)+  Z_{i,t})}+\eta_t^2\Vert F_i(  X_t)+ Z_{i,t}\Vert^2_2.
\end{align*}
By adding and subtracting $\eta_t \tr{(X_{i,t}-X_i)^TF_i(X)}$, we get
\begin{align}
\label{eq:mon}  
&H_i(  {X_i},  Y_{i,t+1}) 
%\leq
% H(  {X},  Y_{t})
%\\&-\eta_t \tr{( X_t- X)(F(X_t)+  Z_t-F(X))}\\&-\eta_t \tr{(  X_t-  X)  F(  X)}+\eta_t^2\Vert   F(  X_t)+  Z_t\Vert^2_2\\
\leq H_i(  {X_i},  Y_{i,t})-\eta_t \tr{(  X_{i,t}-  X_i)^T  Z_{i,t}}-\eta_t \tr{(  X_{i,t}-  X_i)^T(  F_i(  X_t)-F_i(  X)}\notag\\&-\eta_t \tr{(  X_{i,t}-  X_i)^T  F_i(  X)}+\eta_t^2\Vert   F_i(  X_t)+  Z_{i,t}\Vert^2_2.
\end{align}  
%where we used the monotonicity of mapping $F$. 
}
% and the fact that the term $-\eta_t \tr{(  X_t-  X)(  F(  X_t)-  F(  X))}$ is nonpositive
\begin{comment}
Equivalently, we have
\begin{align}
&\eta_t \tr{(X_t-X)F(X)} \leq   H({X},Y_{t})-H({X},Y_{t+1})\notag\\
&-\eta_t \tr{(  X_t-  X)   Z_t}+\eta_t^2\Vert   F(  X_t)+  Z_t\Vert^2_2.
\end{align}
\end{comment}
\mn{
Let us define an auxiliary sequence $U_{i,t}$ such that $  U_{i,t+1}\triangleq  U_{i,t}+\eta_t  Z_{i,t}$, where $  U_{i,0} =\mathbf I_{n_i}$ and define $  V_{i,t}\triangleq   \nabla \omega^*_i(  U_{i,t})$. From \eqref{eq:mon}, invoking the definition of $  Z_{i,t}$ and by adding and subtracting $  V_{i,t}$, we obtain
\begin{align}  
\label{eq:13}
&\eta_t \tr{(X_{i,t}-X_i)^TF_i(X)} \leq H(X_i,Y_{i,t})-H_i(X_i,Y_{i,t+1})-\eta_t \tr{(  X_{i,t}-  X_i)^T(  F_i(  X_t)-F_i(  X)}\notag\\
&+\eta_t \tr{(V_{i,t}-X_{i,t})^T Z_{i,t}}+\eta_t \tr{(X_i-V_{i,t})^T Z_{i,t}}+\eta_t^2\Vert \Phi_{i,t}\Vert^2_2,
\end{align}
where for simplicity of notation we use $\Phi_{i,t}$ to denote $\Phi_i({X}_t,\xi_t)$. Then, we estimate the term $\eta_t \tr{(X_i-V_{i,t})^T Z_{i,t}}$. By Lemma \ref{pre:smoothstrong} and setting $  Y_i:=  U_{i,t}$ and $  Z_i:=\eta_t  Z_{i,t}$, we get
}
\begin{comment}
\begin{align*}
	H_i(  {X_i},  U_{i,t+1}) \leq H_i(  {X_i},  U_{i,t})+\eta_t \tr{(  V_{i,t}-  X_i)^T  Z_{i,t}}+\eta_t^2\Vert   Z_{i,t}\Vert^2_2,
\end{align*}
%Implying the non-expansivity of the projection operator, we get
%\begin{align}
	%\Vert   U_{t+1}-   {X}\Vert ^2_F \leq  \Vert   U_{t}+\eta_t   Z_t-   {X}\Vert ^2_F=\Vert   U_{t}-   {X}\Vert ^2_F+2\eta_t\tr[(  U_{t}-  {X})  Z_t]+\eta_t^2\Vert   Z_t \Vert ^2_F
%\end{align}
 %where the last equality follows by $\Vert A+B \Vert ^2_F=\Vert A \Vert ^2_F+\Vert B \Vert ^2_F+2\tr[AB]$. 
and therefore, we have
\end{comment}
\mn{
\begin{align*}
\eta_t \tr{(X_i-V_{i,t})^T Z_{i,t}}&\leq H_i({X_i},U_{i,t})-H_i({X_i},U_{i,t+1})+\eta_t^2\Vert Z_{i,t}\Vert^2_2.
%	\eta_t \tr[(  X-  U_t)   Z_t]\leq \Vert   U_{t}-   {X}\Vert ^2_F-\Vert   U_{t+1}-   {X}\Vert ^2_F+%\eta_t^2\Vert   Z_t \Vert ^2_F.
\end{align*}
By plugging the above inequality into \eqref{eq:13}, we get
\begin{align*}
&\eta_t \tr{(X_{i,t}-X_i)^T F_i(X)}\leq H_i({X_i},Y_{i,t})-H_i({X_i},Y_{i,t+1})+H_i({X_i},U_{i,t})-H_i({X_i},  U_{i,t+1})\\
	&+\eta_t^2\Vert   Z_{i,t}\Vert^2_2+\eta_t \tr{(  V_{i,t}-  X_{i,t})^T   Z_{i,t}}+\eta_t^2\Vert \Phi_{i,t}\Vert^2_2-\eta_t \tr{(  X_{i,t}-  X_i)^T(  F_i(  X_t)-F_i(  X)}. %\\ & \leq &  
	%D(  {X},  Y_{t})+\Vert   U_{t}-   {X}\Vert ^2_F
	%+\eta^2\Vert   Z_t \Vert ^2_F+\eta_t \tr[(  U_t-  X)   Z_t]+\eta_t^2\Vert \Phi_t\Vert^2_2
\end{align*}
%By multiplying both sides by $\eta_t^{r-1}$ where $r\geq 1$ is a constant, we get,
%\begin{align*}
%\eta_t^r \tr[(  X_t-  X)  numerF(  X)] &\leq \eta_t^{r-1}H(  {X},  Y_{t})-\eta_t^{r-1}H(  {X},  Y_{t%+1})+\eta_t^{r-1}H(  {X},  U_{t})-	\eta_t^{r-1}H(  {X},  U_{t+1})+\eta_t^{r+1}\Vert   Z_t\Vert^2_2\\
	%&+\eta_t^r \tr[(  V_t-  X_t)   Z_t]+\eta_t^{r+1}\Vert \Phi_t\Vert^2_2.   
%\end{align*}
%Then, we add and subtract $\eta_t^{r-1}H(  {X},  Y_{t+1})$ and $\eta_t^{r-1}H(  {X},  U_{t+1})$ to the right-hand side of the inequality. Summing over $t=0,\ldots,T-1$, and rearranging gives,
Let us define $V_t:=\text{diag}\ (V_{1,t},\ldots,V_{N,t})$. By summing the above inequality form $i=1$ to $N$, we get
\begin{align*} 
&\eta_t \tr{(X_t-X)^T F(X)}\leq \sum\nolimits_{i=1}^{N} H_i({X_i},Y_{i,t})-\sum\nolimits_{i=1}^{N}H_i({X_i},Y_{i,t+1})+\sum\nolimits_{i=1}^{N}H_i({X_i},U_{i,t})\\&-\sum\nolimits_{i=1}^{N}H_i({X_i},  U_{i,t+1})+\eta_t^2\sum\nolimits_{i=1}^{N}\Vert   Z_{i,t}\Vert^2_2+\eta_t \tr{(  V_t-  X_t)^T   Z_t}+\eta_t^2\sum\nolimits_{i=1}^{N}\Vert \Phi_{i,t}\Vert^2_2,
\end{align*}
where we used the monotonicity of mapping $F$, i.e. $\tr{(X_t-X)(F(X_t)-F(X))}\geq 0$ . By summing the above inequality form $t=0$ to $T-1$, we have
\begin{align} 
&\sum_{t=0}^{T-1}\eta_t \tr{(X_t-X)^TF(X)}\leq \sum_{i=1}^{N} H_i({X_i},Y_{i,0})-\sum_{i=1}^{N}H_i({X_i},Y_{i,T})+\sum_{i=1}^{N}H_i({X_i},U_{i,0})\nonumber\\&-	\sum_{i=1}^{N}H_i({X_i},U_{i,T})+\sum_{t=0}^{T-1}\eta_t^{2}\sum_{i=1}^{N}\Vert   Z_{i,t}\Vert^2_2+\sum_{t=0}^{T-1}\eta_t \tr{(  V_t-  X_t)^T   Z_t}+\sum_{t=0}^{T-1}\eta_t^{2}\sum_{i=1}^{N}\Vert \Phi_{i,t}\Vert^2_2 \nonumber\\ 
\label{eq:lmerror-4}  
 &\leq \sum_{i=1}^{N} H_i({X_i},Y_{i,0})+\sum_{i=1}^{N}H_i({X_i},U_{i,0})+\sum_{t=0}^{T-1}\eta_t^{2}\sum_{i=1}^{N}\Vert   Z_{i,t}\Vert^2_2+\nonumber\\ 
&\sum_{t=0}^{T-1}\eta_t \tr{(  V_t-  X_t)^T   Z_t}+\sum_{t=0}^{T-1}\eta_t^{2}\sum_{i=1}^{N}\Vert \Phi_{i,t}\Vert^2_2,
\end{align}
\nm{where the last inequality holds by $H_i(X_i,Y_i)\geq 0$ implied by Fenchel's inequality. Recall that for $X_i \in \mathcal X_i$, $\tr{X_i}=1$ and $-\log(n_i)\leq\tr{X_i \log X_i}\leq 0$ (\cite{carlen2010trace}). By choosing $Y_{i,0}=U_{i,0}=\mathbf I_{n_i}/{n_i}$ and from \eqref{eq:entropy}, \eqref{eq:conjugate} and \eqref{eq:fenchel},} we have
%\cite{carlen2010trace}
\begin{align*}
	&H_i(  {X_i},  Y_{i,0})=H_i(  {X_i},  U_{i,0})=\tr{  X_i \log   X_i-  X_i}+\log\tr{\exp(\mathbf I_{n_i}+\frac{\mathbf I_{n_i}}{n_i})}-\tr{\frac{X_i}{n_i}}\\& \leq 0-1+\log(n_i+1)-\frac{1}{n_i} \leq \log(n_i+1).
\end{align*}
Plugging the above inequality into \eqref{eq:lmerror-4} yields
%\begin{align*}  
%\label{eq:}
%\sum_{t=0}^{T-1}\eta_t \tr[(  X_t-  X)  F(  X)] &\leq 2 \log(n)+\sum_{t=0}^{T-1}\eta_t^2\Vert   Z_t \Vert ^2_2+\sum_{t=0}^{T-1}\eta_t \tr[(  V_t-  X_t)   Z_t]+\sum_{t=0}^{T-1}\eta_t^2\Vert \Phi_t\Vert^2_2.
%\end{align*}
\begin{align}  
\label{eq:bound-5}
&\sum_{t=0}^{T-1}\eta_t \tr{(X_t-X)^T F(X)}=\tr{\sum_{t=0}^{T-1}\eta_t (  X_t-  X)^T F(  X)} 
 \leq 2\sum_{i=1}^{N}\log(n_i+1)+\sum_{t=0}^{T-1}\eta_t^{2}\sum_{i=1}^{N}\Vert   Z_{i,t}\Vert^2_2+\nonumber\\ 
&\sum_{t=0}^{T-1}\eta_t \tr{(  V_t-  X_t)^T   Z_t}+\sum_{t=0}^{T-1}\eta_t^{2}\sum_{i=1}^{N}\Vert \Phi_{i,t}\Vert^2_2.
\end{align}
Let us define $\gamma_t\triangleq\frac{\eta_t}{\sum_{k=0}^{T-1} \eta_k}$, then, we have $\overline {  X}_T\triangleq\sum_{t=0}^{T-1}\gamma_t   X_t$ by Lemma \ref{lemma:average}. We divide both sides of \eqref{eq:bound-5} by $\sum_{t=0}^{T-1} {\eta_t}$. Then for all $X\in \mathcal X$,
\begin{align*}  
&\tr{\left(\sum_{t=0}^{T-1} \gamma_t  X_t-X\right)^T F(X)} =
\tr{\left(\overline{  X}_T-  X\right)^T  F(  X)}  \leq \frac{1}{\sum_{t=0}^{T-1}\eta_t} \Bigg(2\sum_{i=1}^{N}\log(n_i+1)\\
&+\sum_{t=0}^{T-1}\eta_t^{2}\sum_{i=1}^{N}\Vert   Z_{i,t}\Vert^2_2+\sum_{t=0}^{T-1}\eta_t \tr{(  V_t-  X_t)^T   Z_t}+\sum_{t=0}^{T-1}\eta_t^{2}\sum_{i=1}^{N}\Vert \Phi_{i,t}\Vert^2_2\Bigg).
\end{align*}
Note that the set $\mathcal X$ is a convex set. Since $\gamma_t>0$ and $\sum_{t=0}^{T-1}\gamma_t=1$,  $\overline{X}_T \in \mathcal X$. 
%is a convex combination of $  X_t \in \mathcal X$ and hence it belongs to the set $\mathcal X$ 
Now, we take the supremum over the set $\mathcal X$ with respect to $X$ and use the definition of the $G$ function given by Definition \ref{def:gap}. Note that the right-hand side of the preceding inequality is independent of $X$.
\begin{align*}  
G(\overline{  X}_T) &\leq \frac{1}{\sum_{t=0}^{T-1}\eta_t}\Bigg(2\sum_{i=1}^{N}\log(n_i+1)+\sum_{t=0}^{T-1}\eta_t^{2}\sum_{i=1}^{N}\Vert   Z_{i,t}\Vert^2_2+\sum_{t=0}^{T-1}\eta_t \tr{(  V_t-  X_t)^T   Z_t}\\
&+\sum_{t=0}^{T-1}\eta_t^{2}\sum_{i=1}^{N}\Vert \Phi_{i,t}\Vert^2_2 \Bigg).
\end{align*}
By taking expectations on both sides, we get 
\begin{align*}  
&\EX[G(\overline{  X}_T)] \leq \frac{1}{\sum_{t=0}^{T-1}\eta_t}\Bigg(2\sum_{i=1}^{N}\log(n_i+1)+\sum_{t=0}^{T-1}\eta_t^{2}\sum_{i=1}^{N}\EX[\Vert    Z_{i,t}|\mathcal F_t \Vert ^2_2]+\\&\sum_{t=0}^{T-1}\eta_t \EX[\tr{(  V_t-  X_t)^T   Z_t|\mathcal F_t}]+\sum_{t=0}^{T-1}\eta_t^{2}\sum_{i=1}^{N}\EX[\Vert \Phi_{i,t}|\mathcal F_t\Vert^2_2] \Bigg).
\end{align*}
%Recall that $\Vert   A\Vert^2_F \leq \sqrt{rm}\Vert   A\Vert^2_2$, where matrix $A \in \mathbb R^{m\times n}$ is of rank $r$.
By definition, both $X_t$ and $V_t$ are $\mathcal F_t$-measurable. Therefore, $V_t-X_t$ is $\mathcal F_t$-measurable. In addition, $Z_t$ is $\mathcal F_{t+1}$-measurable. Thus, by Assumption \ref{ass:boundphi}(c), we have $\EX[\tr{(  V_t-  X_t)^T   Z_t}|\mathcal F_t]=0$. Applying Assumption \ref{ass:boundphi}(b), we have
\begin{align*}  
\EX[G(\overline{  X}_T)] &\leq  \frac{2}{\sum\nolimits_{t=0}^{T-1}\eta_t}  \left(\sum_{i=1}^{N}\log(n_i+1)+\sum\nolimits_{t=0}^{T-1}\eta_t^{2} \sum_{i=1}^{N}C_i^2\right). 
%(\sum_{t=0}^{T-1}\eta_t) \EX[Gap(\overline{  X}_T)] &\leq& m+4M^2+\sum_{t=0}^{T-1}\eta_t^2 (\sqrt{rN} C^2+A^2).
\end{align*}
}
\end{proof}
\begin{comment}
In convergence analysis, we take advantage of the following Lemma. For more details and the proof, please refer to \cite{yousefian2014smoothing} (Lemma 10). 
\begin{lemma} For any scalar $a$ and integers $l$ and $T$ where $0\leq l\leq T-1$, we have:
\label{lm:bound-t}
\begin{itemize}
\item [(a)] $\ln\left(\frac{T+1}{l+1}\right)\leq\sum_{t=l}^{T-1}\frac{1}{t+1}\leq\frac{1}{l+1}+\ln\left({\frac{T}{l+1}}\right)$.
\item [(b)] $\displaystyle\frac{T^{a+1}-(l+1)^{a+1}}{a+1}\leq\sum_{t=l}^{T-1}(t+1)^a\leq (l+1)^a+\frac{(T+1)^{a+1}-(l+1)^{a+1}}{a+1}$ for any $a\neq-1$ . 
\end{itemize}
\end{lemma}
\end{comment}
Next, we present the convergence rate of the A-M-SMD scheme.
\begin{thm} 
Consider Problem \eqref{eq:VI} and let the sequence $\{\overline{  X}_t\}$ be generated by A-M-SMD algorithm. Suppose Assumption \ref{ass:boundphi} holds. Given a fixed $T>0$, let $\eta_t$ be a sequence given by 
\begin{comment} 
the following,
%\begin{itemize}
\item [(a)] Constant stepsize: for the case that number of iterations is fixed beforehand and $\eta_t=\eta$ for all $t \geq 0$, we have,
\end{comment}
\begin{align}
\label{eq:stepsize}
	&\eta_t=\frac{1}{\sum_{i=1}^{N}C_i}\sqrt{\frac{\sum_{i=1}^{N}\log(n_i+1)}{T}}, \quad \text{for all} \quad t\geq 0. 
	\end{align}
	Then, we have,
	\begin{align}
	\label{eq:rate}
	&\EX[G(\overline{  X}_T)] \leq 3\sum_{i=1}^{N}C_i\sqrt{\frac{\sum_{i=1}^{N}\log(n_i+1)}{T}}={\cal O}\left(\frac{1}{\sqrt {T}}\right).
\end{align}
\begin{comment}
\item [(b)] Diminishing stepsize: for the case that  number of iterations is not known in advance and stepsize decreases in every iteration, we have,
\begin{align*}
	&\eta_t=\frac{1}{C}\sqrt{\frac{\log(n)}{t+1}}, \quad \text{for all} \quad t\geq 0, \\ &\EX[G(\overline{  X}_T)] \leq \frac{ 2C\sqrt{\log(n)}(2+\ln(T))}{2(\sqrt T-1)}=\mathcal O(\frac{\ln(T)}{\sqrt T}).
\end{align*}
%\end{itemize}
\end{comment}
\end{thm}
  \begin{proof}
\nm{Consider relation \eqref{eq:bound-1}.} Assume that the number of iterations $T$ is fixed and $\eta_t=\eta$ for all $t \geq 0$, then, we get
\begin{align*}  
\EX[G(\overline{X}_T)] &\leq \frac{2 \left(\sum_{i=1}^{N}\log(n_i+1)+T\eta^{2} \sum_{i=1}^{N}C_i^2\right)}{T\eta}.
\end{align*}
Then, by minimizing the right-hand side of the above inequality over $\eta>0$, we obtain the constant stepsize \eqref{eq:stepsize}.
\begin{comment}
\begin{align*}
	\eta_t=\frac{1}{C}\sqrt{\frac{\log(n)}{T}}, \quad \text{for all} \quad t\geq 0,
\end{align*}
\end{comment}
By plugging \eqref{eq:stepsize} into \eqref{eq:bound-1}, we obtain \eqref{eq:rate}. 
\begin{comment}
\begin{align*}
\EX[G(\overline{  X}_T)] &\leq 3C\sqrt{\frac{\log(n)}{T}}.
\end{align*}
Now assume the number of iterations $T$ is not fixed beforehand. So, the following stepsizes can be substituted for the the constant stepsizes 
\begin{align*}
	\eta_t=\frac{1}{C}\sqrt{\frac{\log(n)}{t+1}}, \quad \text{for all} \quad t\geq 0.
\end{align*}
By plugging the above stepsize into \eqref{eq:bound-1}, we obtain
\begin{align*}
&\EX[G(\overline{  X}_T)] \leq \frac{2\log(n)(1+\sum_{t=0}^{T-1}(t+1)^{-1})}{\sum_{t=0}^{T-1}c^{-1}log(n)^{\frac{1}{2}}(t+1)^{\frac{-1}{2}}}=\\&\frac{2C\sqrt{\log(n)}(1+\sum_{t=0}^{T-1}(t+1)^{-1})}{\sum_{t=0}^{T-1}(t+1)^{\frac{-1}{2}}}.
\end{align*}
By recalling Lemma \ref{lm:bound-t}, for $T\geq 1$, we have, 
\begin{align*}
\EX[G(\overline{  X}_T)] \leq { 2C\sqrt{\log(n)}(1+1+\ln(T))}\times{\frac{1}{2(T^{\frac{1}{2}}-1)}}. 
\end{align*}
\end{comment}
% we obtain \ref{eq:rate-1}.
\end{proof}

%\section{ Application in Wireless Communication Networks }
 %The type of player can be defined as the channel gain to the target receiver. This channel gain is known by the corresponding transmitter, but not by other transmitter–receiver pairs. Other transmitters know only the probability distribution for the channel gain of other players. In this case, the payoff for one player is a function of the transmit power of all players and of the channel gains.
%to optimize the input covariance matrix of each transmitter in the presence of interference from other network users. In this way, one obtains a nonlinear (and possibly non-convex) optimization problem defined over a set of positive-definite matrices, representing the users’ power allocation policies
%a semidefinite optimization problem,
\section{Numerical Experiments}
\label{sec:num}
In this section, we examine the behavior of A-M-SMD method on throughput maximization problem in a multi-user MIMO wireless network as described in Section \ref{sec:motiv}. 
\subsection{Preliminary Analysis}
First, we need to show that the Nash equilibrium of game \eqref{eq:Rgame} is a solution of VI$(\mathcal X,F)$. In order to apply Lemma \ref{lemma:nash}, we need to prove that the throughput function $R_i( X_i, X_{-i})$ is a concave function.
%Matrix monotone functions are automatically matrix concave \cite{kwong1989some}.
In the next lemma, we show the sufficient conditions on two functions that guarantee the concavity of their composition. The proof can be found in Appendix.
\begin{lemma}
\label{lem:concaveconvex}
Suppose $h:\mathbb H_n \rightarrow \mathbb R$ and $g:\mathbb H_m\rightarrow \mathbb H_n$. Then, $f(X)=h(g(X))$ is concave if $h$ is concave and matrix monotone increasing \ks{(cf. Definition \ref{def:matrix}-e)} and $g$ is concave.
\end{lemma}
Now, we apply Lemma \ref{lem:concaveconvex} to show each player's objective function $R_i( X_i,X_{-i})$ is concave.
\begin{lemma}
\label{lem:concavity}
The user's transmission throughput function $R_i( X_i, X_{-i})$ is concave in $\mathcal X_i$.
\end{lemma}
\begin{proof}
Let us define $ W( X_i)=\mathbf I_{m_i}+\sum_{j\ne i} H_{ji} X_j  H_{ji}^\dagger+ H_{ii} X_i  H_{ii}^\dagger$. The function $ W( X_i)$ is a linear function in terms of $ X_i$.
%Then, we have 
%\begin{align*}
% W (\lambdaX '_i+(1-\lambda)X ''_i) &=	\mathbf I+\sum_{j\ne i} H _{ji}X _j  H _{ji}^\dagger+ H _{ii}(\lambdaX '_i+(1-\lambda)X ''_i)  H _{ii}^\dagger\\ 
%&=\lambda(\mathbf I+\sum_{j\ne i} H _{ji}X _j  H _{ji}^\dagger+ H _{ii}X '_i  H _{ii}^\dagger)+(1-\lambda)(\mathbf I+\sum_{j\ne i} H _{ji}X _j  H _{ji}^\dagger+ H _{ii}X ''_i  H _{ii}^\dagger)\\
%&=\lambda  W (X '_i)+(1-\lambda)  W (X ''_i).
%\end{align*}
%Then 
%\begin{align*}
% W (\lambdaX '_i+(1-\lambda)X ''_i) -\left(\lambda  W (X '_i)+(1-\lambda)  W (X ''_i)\right)=\mathbf 0.
%\end{align*}
Note that every linear transformation $T$ of the form $T: A\rightarrow \sum_{i}\alpha_i H_{ii}^\dagger A^T  H_{ii}$ preserves Hermitian matrices (\cite{de1967linear}), where $\alpha_i$ is a real scalar, and each $ H_{ii}$ is a certain matrix depending on $T$. Therefore, $ W( X_i)$ is Hermitian. Therefore, by definition \ref{def:matrix}(c), $ W( X_i)$ is both convex and concave in $ X_i$.
\\
We also know that $\log\det( X^{-1})$ is monotone decreasing (\cite{vandenberghe1998determinant}), meaning that if $ A\geq  B$, then $\log\det( A^{-1})\leq \log\det( B^{-1})$. Then, we have 
%\begin{align*}
$\log\det(\mathbf I_{m_i})=\log\det( A A^{-1})=\log\det( A)+\log\det( A^{-1})$,	
%\end{align*}
which results in $\log(1)=0=\log\det( A )+\log\det( A ^{-1})$. Therefore, $\log\det( A )\geq \log\det( B )$ which means $\log\det( X)$ is monotone increasing.

We also know that $g( X)=\log \det ( X)$ is a concave function (\cite{boyd2004convex}, page 74). From convexity of $W( X_i)$ and Lemma \ref{lem:concaveconvex}, we conclude that $R_i(X_i, X_{-i})=\log \det\big(\mathbf I_{m_i}$ $+\sum_{j} H_{ji} X_j  H_{ji}^\dagger\big)-\log \det( W_{-i})$ is a concave function in $ X_i$. 
\end{proof}
The following Corollary shows that sufficient equilibrium conditions are satisfied, therefore a Nash equilibrium of game \eqref{eq:Rgame} is a solution of variational inequality Problem \eqref{eq:VI}.
\begin{col}
The Nash equilibrium of \eqref{eq:Rgame} is a solution of VI$(\mathcal X,F)$ where $\mathcal X\triangleq\prod_i \mathcal X_i$ and $F(X)\triangleq-\text{diag}\left( H_{11}^\dagger W^{-1} H_{11},\cdots, H_{NN}^\dagger W^{-1} H_{NN}\right)$.
\end{col}
\begin{proof}
Please note that $\nabla_{X_i} R_i( X_i, X_{-i})=\nabla_{ X_i} \log \det\left(\mathbf I_{m_i}+\sum_{j} H_{ji} X_j  H_{ji}^\dagger\right)$ since the second term,
$\log \det( W_{-i})$, is independent of $ X_i$.
Let us define $ W=\left(\mathbf I_{m_i}+\sum_{j} H_{ji} X_j  H_{ji}^\dagger\right)$. Then, we have $\nabla_{ X_i} R_i(X_i,X_{-i})= H_{ii}^\dagger W^{-1} H_{ii}$ (\cite{mertikopoulos2016learning}). By Lemma \ref{lem:concavity},
each player's objective function $R_i( X_i, X_{-i})$ is concave in $ X_i$. We also know that $\mathcal X_i$ is a convex set. Therefore, using Lemma \ref{lemma:nash}, we have sufficient conditions to state the game \eqref{eq:Rgame} as a variational inequality problem VI(${\mathcal X}, F$).
\end{proof}
The next two lemmas show that the mapping $F$ defined by \eqref{Fdefinition} is monotone. The proof of the next lemma can be found in Appendix.
\begin{lemma}
\label{lem:gradientmonotone}
Suppose $f:\mathbb H_m\rightarrow \mathbb R$ is a differentiable function. If $f$ is a convex function, then $\nabla f$ is monotone, i.e., $\tr{\left(\nabla_{{ X}}^T f({ X})-\nabla_{{ Z}}^T f({ Z})\right)(X- { Z})} \geq 0$, for all $ X, Z\in \mathbb H_m$.
%\begin{align*}
%\label{eq:gradientmonotone}
	%  ,\quad \text{for all} \quad X , Z \in \mathcal X.
%\end{align*}
\end{lemma}
\begin{lemma}
\label{lemma:mon}
Consider the function $R_i$ given by \eqref{eq:game} and its gradient $\nabla_{X _i}^T \left(R_i(X _i,X _{-i})\right)=( H _{ii}^\dagger W^{-1} H_{ii})^T$. The mapping $F(X )\triangleq-\text{diag}\left( \nabla_{X _1} R_1(X _1,X _{-1}),\ldots,\nabla_{ X_N} R_N( X_N,X_{-N})\right)=-\text{diag}\big( H_{11}^\dagger W^{-1} H _{11},\cdots$ $, H_{NN}^\dagger  W^{-1} H_{NN}\big)$ is monotone.
\end{lemma}
\begin{proof}
The function $R_i(X _i,X _{-i})$ is concave in $X _i$ by Lemma \ref{lem:concavity} and as a result $-R_i(X _i,X _{-i})$ is a convex function. Therefore, $\nabla_{X _i}^T \left(-R_i(X _i,X _{-i})\right)=-( H _{ii}^\dagger W ^{-1} H _{ii})^T$ is monotone in $X _i$ by Lemma \ref{lem:gradientmonotone}. In other words,
\begin{align}
\label{eq:positivity}
	&-\tr{\left(\nabla_{{X _i}}^T R_i(X _i,X _{-i})-\nabla_{{ Z _i}}^T R_i( Z _i, Z _{-i})\right)(X _i- { Z _i})}= \notag\\
	&-\tr{\left( H _{ii}^\dagger W ^{-1}(X _i) H _{ii}- H _{ii}^\dagger W ^{-1}( Z _i) H _{ii}\right)^T\left(X _i- { Z _i}\right)}
	\geq 0, \quad \text{for all}\quad X _i, Z _i\in \mathcal X_i.
	\end{align}
Then, we have
\begin{align*}
&\tr{( F (X )- F ( Z ))(X - Z )}=\\ 
&\text{tr}(-\text{diag}( \nabla_{X _1} R_1(X _1,X _{-1})-\nabla_{ Z _1} R_1( Z _1, Z _{-1}),\ldots,\nabla_{X _N} R_N(X _N,X _{-N})-\nabla_{ Z _N} R_N( Z _N, Z _{-N})) \\
&\times\text{diag}({X _1}-  Z _1,\ldots,{X _N}-  Z _N ))\\
&=\text{tr}\big( -\text{diag}\left( H _{11}^\dagger W ^{-1}(X _1) H _{11}- H _{11}^\dagger W ^{-1}( Z _1) H _{11},\ldots, H _{NN}^\dagger W ^{-1}(X _N) H _{NN}- H _{NN}^\dagger W ^{-1}(Z_N)H _{NN} \right)\\\times&\text{diag}\left({X _1}-  Z _1,\ldots,{X _N}-  Z _N \right)\big)
\\
&=-\sum_{i=1}^N\sum_{u=1}^{m_i}\sum_{v=1}^{m_i}[ ( H _{ii}^\dagger W ^{-1}(X _i) H _{ii}- H _{ii}^\dagger W ^{-1}( Z _i) H _{ii})^T ]_{uv}[(X _i- Z _i)]_{uv}\geq 0,
\end{align*}
where the last relation follows by inequality \eqref{eq:positivity}.

\end{proof}
\ks{
\begin{rem}
Using Lemma \ref{lemma:mon}, the mapping $F$ defined by \eqref{Fdefinition} is monotone. Therefore, applying Lemma \ref{lemma:convergence}, the sequence $\overline{X }_t$ generated by A-M-SMD algorithm converges to the weak solution of variational inequality \eqref{eq:VI}.
\end{rem}
}
\subsection{Problem Parameters and Termination Criteria}
We consider a MIMO multi-cell cellular network composed of seven hexagonal cells (each with a radius of $1$ km) as Figure \ref{fig:cell}. We assume there is one MIMO link (user) in each cell which corresponds to the transmission from a transmitter (T) to a receiver (R). Following \cite{scutari2009mimo} we generate the channel matrices with a Rayleigh distribution, in other words, each element is generated as circularly symmetric Gaussian random variable with variance equal to the inverse of the square distance between the transmitters and receivers. In this regard, we normalize the distance between transmitters and receivers at first. The network can be considered as a 7-users game where each link (user) is a MIMO channel. 
\begin{figure}[H]
\centering
  \includegraphics[scale=0.4]{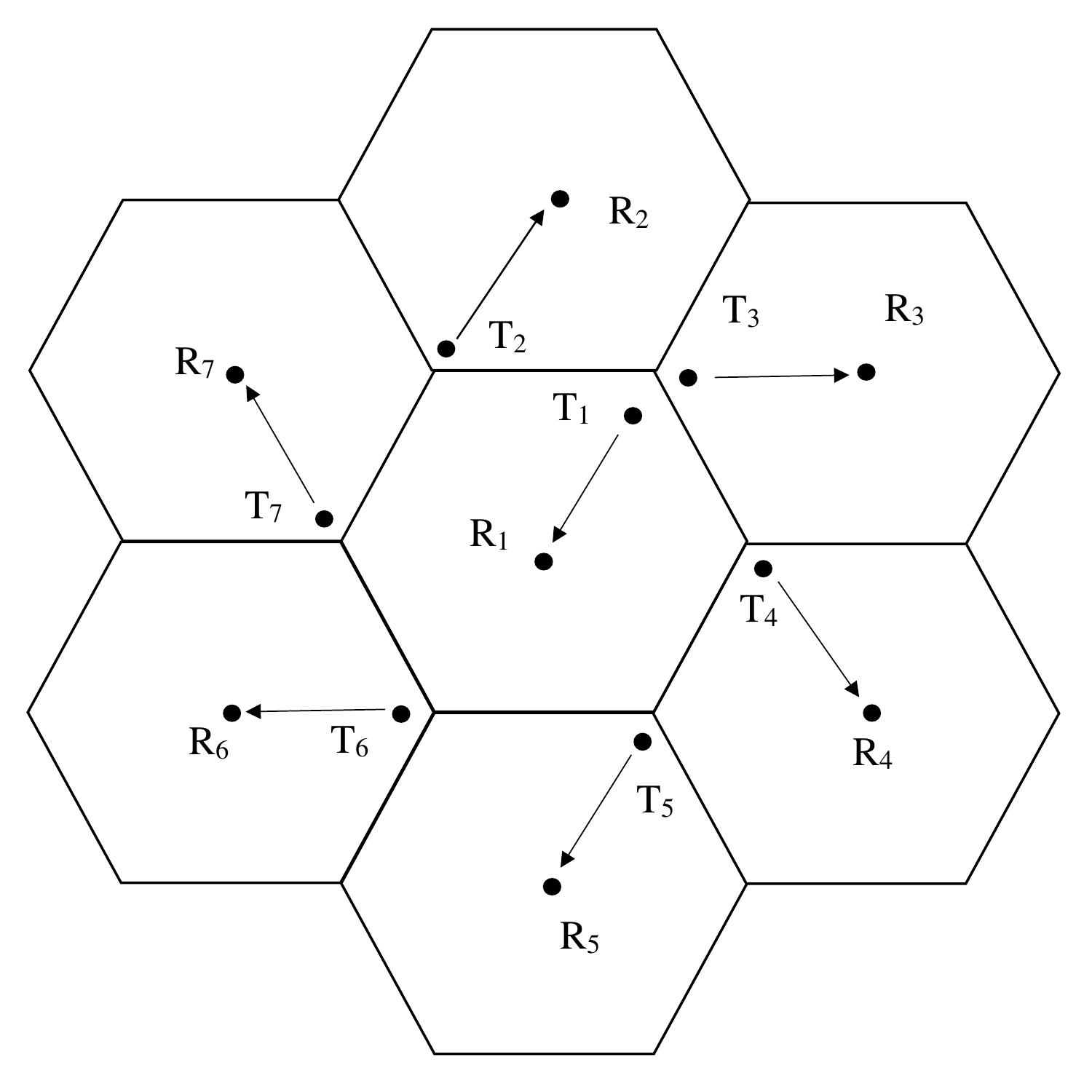}
	\caption{Multicell cellular system}
	\label{fig:cell}
\end{figure}
Distances between different receivers and transmitters are shown in Table \ref{table:distance}. It should be noted that the channel matrix between any pair of transmitter $i$ and receiver $j$ is a matrix with dimension of $m_j \times n_i$. In the experiments, we assume $m_j=m$ for all $j\in\{1,\ldots,7\}$ $n_i=n$ for all $i\in\{1,\ldots,7\}$.
%As an example, the channel matrix between transmitter 4 and receiver 5, where $n=m=4$ is represented in Table \ref{table:channel}.  
As mentioned before, $p_{max}$ is the maximum average transmitted power in units of energy per transmission. In the experiments, the transmitters have a maximum power of $1$ decibels of the measured power referenced to one milliwatt (dBm). 
\begin{table}[ht]
\footnotesize
\caption{Distance matrix (in terms of kilometer)}
\label{table:distance}
\centering
\begin{tabular}{|c|c c c c c c c|}
\hline
\backslashbox{\scriptsize{Receiver}}{\scriptsize{Transmitter}} & R1 & R2 & R3 & R4 & R5 & R6 & R7  \\ \hline
T1 & 0.8944 & 1.0143 & 1.0568 & 1.1020 & 1.0143 & 1.0568 & 1.1020  \\
T2 & 1.0143 & 0.8944 & 1.0568 & 2.1079 & 2.6940 & 2.6677 & 1.9964	\\
T3 & 1.1020 & 1.9011 & 0.8944 & 1.0143 & 2.1079 & 2.7265 & 2.7203	 \\ 
T4 & 1.9964 & 2.6159 & 1.9493 & 0.8944 & 1.1020 & 2.1056 & 2.7620  \\
T5 & 2.5635 & 2.6940 & 2.6677 & 1.9964 & 0.8944 & 1.0568 & 2.1079 \\
T6 & 2.5270 & 2.1079 & 2.7265 & 2.7203 & 1.9011 & 0.8944 & 1.0143 \\
T7 & 1.9011 & 1.1020 & 2.1056 & 2.7620 & 2.6159 & 1.9493 & 0.8944 \\
\hline
\end{tabular}
\end{table}
\begin{comment}
\begin{table}[ht]\footnotesize
\TABLE
{Channel matrix between transmitter 4 and receiver 5 (in terms of decibels)\label{table:channel}}
{\begin{tabular}{|c|c c c c|}
\hline
\up\down\backslashbox{\scriptsize{Transmitter}}{\scriptsize{Receiver}} & Antenna1 & Antenna2 & Antenna3 & Antenna4  \\ \hline
\up Antenna1 & -0.5468-0.7146i & -1.3902+2.2441i & 0.6536-2.1707i & 0.8415+0.1724i  \\
Antenna2 & -0.1398-0.7192i & -0.1470+0.8837i & 0.0949-1.6759i & -1.2262-0.2523i	\\
Antenna3 & 1.3908+2.3405i  & -0.1730+1.2379i & 1.0045+0.2313i & 1.7228-0.3344i	 \\ 
\down Antenna4 & 2.4070-0.9706i  & 1.1076-1.0744i  & 2.9480-2.0018i & 0.2192-1.6474i  \\
\hline
\end{tabular}}
{}
\end{table}
\end{comment}
We investigate the robustness of A-M-SMD algorithm under imperfect feedback. To simulate imperfections, the elements of $Z_{i,t}$ are generated as zero-mean circularly symmetric complex Gaussian random variables with variance equal to $\sigma$.
 %we generate a zero-mean circularly symmetric complex Gaussian noise vector $ Z _{t,i}$ with covariance matrix $\sigma \mathbf I_{m_i}$. 
To demonstrate the performance of the methods in this section, we employ the following gap function $Gap(X )$ which is equal to zero for a strong solution.
\begin{defn} [A gap function] \label{def:gap2} Define the following function $Gap: \mathcal X \rightarrow \mathbb R$
\begin{equation}
\label{gap3}
	Gap( X )= \underset { Z  \in \mathcal X} {\sup}\ \tr{( X - Z )^T F(X )}, \quad \text{for all}~  X  \in  \mathcal X.
\end{equation}
\end{defn}
In the following lemma, we provide some properties of the Gap function. The proof can be find in Appendix.
\begin{lemma} [Properties of the Gap function]
\label{lm:gap properties}
The function $Gap( X )$ given by Definition \ref{def:gap2} is a well-defined gap function, in other words, $(i)$ $Gap( X )$ is nonnegative for all $X  \in \mathcal X$; and $(ii)$ $X ^*$ is a strong solution to Problem \eqref{eq:VI} iff $Gap( {X^*})=0$.
\end{lemma}
The algorithms are run for a fixed number of iterations $T$. We plot the gap function for different number of transmitter antennas ($n$) and receiver antennas ($m$). We also plot the gap function for different values of $\sigma$ including $0.5, 1, 5$. We use MATLAB to run the algorithms and CVX software to solve the optimization Problem \eqref{gap3}. Computational experiments are performed using the same PC running on an Intel Core i5-520M 2.4 GHz processor with 4 GB RAM.
\subsection{Averaging and Non-averaging Matrix Stochastic Mirror Descent methods}
First, we look into the first 100 iterations in one sample path to see the impact of averaging on the initial performance of matrix stochastic mirror descent (M-SMD) algorithm. Figure \ref{fig:twoplots} compares the performance of averaging stochastic mirror descent (A-M-SMD) algorithm with M-SMD in the first 100 iterations. The pair of $(n,m)$ denotes the number of transmitter and receiver antennas. The vertical axis displays the logarithm of gap function \eqref{gap3} while the horizontal axis displays the iteration number. In these plots, the blue (dash-dot) and black (solid) curves correspond to the M-SMD and A-M-SMD algorithms, respectively. We observe in Figure \ref{fig:twoplots} that A-M-SMD algorithm outperforms the M-SMD in most of the experiments. Importantly, A-M-SMD is significantly more robust with respect to: (i) the imperfections and uncertainty ($\sigma$); and (ii) problem size (the number of transmitter and receiver antennas). 
Then, we run both A-M-SMD algorithm and M-SMD for $T=4000$ iterations and plot their performance in Figure \ref{fig:fourplots}. In this figure, the vertical axis displays the logarithm of expected gap function \eqref{gap3} while the horizontal axis displays the iteration number. The expectation is taken over $ Z _t$, we repeat the algorithm for $10$ sample paths and obtain the average of the gap function. For comparison purposes, we also plot the performance of M-SMD and A-M-SMD algorithms starting from a different initial point with a better gap function value. This point is obtained by running the algorithm for 400 iterations and saving the best solution $X $ to \eqref{gap3} and its corresponding $ Y$.  In these plots, the blue (dash-dot) and magenta (solid diamond) curves correspond to the M-SMD with the initial solution $X _0=X _0^1=\mathbf I_n/n$ and $X _0=X _0^2=X _{400}$ respectively, and the black (solid) and red (dash-dot triangle) curves display the A-M-SMD algorithm with the initial solution $X _0=X _0^1=\mathbf I_n/n$ and $X _0=X _0^2=X _{400}$ respectively. As it can be seen in Figure \ref{fig:fourplots}, A-M-SMD outperforms M-SMD in all experiments. In particular, A-M-SMD is significantly more robust with respect to (i) the imperfections ($\sigma$); and (ii) problem size. It is also observed that A-M-SMD converges to the strong solution with rate of convergence of $\mathcal O(1/T)$ while M-SMD does not converge for larger values of $\sigma$. Moreover, from Figure \ref{fig:fourplots}, it is evident that the A-M-SMD has better performance compared to M-SMD irrespective to the initial solution. 
\begin{table}[h]
\setlength{\tabcolsep}{3pt}
\centering
 \begin{tabular}{c | c  c  c}
$(n,m)$ & $\sigma=0.5$ & $\sigma=1$ & $\sigma=5$ \\ \hline\\
(2,4)
&
\begin{minipage}{.29\textwidth}
\includegraphics[scale=.3, angle=0]{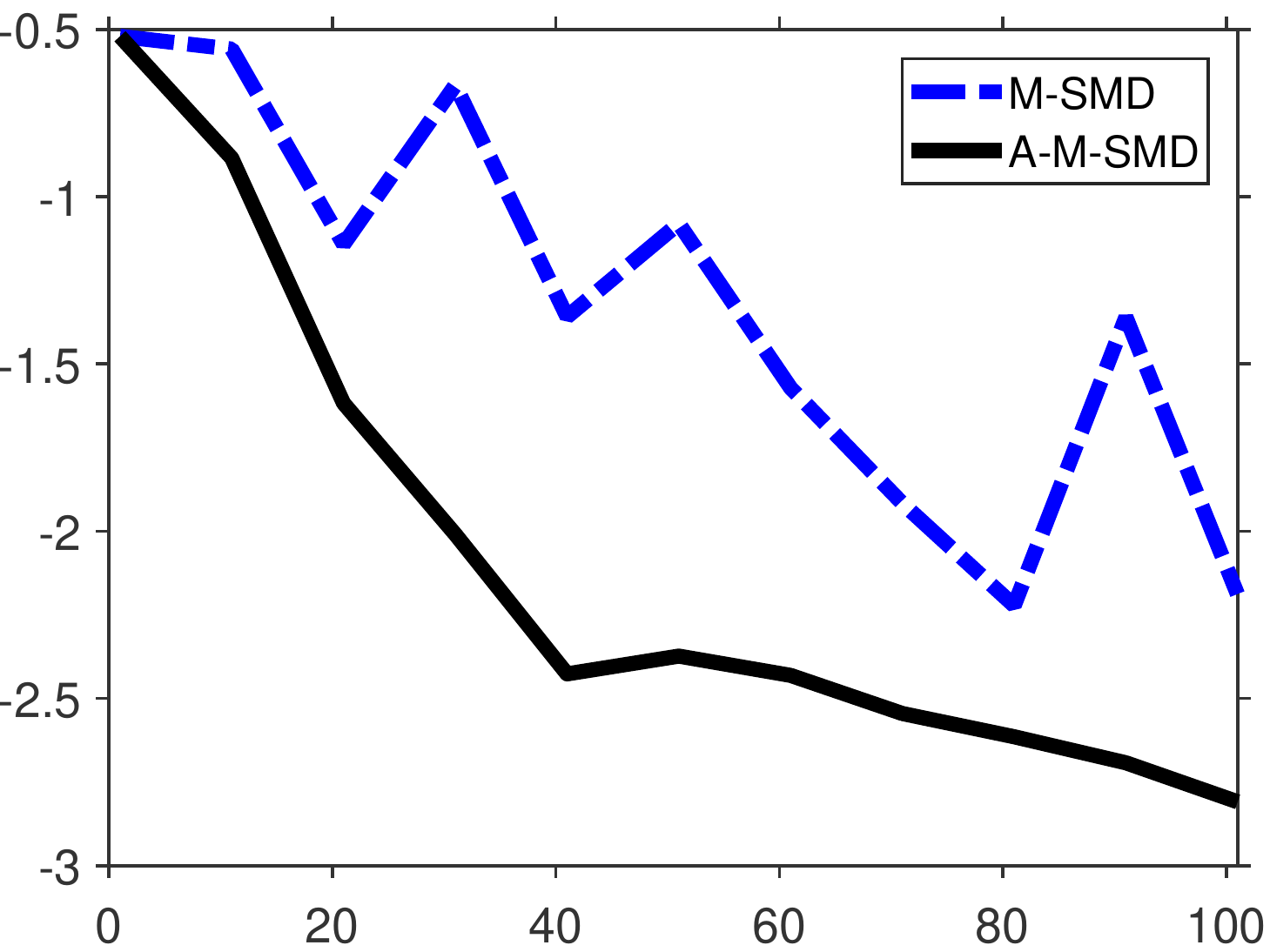}
%\vspace{.25in}
\end{minipage}
&
\begin{minipage}{.29\textwidth}
\includegraphics[scale=.3, angle=0]{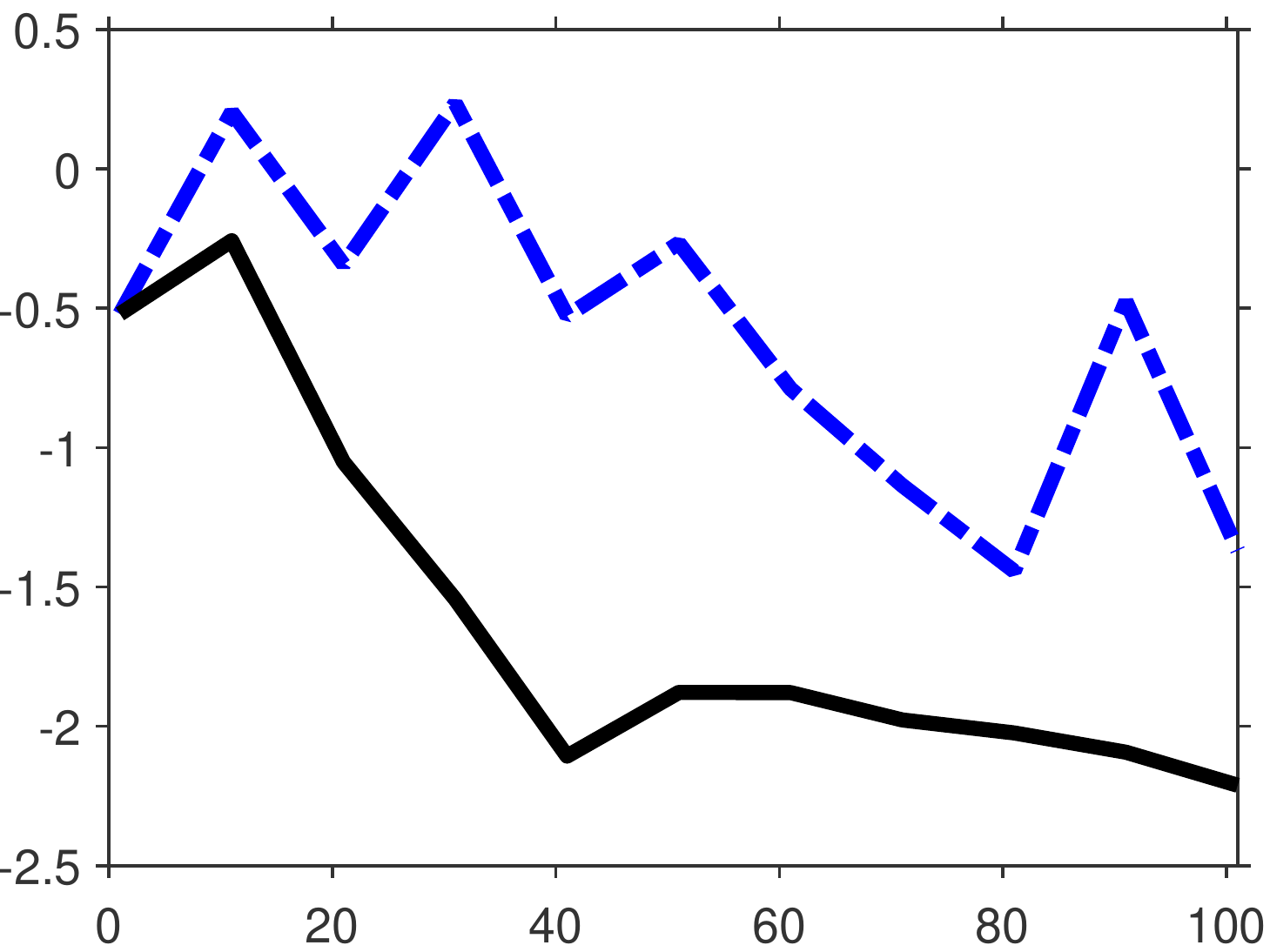}
%\vspace{.25in}
\end{minipage}
	&
\begin{minipage}{.29\textwidth}
\includegraphics[scale=.3, angle=0]{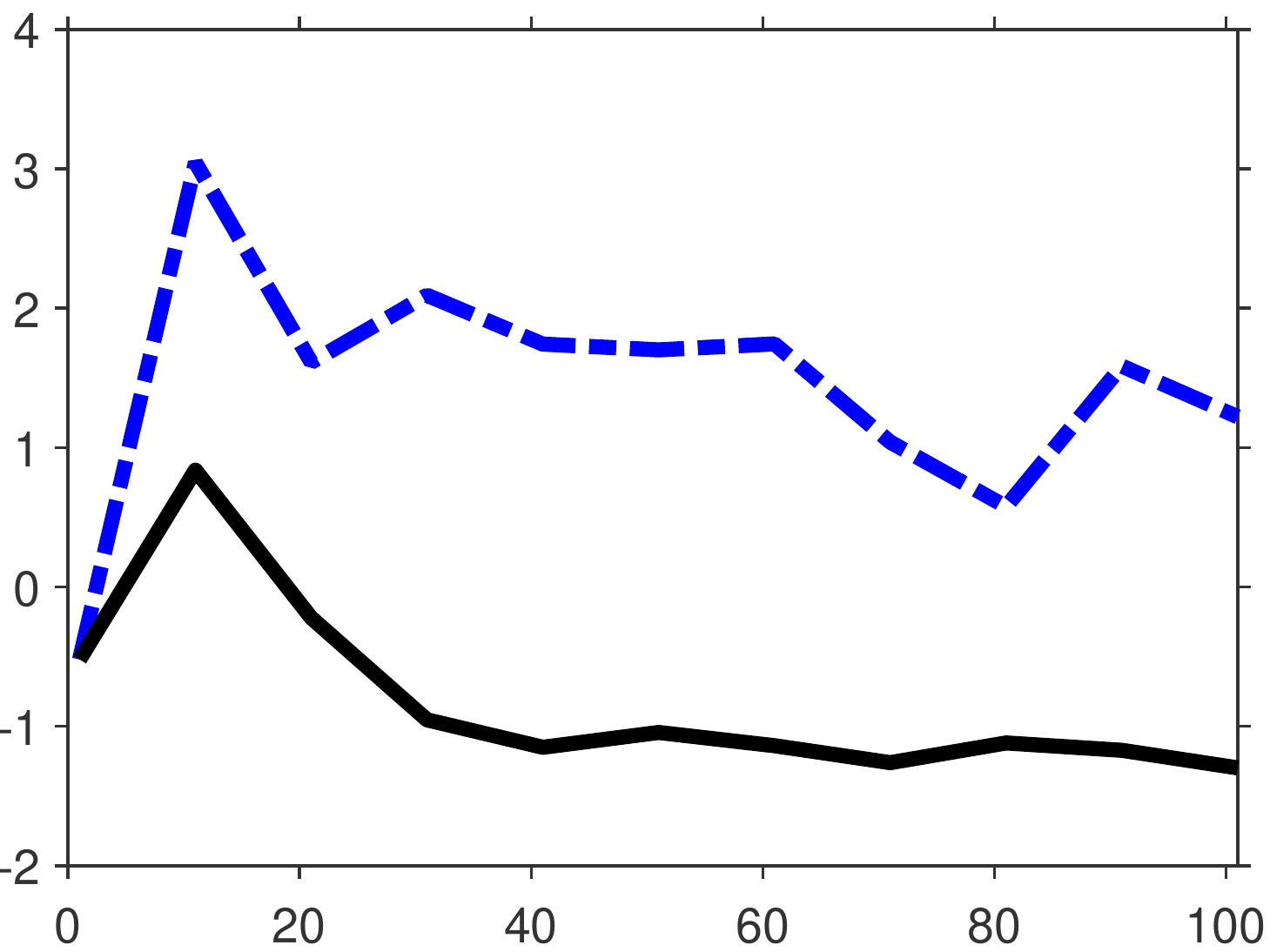}
%\vspace{.25in}
\end{minipage}
\\
(4,2)
&
\begin{minipage}{.29\textwidth}
\includegraphics[scale=.3, angle=0]{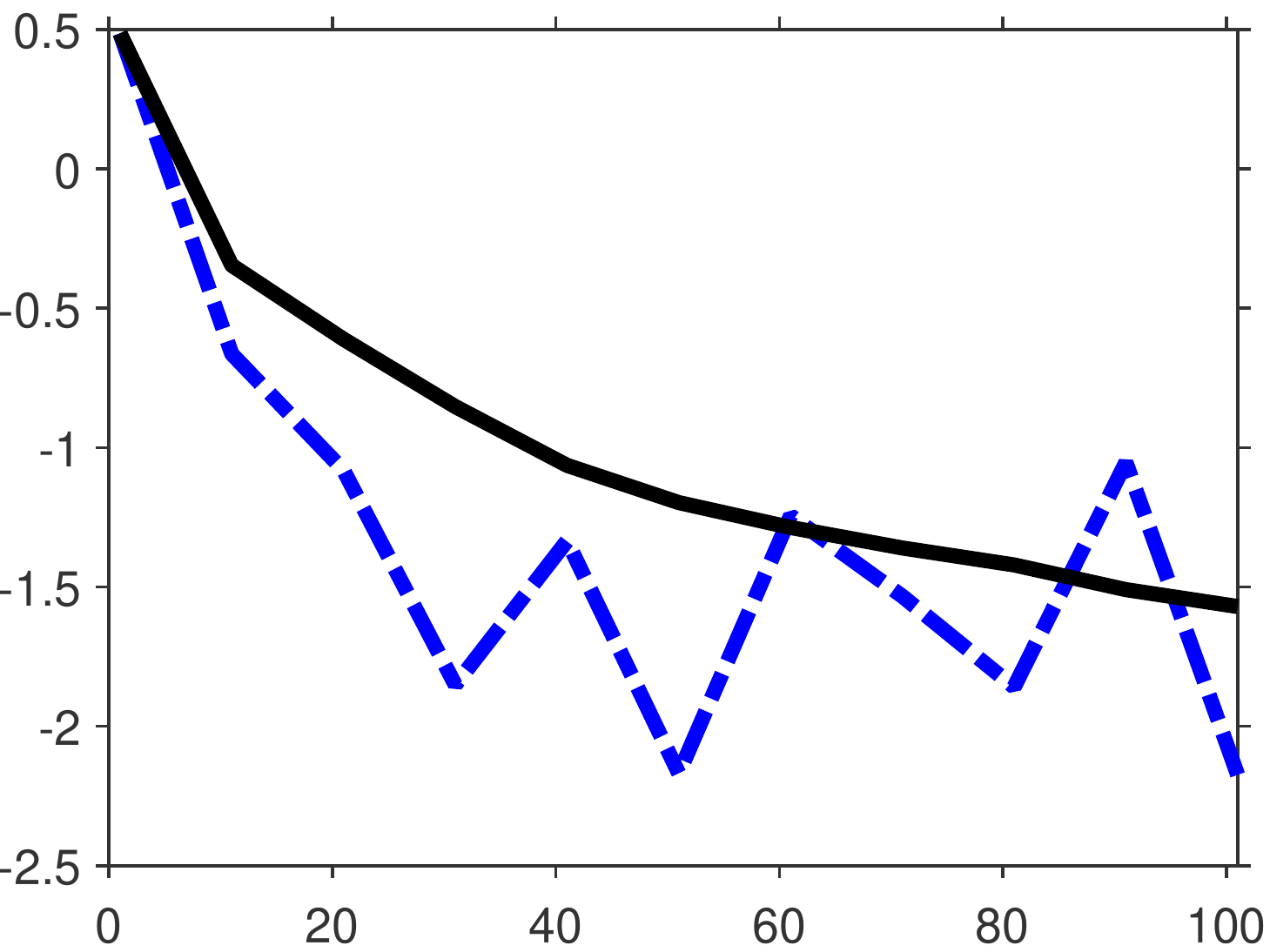}
\end{minipage}
&
\begin{minipage}{.29\textwidth}
\includegraphics[scale=.3, angle=0]{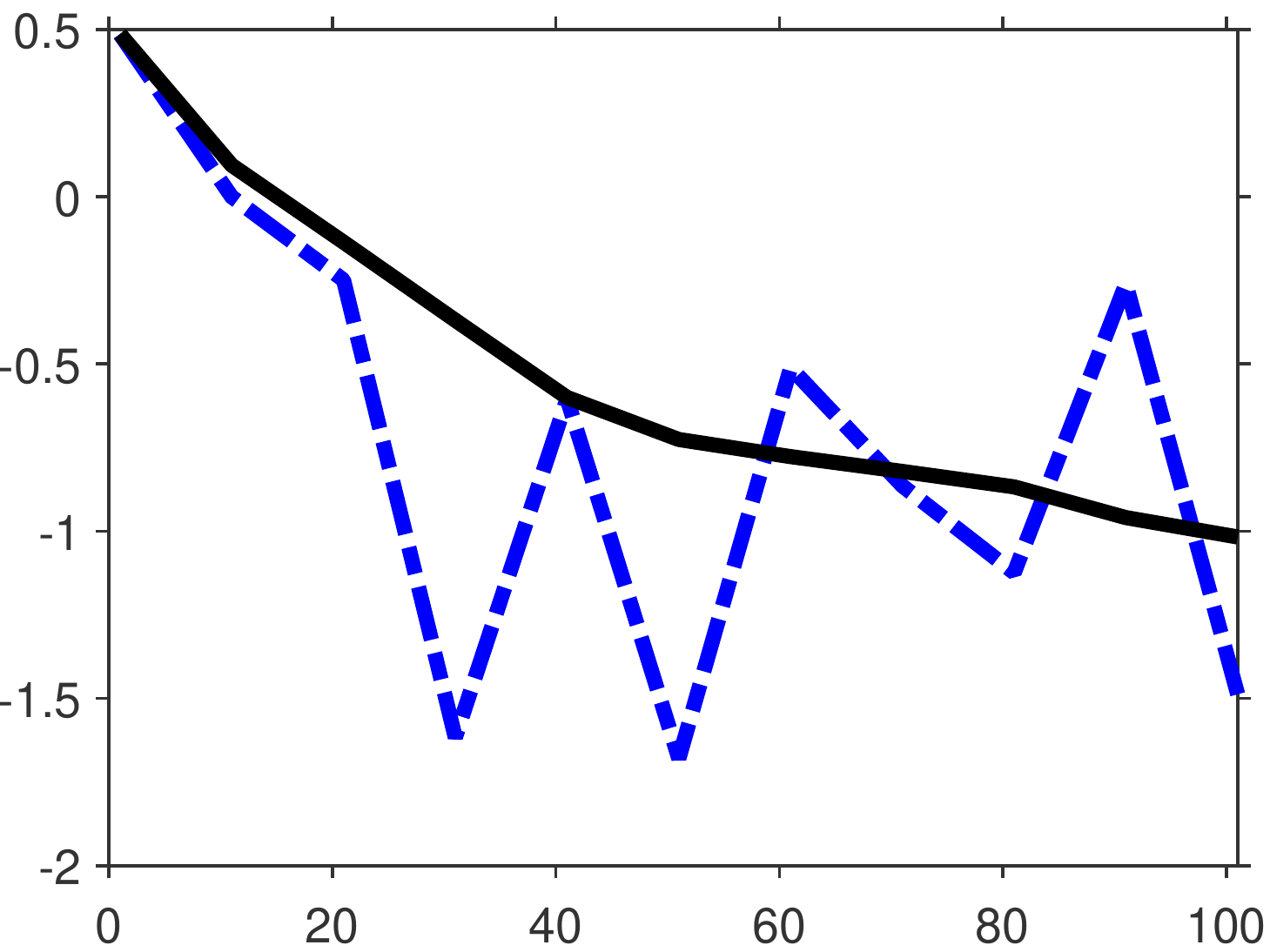}
\end{minipage}
&
\begin{minipage}{.29\textwidth}
\includegraphics[scale=.3, angle=0]{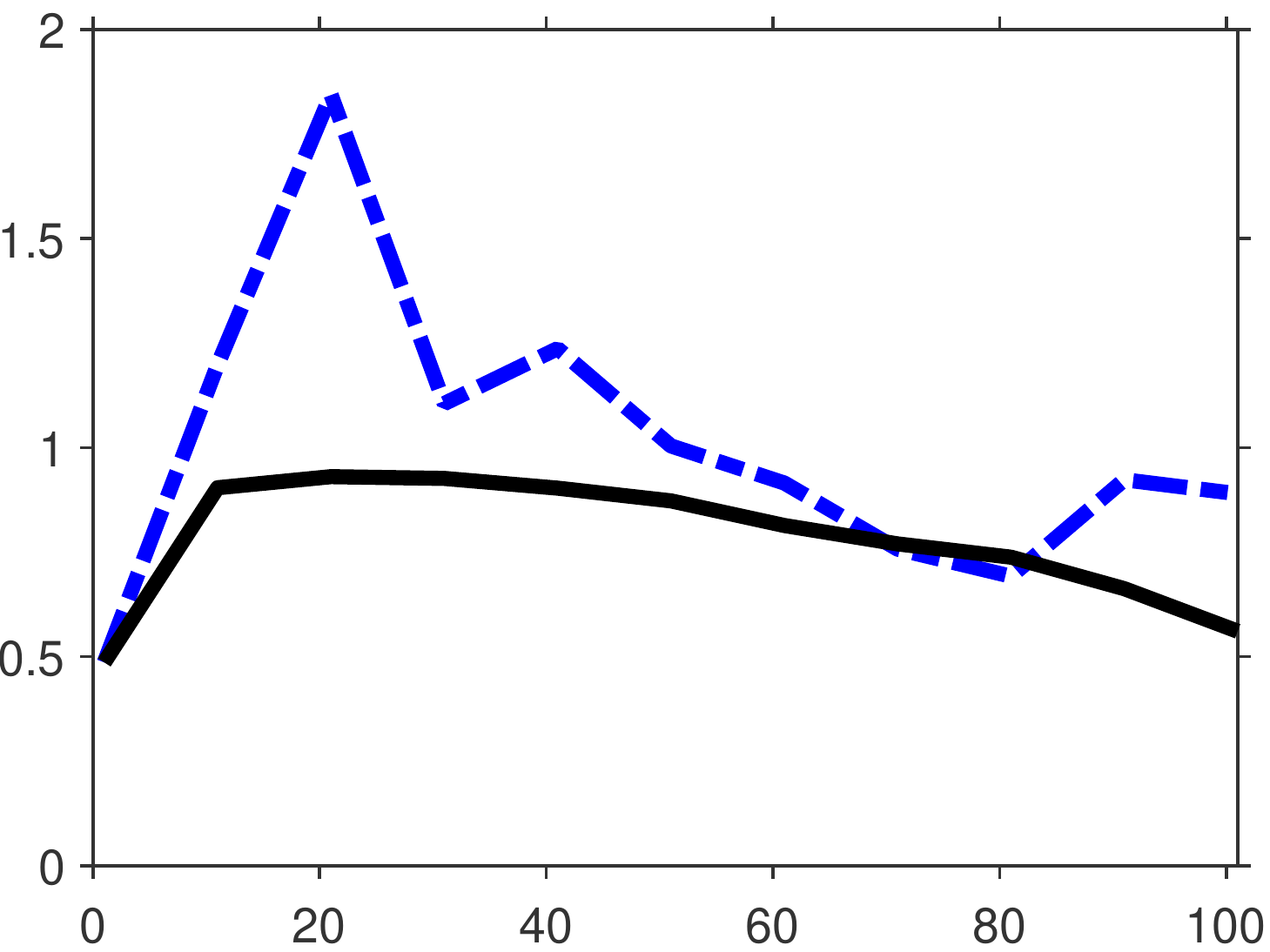}
\end{minipage}
\\
%\\
(4,4)
&
\begin{minipage}{.29\textwidth}
\includegraphics[scale=.3, angle=0]{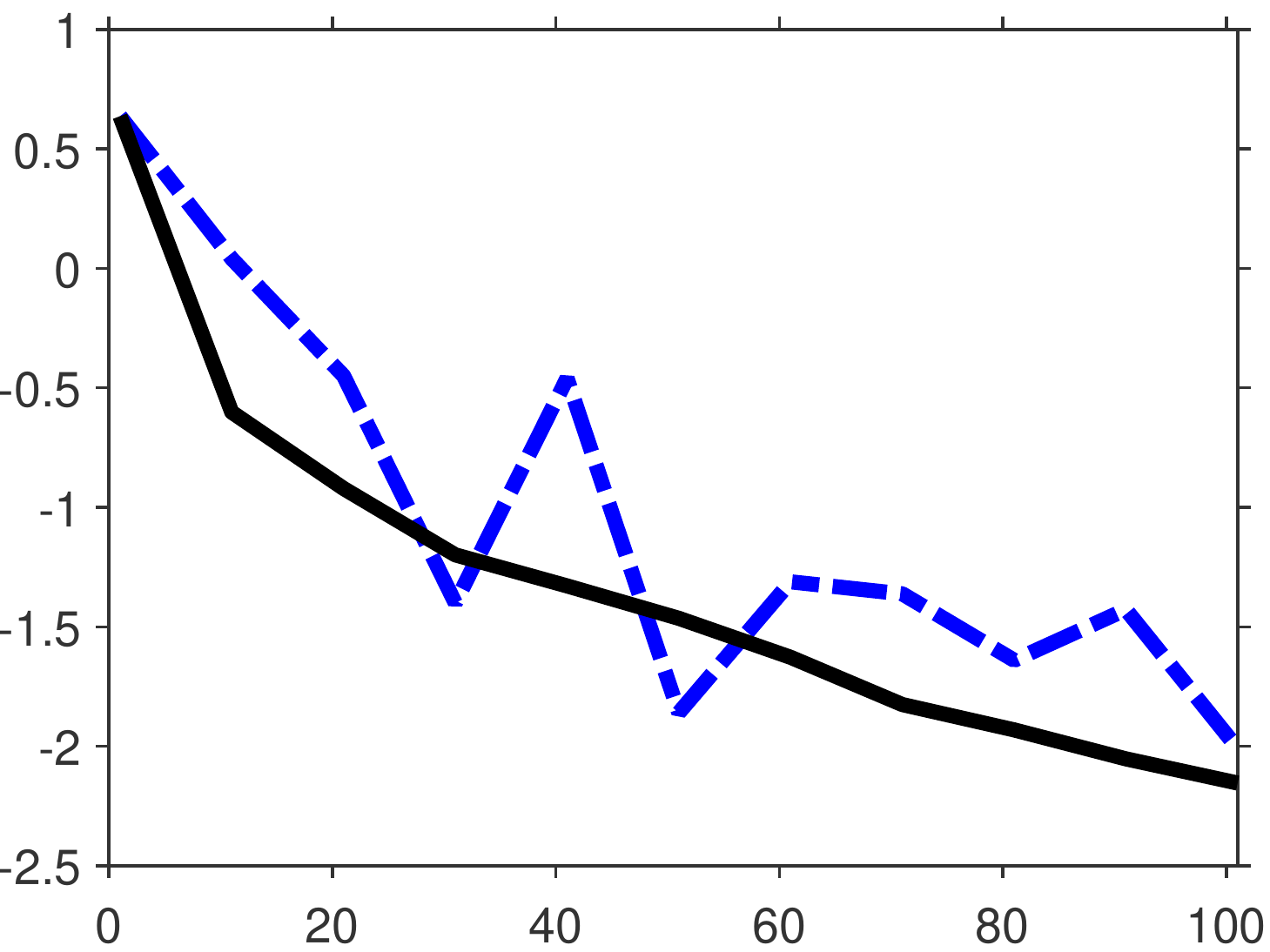}
\end{minipage}
&
\begin{minipage}{.29\textwidth}
\includegraphics[scale=.3, angle=0]{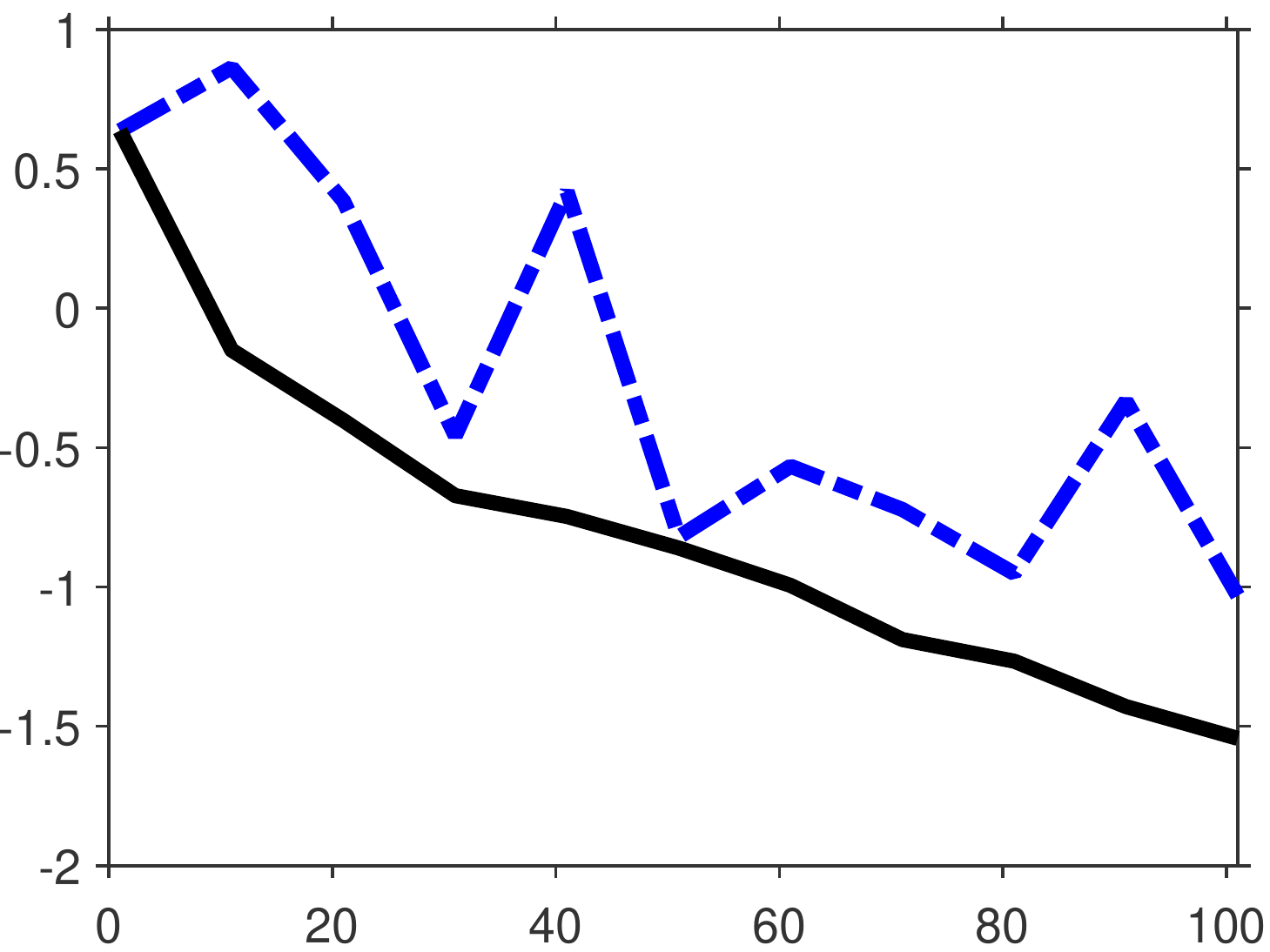}
\end{minipage}
&
\begin{minipage}{.29\textwidth}
\includegraphics[scale=.3, angle=0]{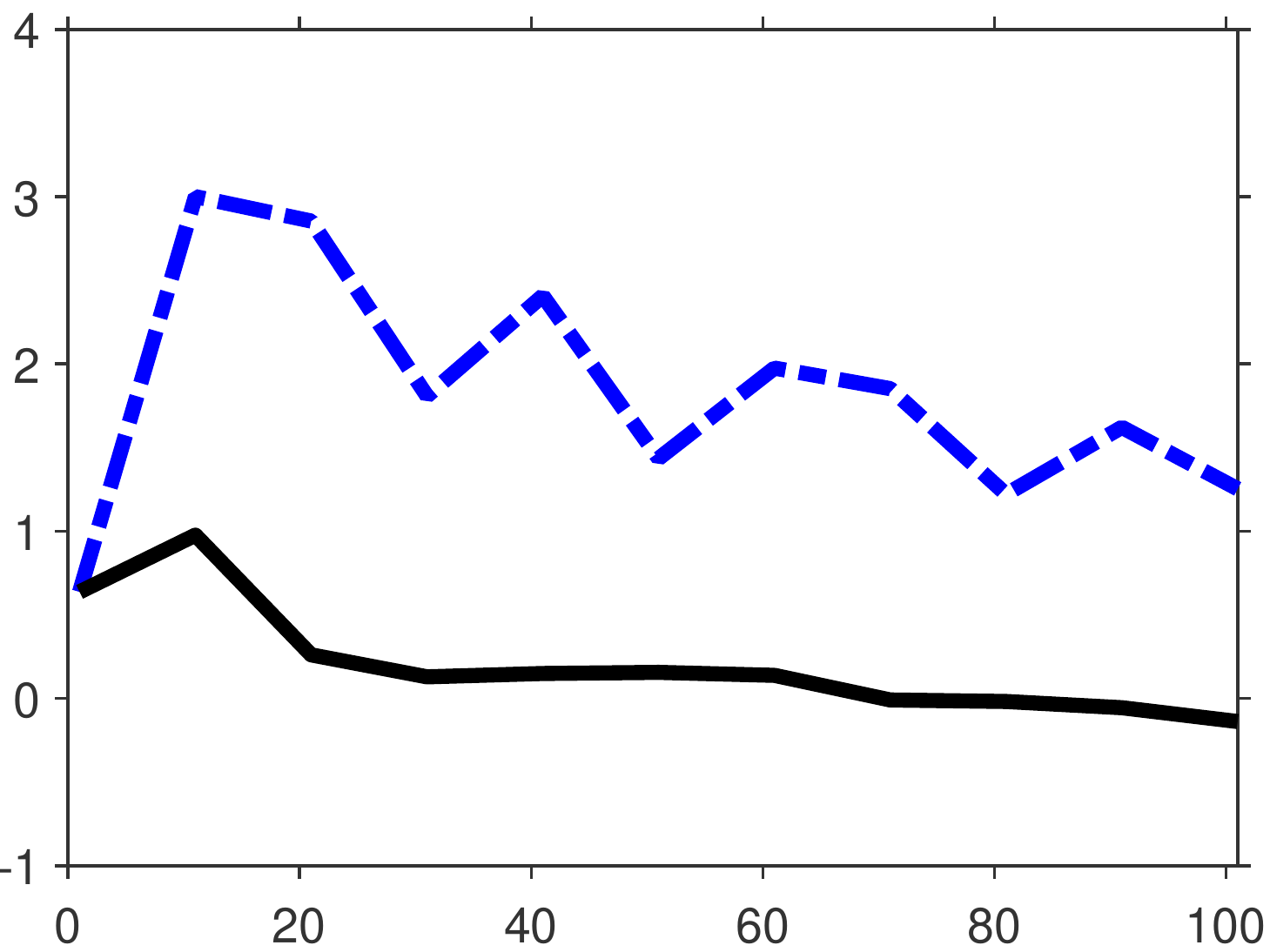}
\end{minipage}
\end{tabular}
\vspace{.1in}
\captionof{figure}{Comparison of M-SMD and A-M-SMD w.r.t. problem size ($n,m$) and uncertainty ($\sigma$) for 100 iterations}
\label{fig:twoplots}
\end{table}
\begin{table}[h]
\setlength{\tabcolsep}{3pt}
\centering
 \begin{tabular}{c|c c c}
$(n,m)$ & $\sigma=0.5$ & $\sigma=1$ & $\sigma=5$ \\ \hline\\
(2,4)
&
\begin{minipage}{.29\textwidth}
\includegraphics[scale=.3, angle=0]{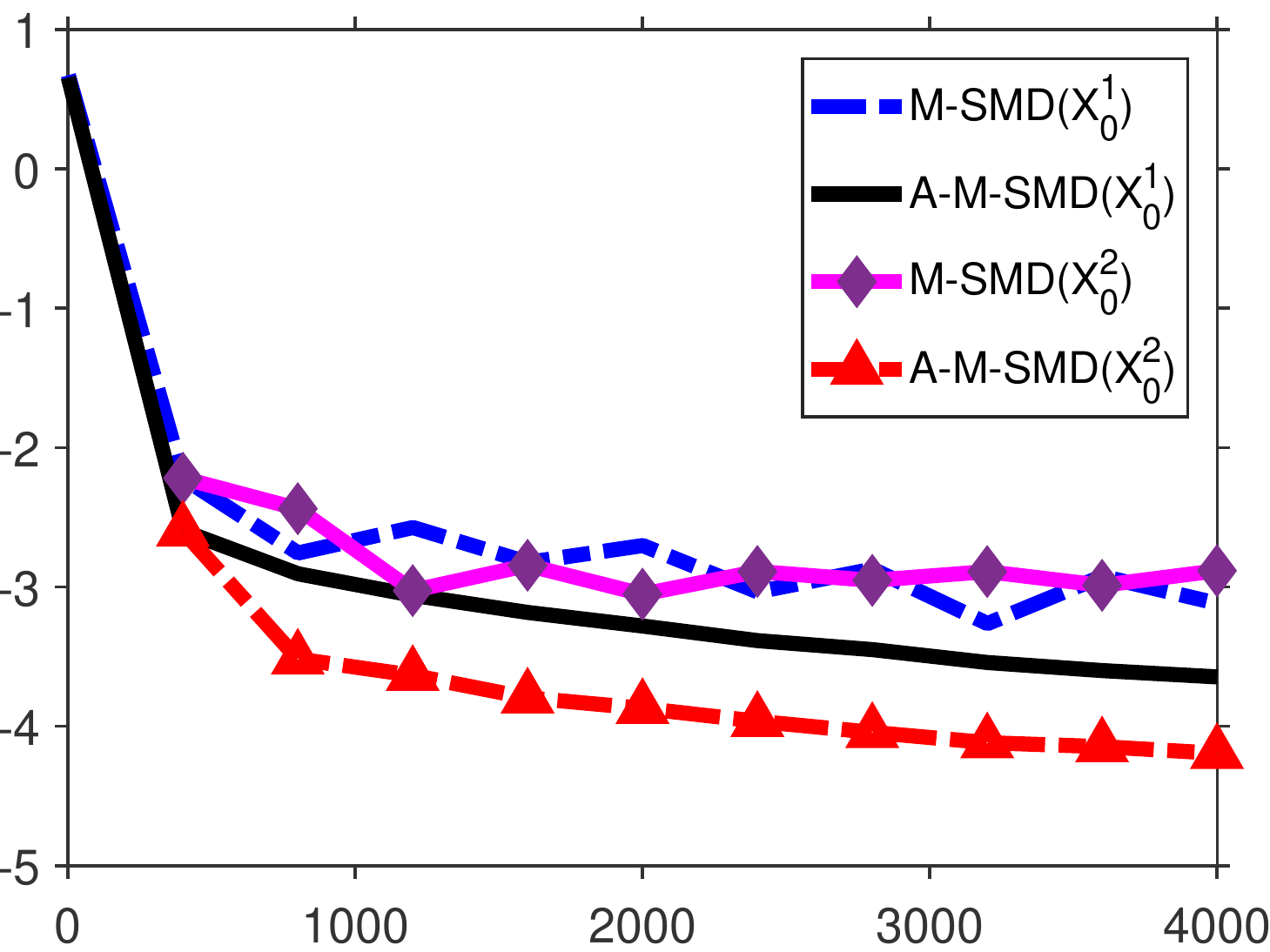}
%\vspace{.25in}
\end{minipage}
&
\begin{minipage}{.29\textwidth}
\includegraphics[scale=.3, angle=0]{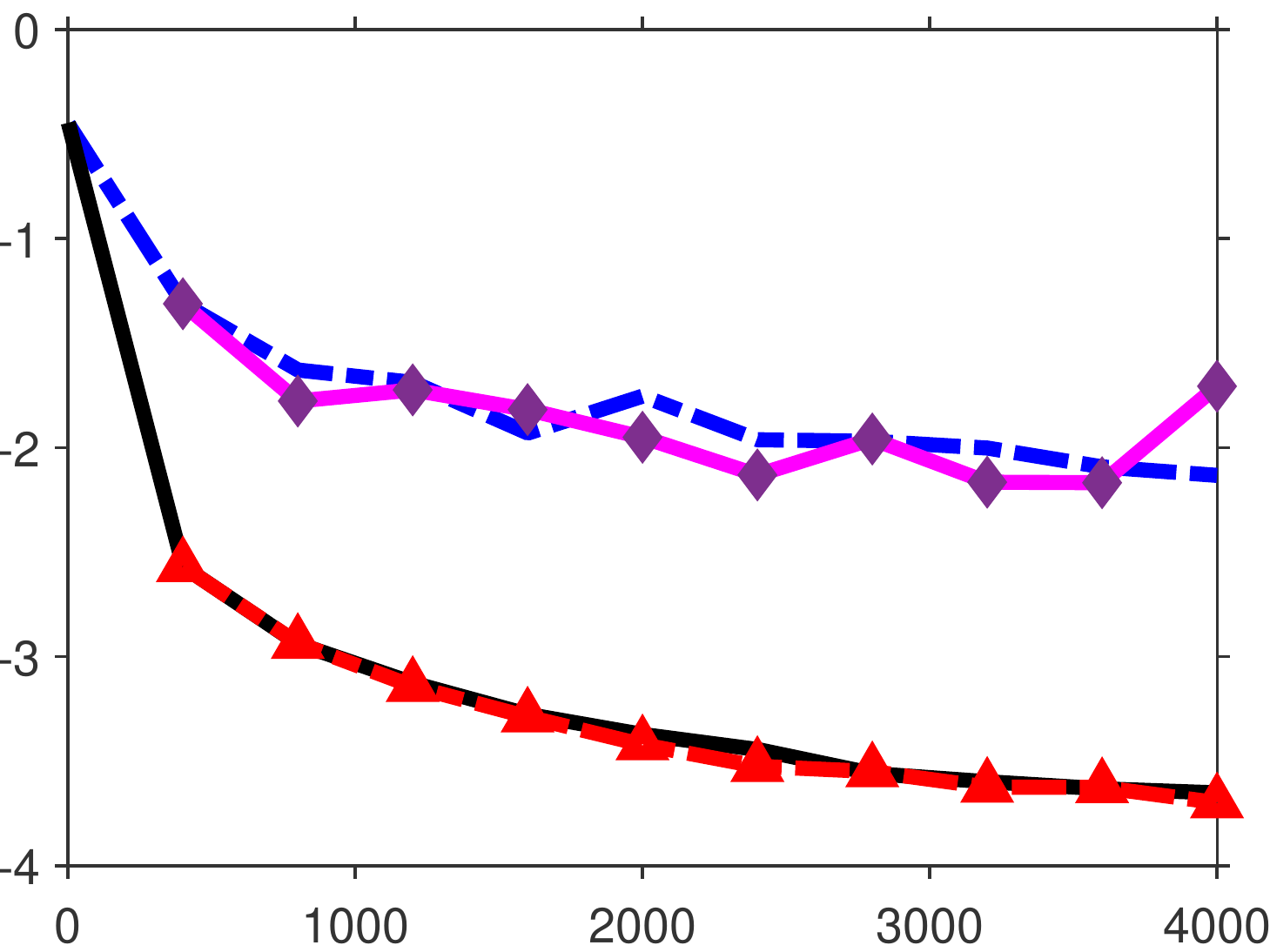}
%\vspace{.25in}
\end{minipage}
	&
\begin{minipage}{.29\textwidth}
\includegraphics[scale=.3, angle=0]{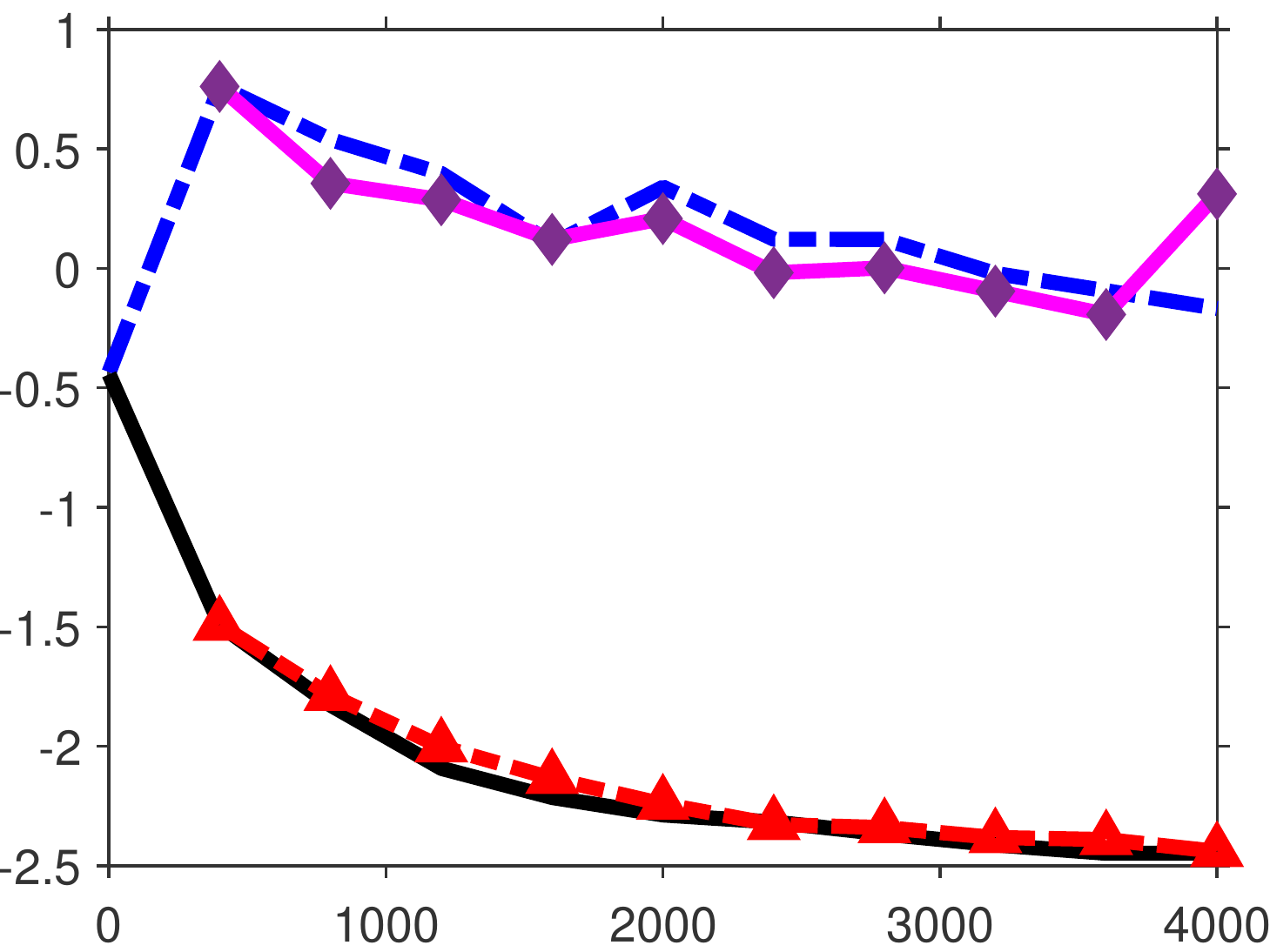}
%\vspace{.25in}
\end{minipage}
\\
(4,2)
&
\begin{minipage}{.29\textwidth}
\includegraphics[scale=.3, angle=0]{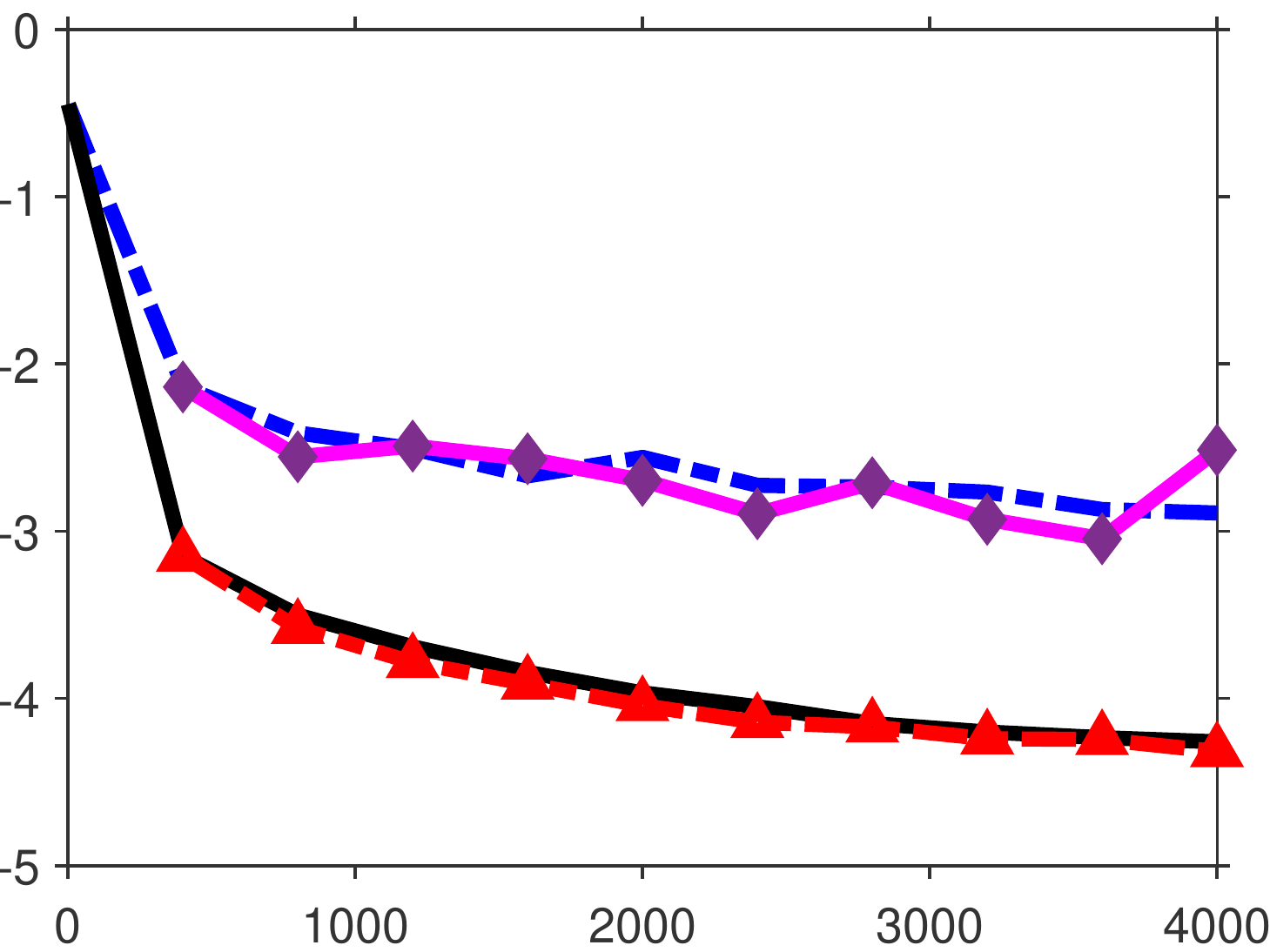}
\end{minipage}
&
\begin{minipage}{.29\textwidth}
\includegraphics[scale=.3, angle=0]{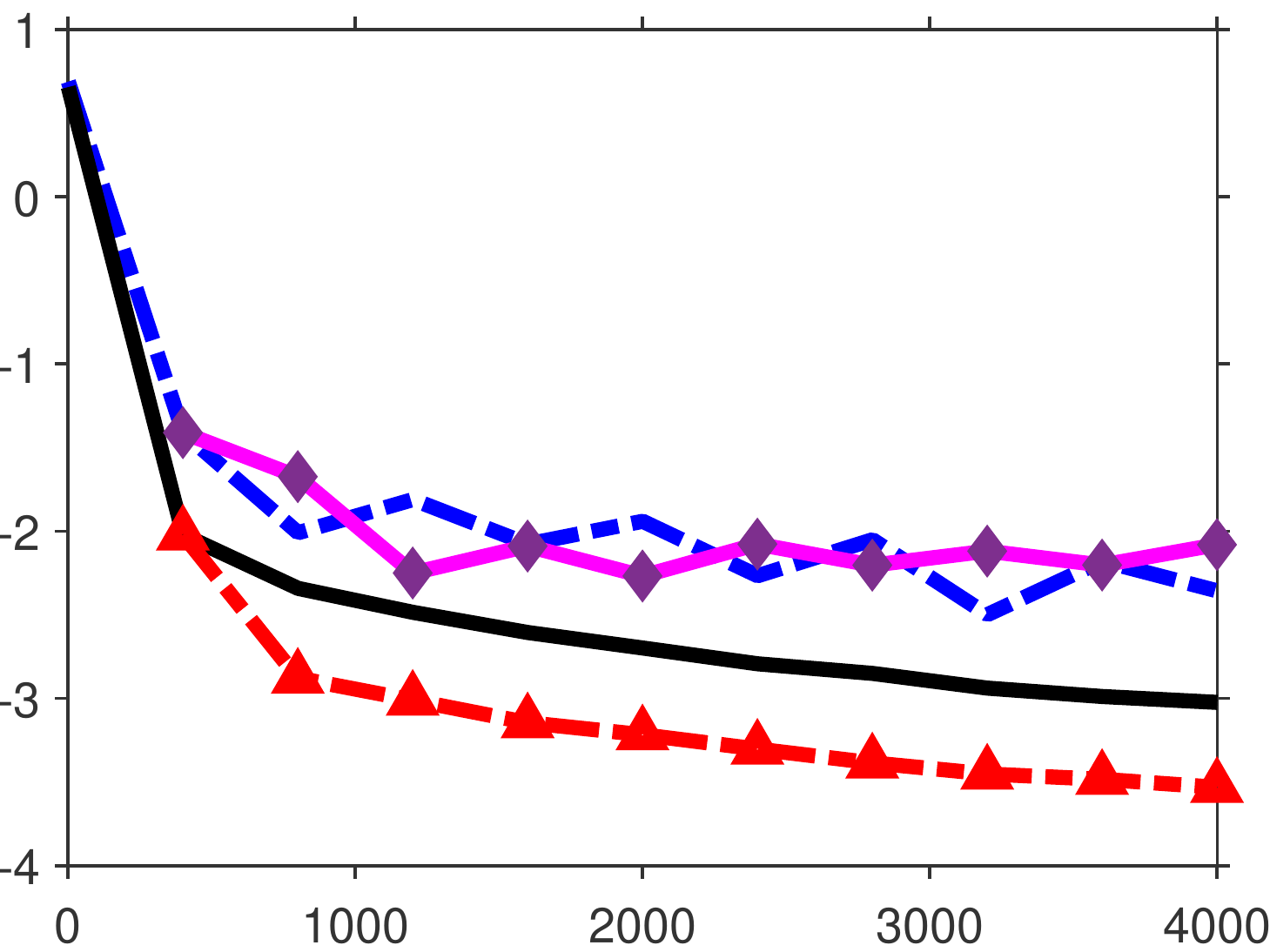}
\end{minipage}
&
\begin{minipage}{.29\textwidth}
\includegraphics[scale=.3, angle=0]{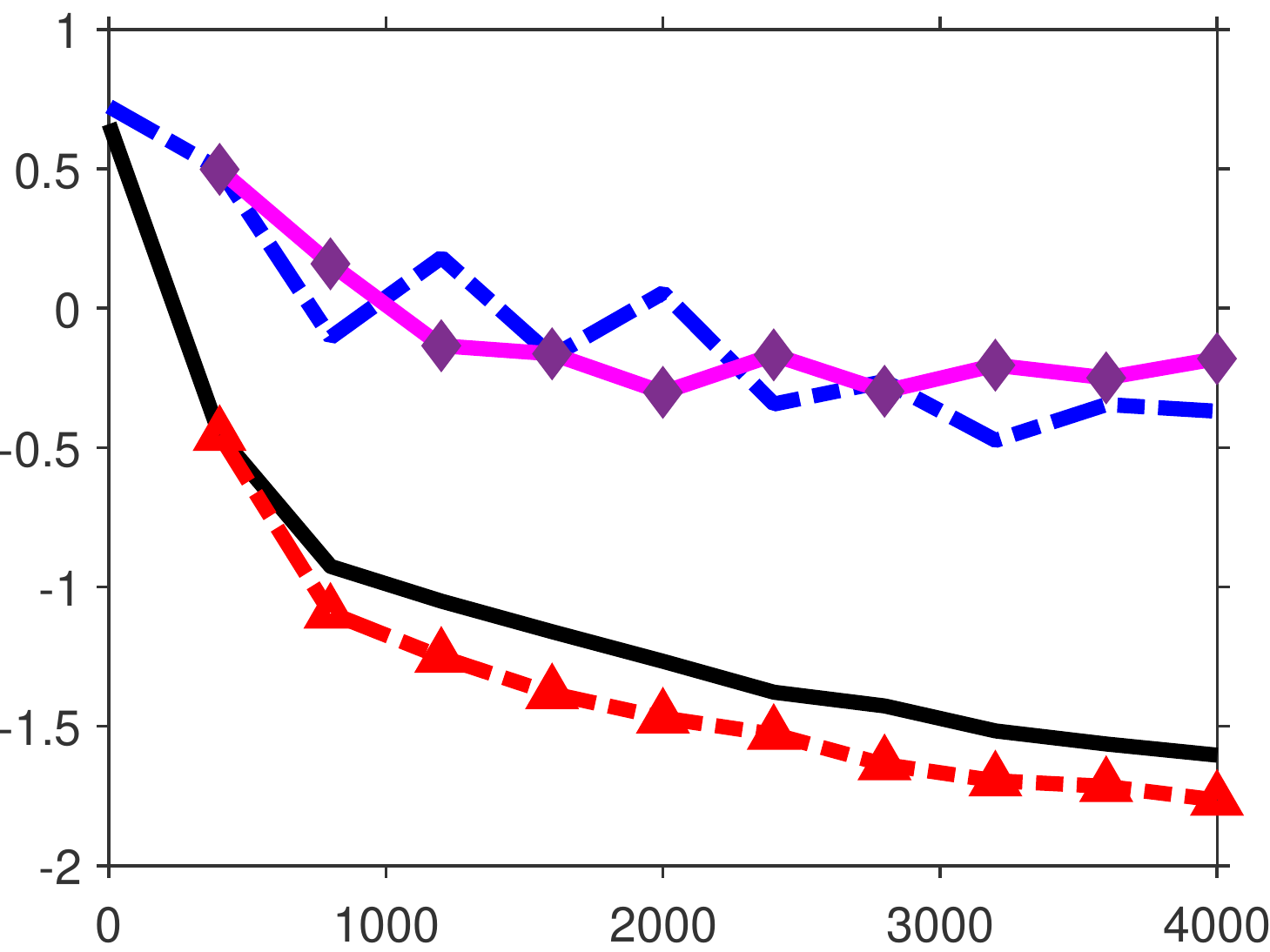}
\end{minipage}
\\
%\\
(4,4)
&
\begin{minipage}{.29\textwidth}
\includegraphics[scale=.3, angle=0]{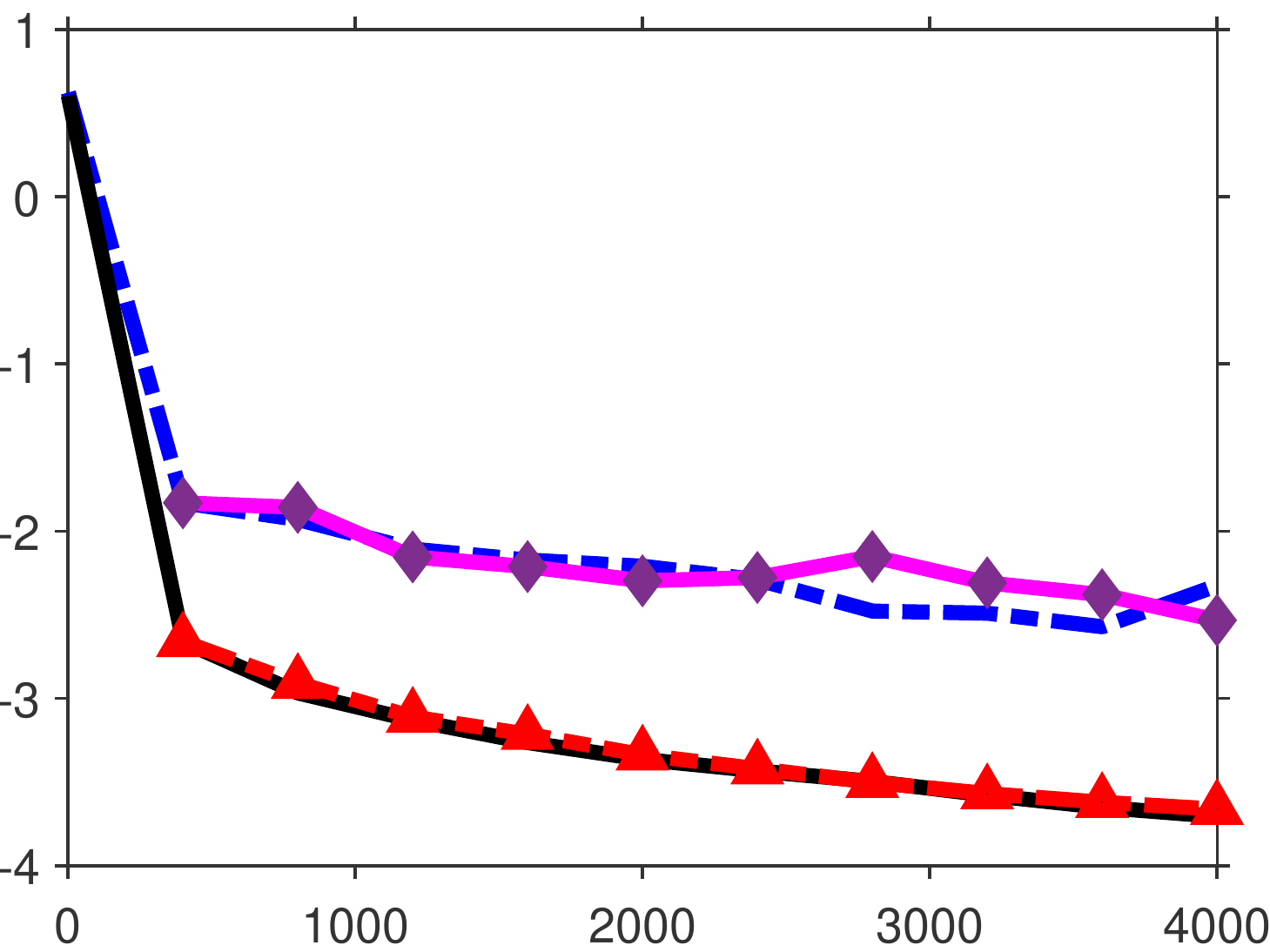}
\end{minipage}
&
\begin{minipage}{.29\textwidth}
\includegraphics[scale=.3, angle=0]{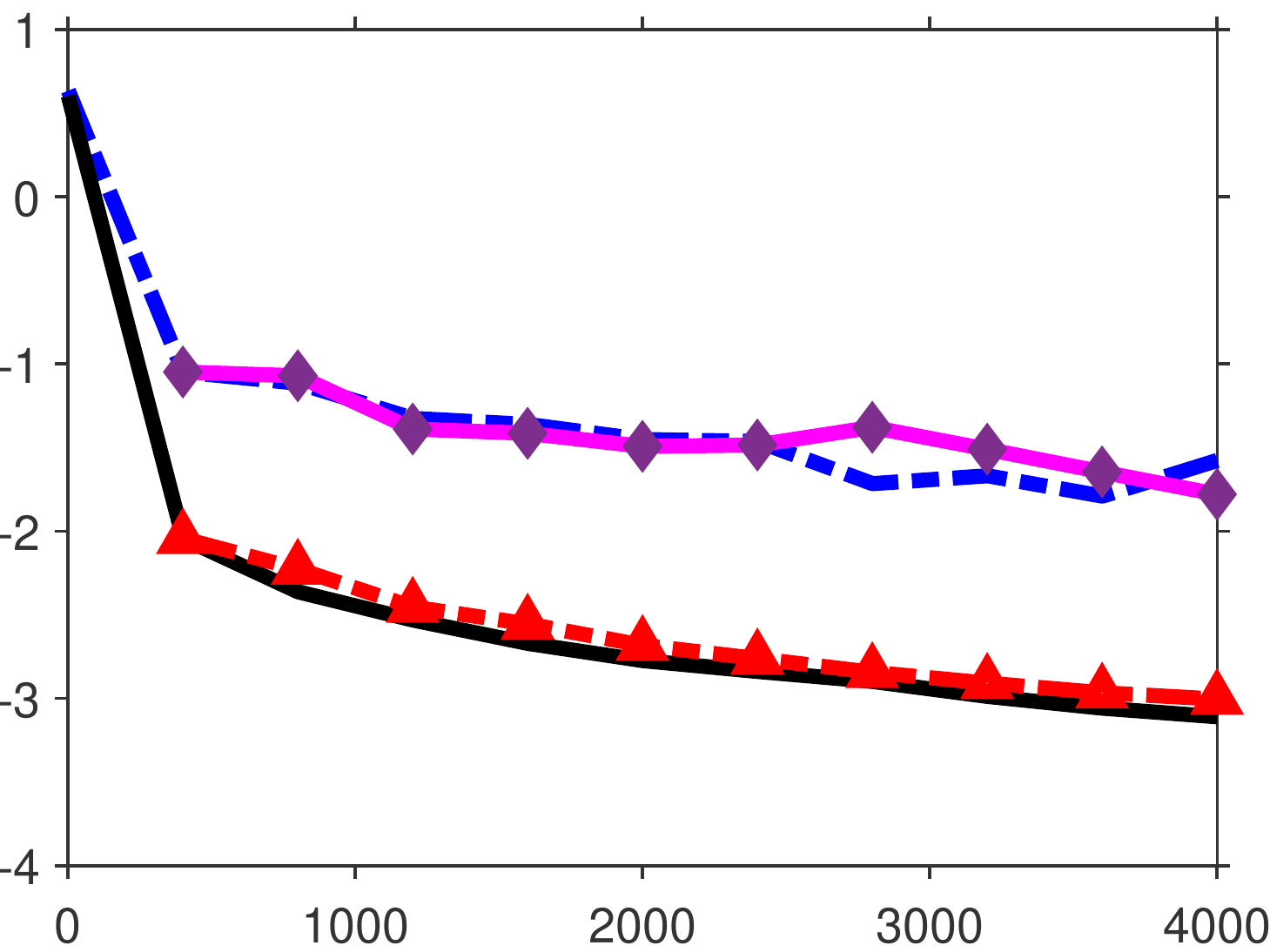}
\end{minipage}
&
\begin{minipage}{.29\textwidth}
\includegraphics[scale=.3, angle=0]{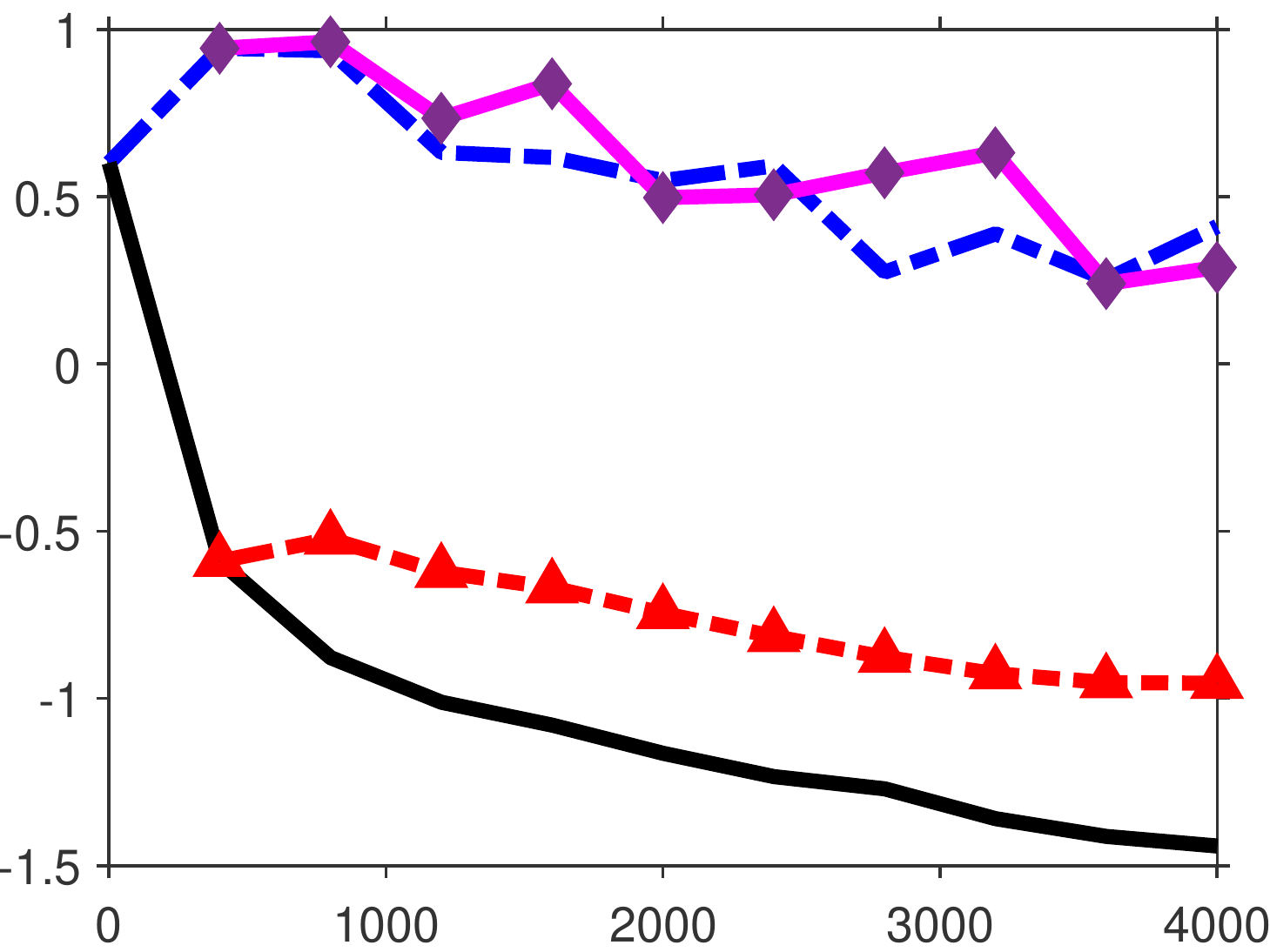}
\end{minipage}
\end{tabular}
\vspace{.1in}
\captionof{figure}{Comparison of M-SMD and A-M-SMD w.r.t. initial point ($X_0$), problem size ($n,m$), and uncertainty ($\sigma$) for 4000 iterations}
\label{fig:fourplots}
\end{table}
\begin{table}[h]
\setlength{\tabcolsep}{3pt}
\centering
 \begin{tabular}{c  c  c}
\begin{minipage}{.3\textwidth}
\includegraphics[scale=.33, angle=0]{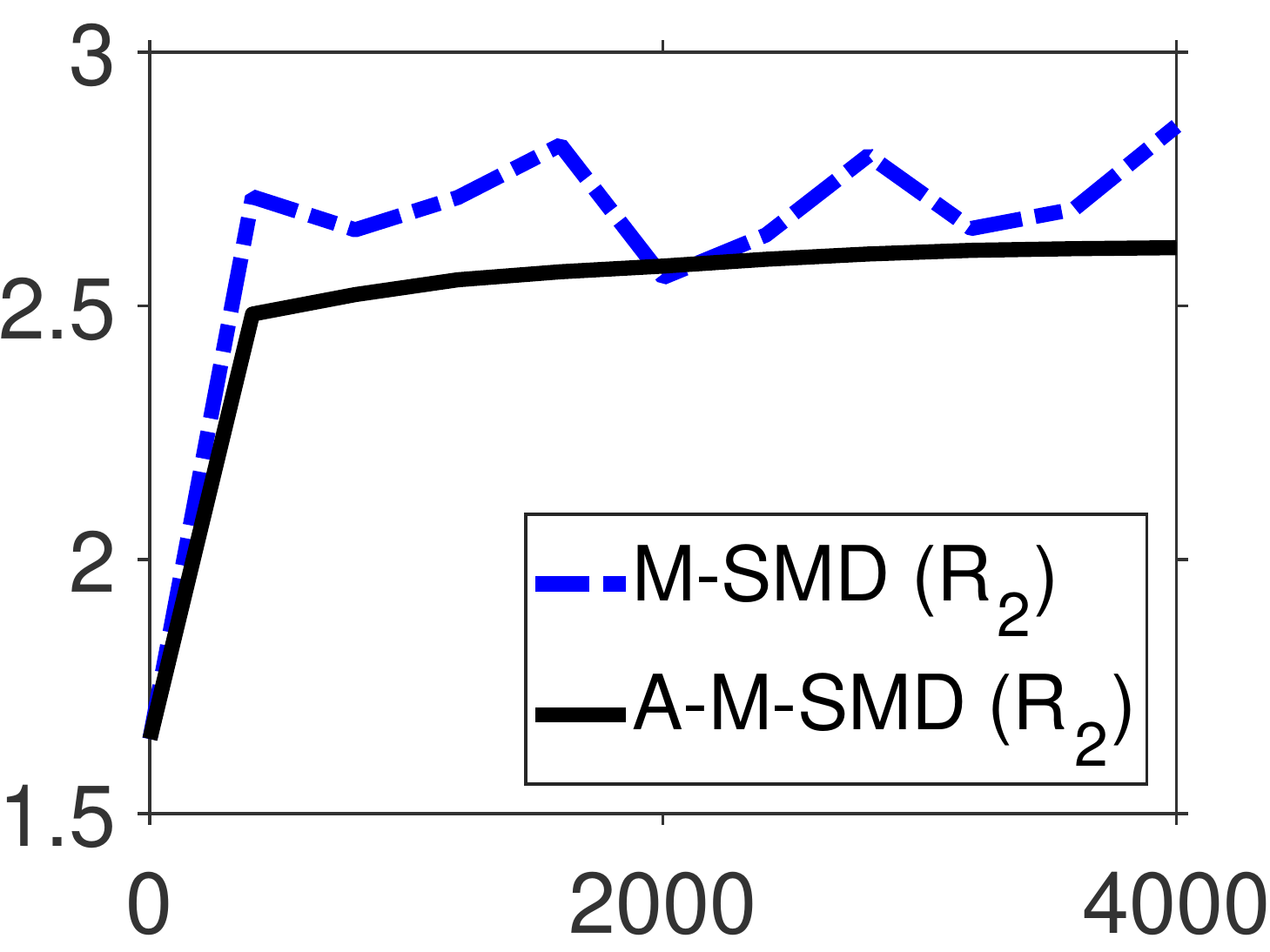}
\end{minipage}
&
\begin{minipage}{.3\textwidth}
\includegraphics[scale=.33, angle=0]{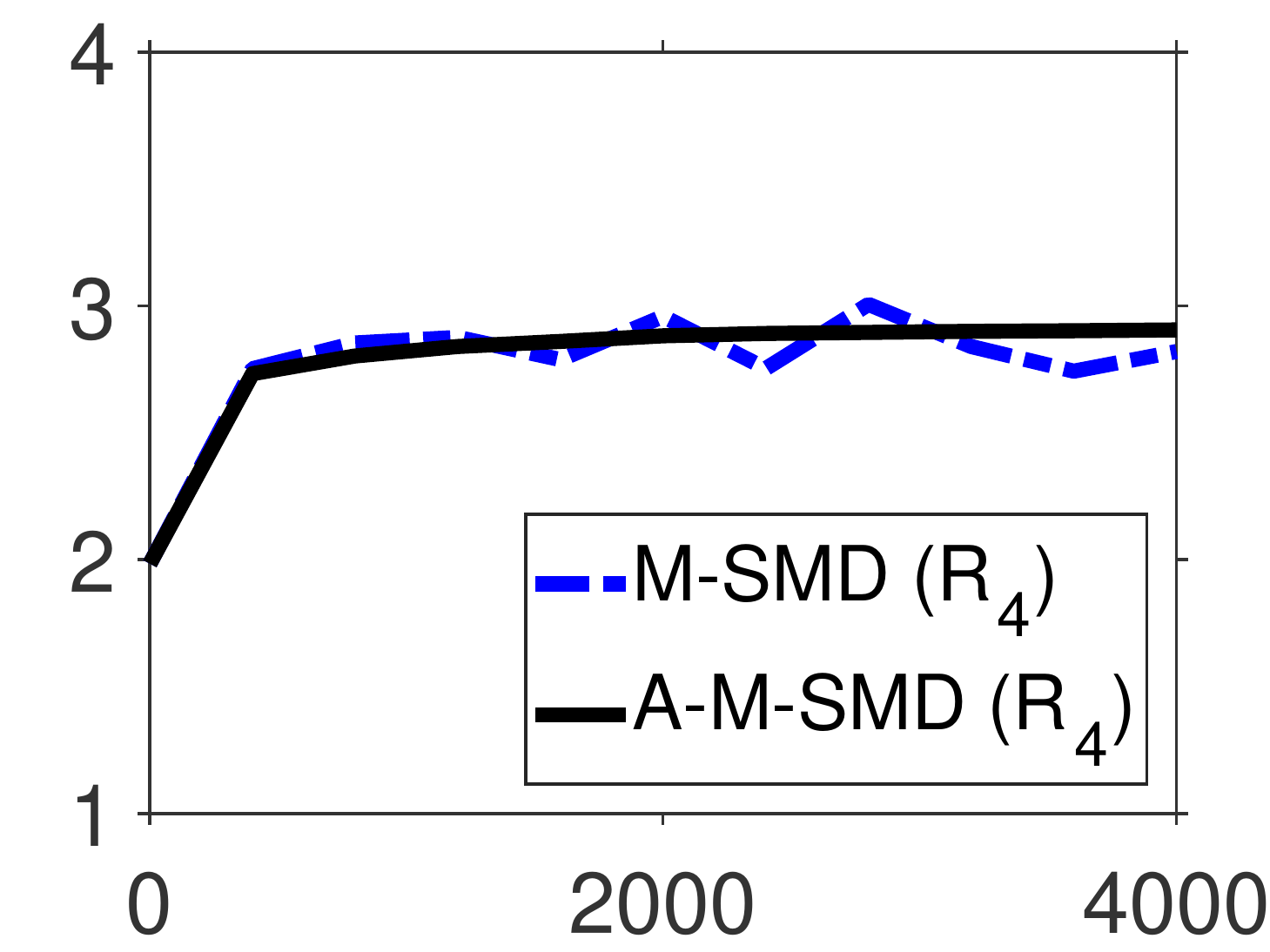}
\end{minipage}
&
\begin{minipage}{.3\textwidth}
\includegraphics[scale=.33, angle=0]{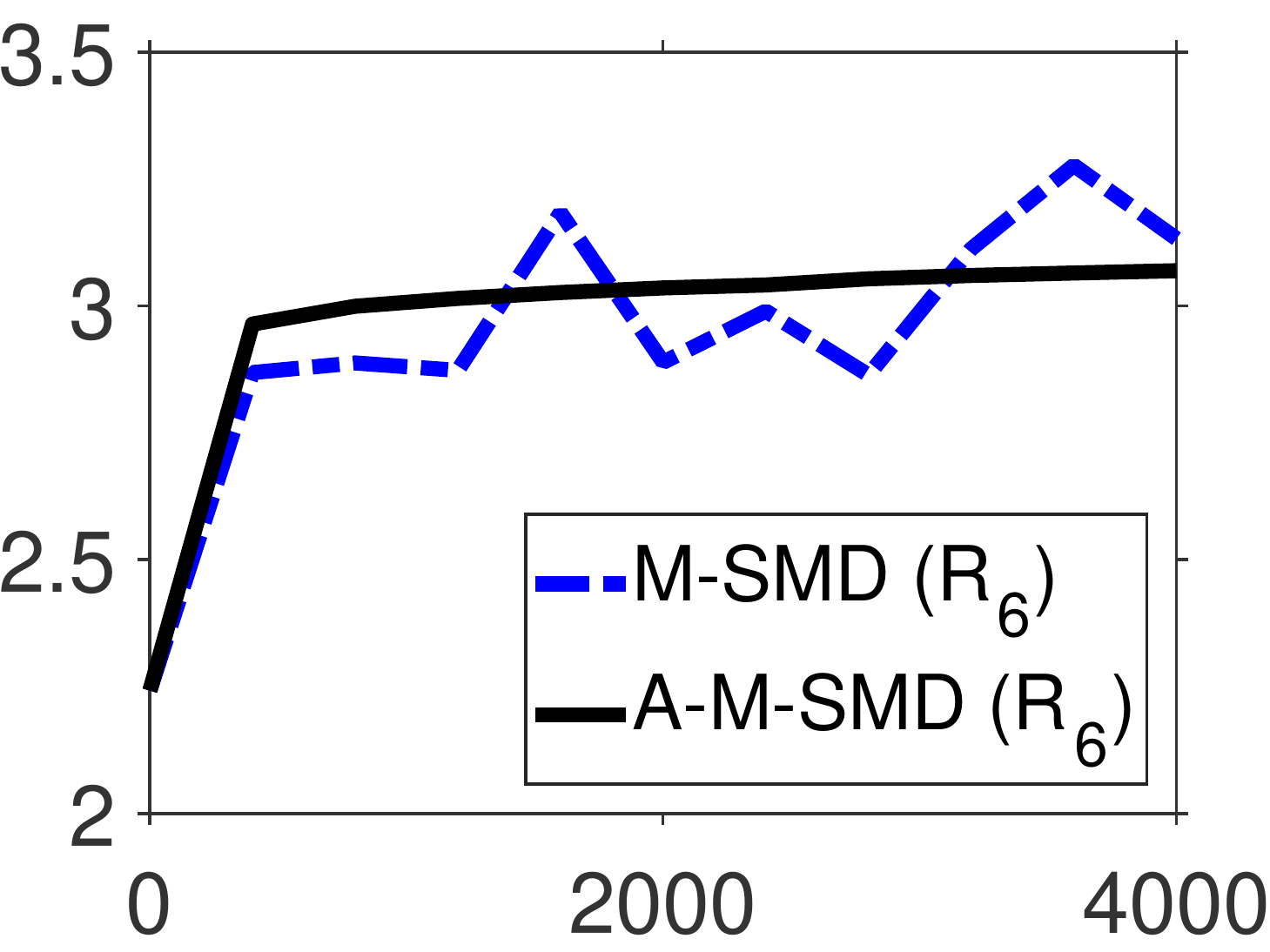}
\end{minipage}
\end{tabular}
\vspace{.1in}
\captionof{figure}{Comparison of stability of M-SMD and A-M-SMD in terms of users' objective function $R_i$ for $i=2,4,6$}
\label{fig:functionR}
\end{table}

\textbf{Stability of M-SMD and A-M-SMD:} To compare the stability of two methods, we also plot the expected objective function value $R_i$ against the iteration number in Figure \ref{fig:functionR}. Here, we choose $n=m=4$ and $\sigma=10$. The algorithm is repeated for $10$ sample paths and the average of objective function is obtained. Each plot represents the performance of both algorithms for one specific player $i \in\{1,\ldots,7\}$. As an example, the first plot compares the stability of A-M-SMD (black solid curve) and M-SMD (blue dash-dot curve) for the first user. It can be seen that for all players, the A-M-SMD algorithm converges to a strong solution very fast while the M-SMD does not converge and oscillates significantly. 
\begin{table}[h]
\setlength{\tabcolsep}{3pt}
\centering
 \begin{tabular}{c| c  c  c}
$(n,m)$ & $\sigma=0.5$ & $\sigma=1$ & $\sigma=5$ \\ \hline\\
(2,4)
&
\begin{minipage}{.29\textwidth}
\includegraphics[scale=.3, angle=0]{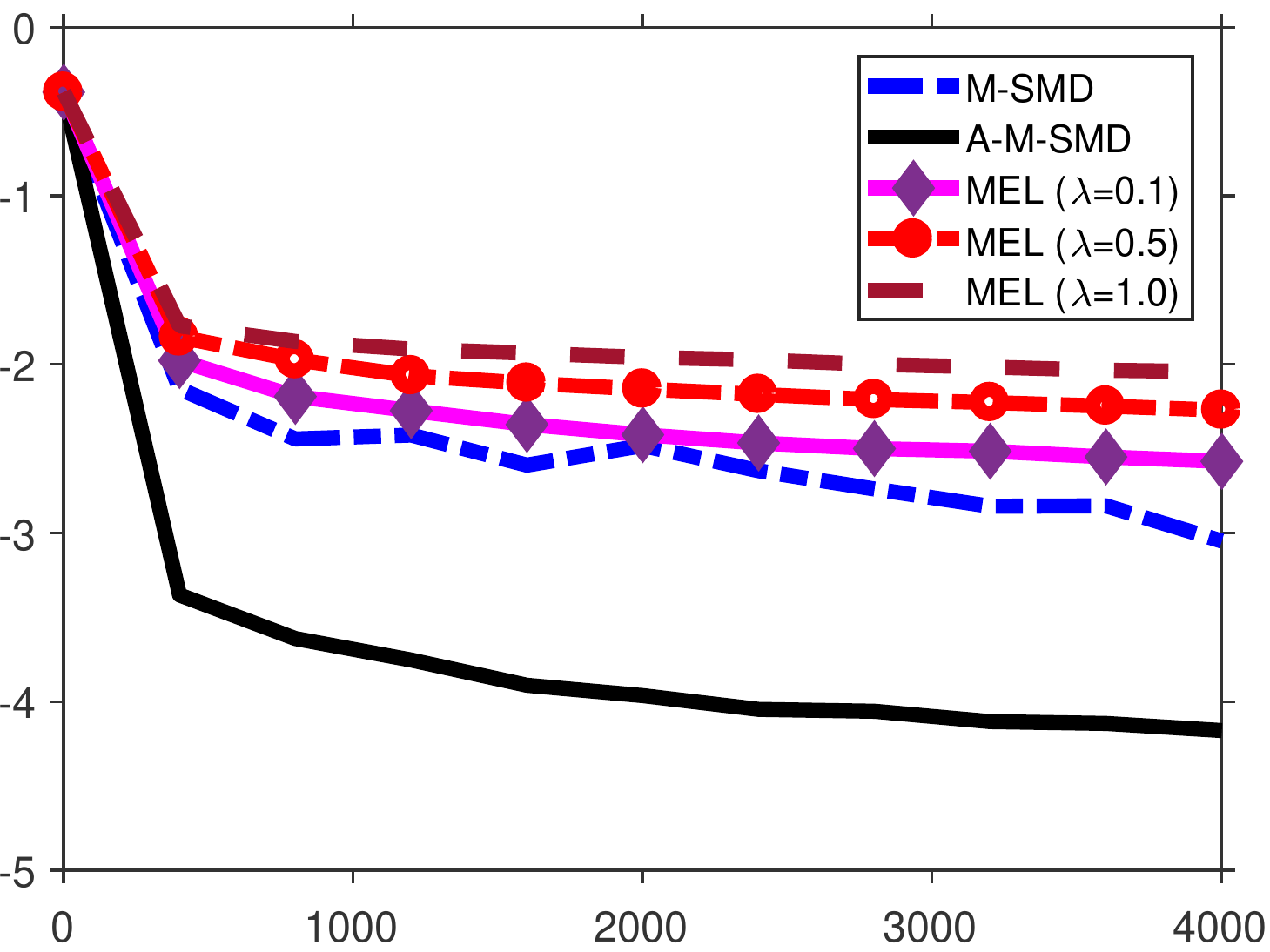}
%\vspace{.25in}
\end{minipage}
&
\begin{minipage}{.29\textwidth}
\includegraphics[scale=.3, angle=0]{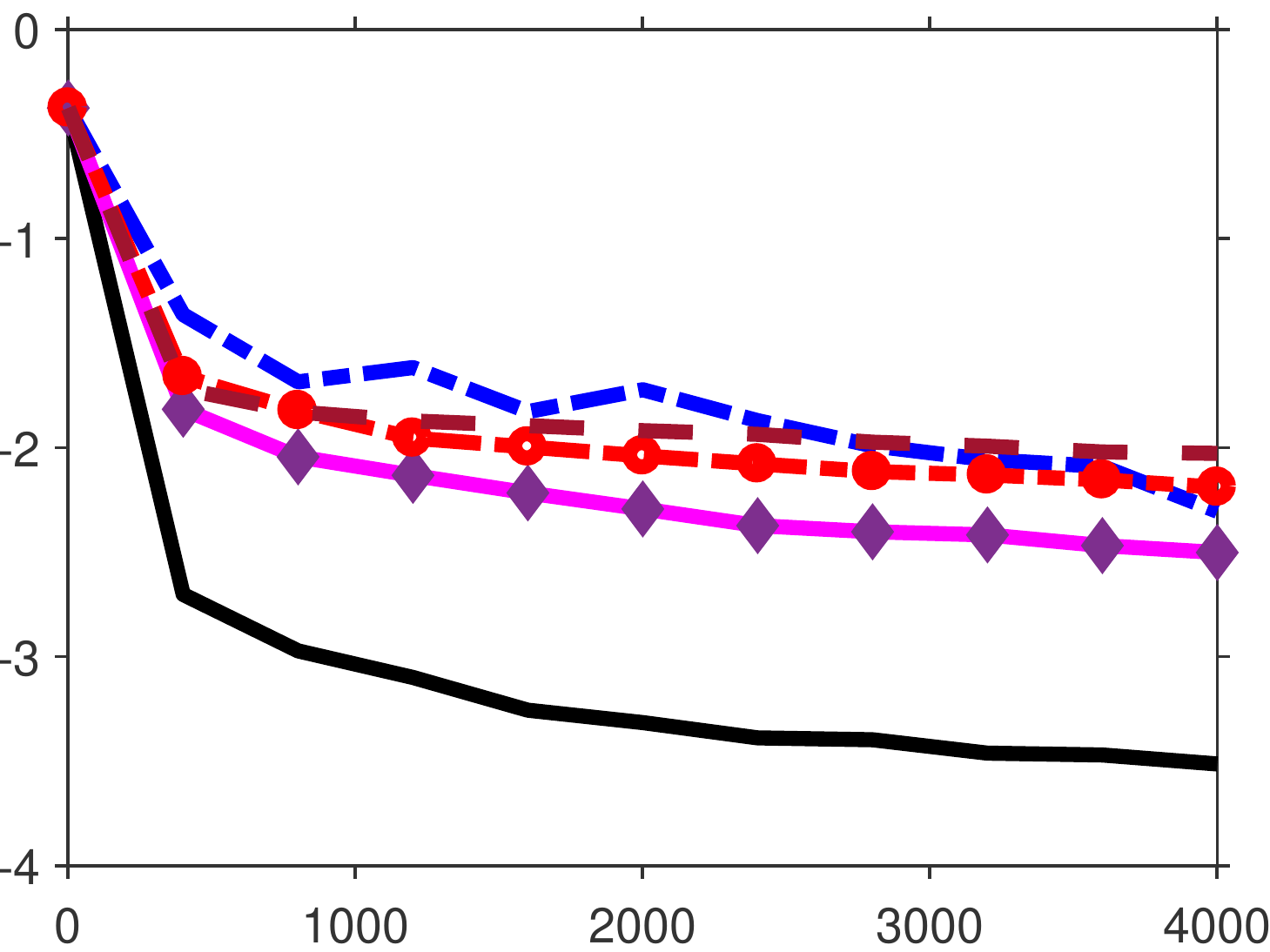}
%\vspace{.25in}
\end{minipage}
	&
\begin{minipage}{.29\textwidth}
\includegraphics[scale=.3, angle=0]{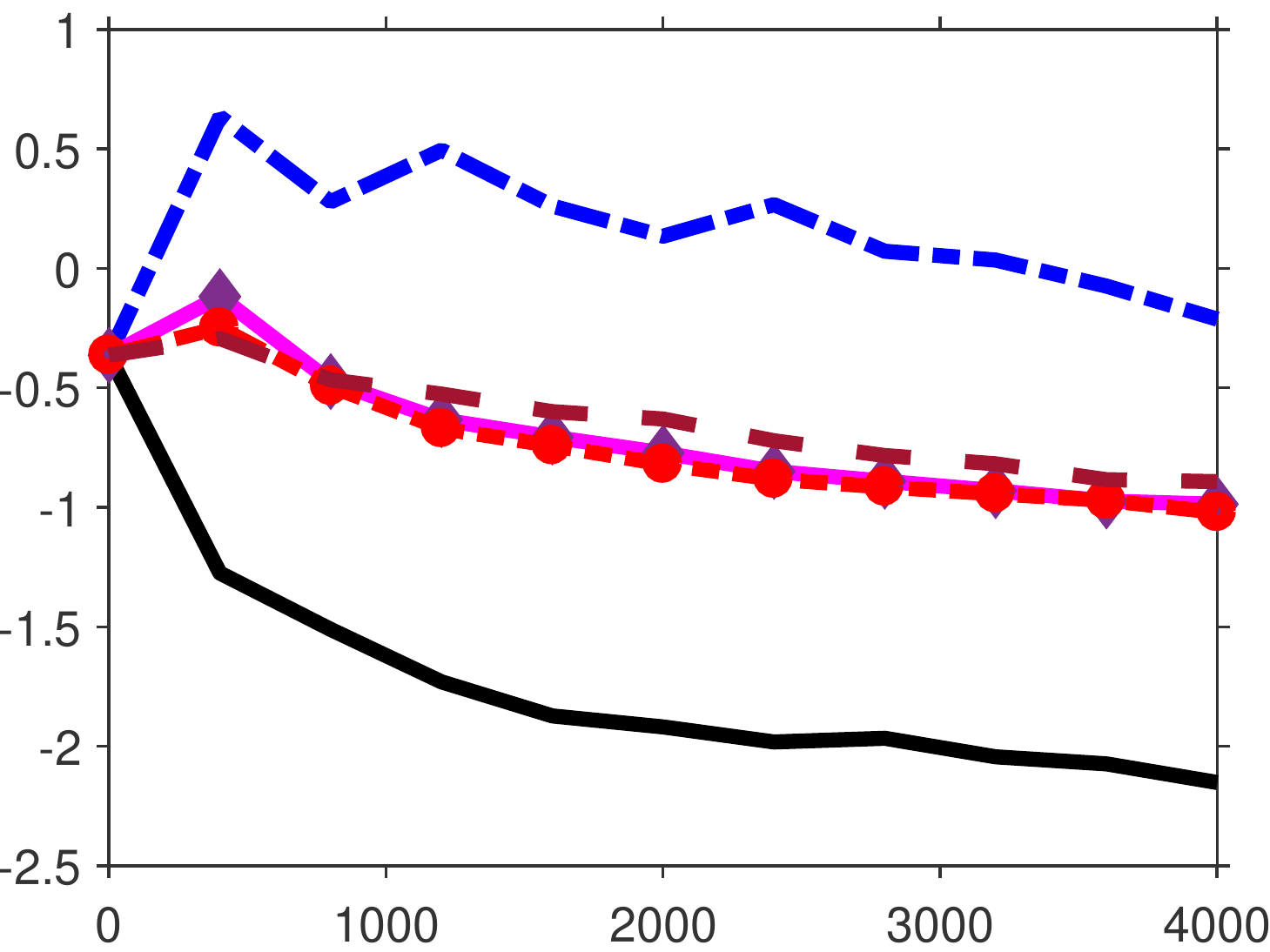}
%\vspace{.25in}
\end{minipage}
\\
(4,2)
&
\begin{minipage}{.29\textwidth}
\includegraphics[scale=.3, angle=0]{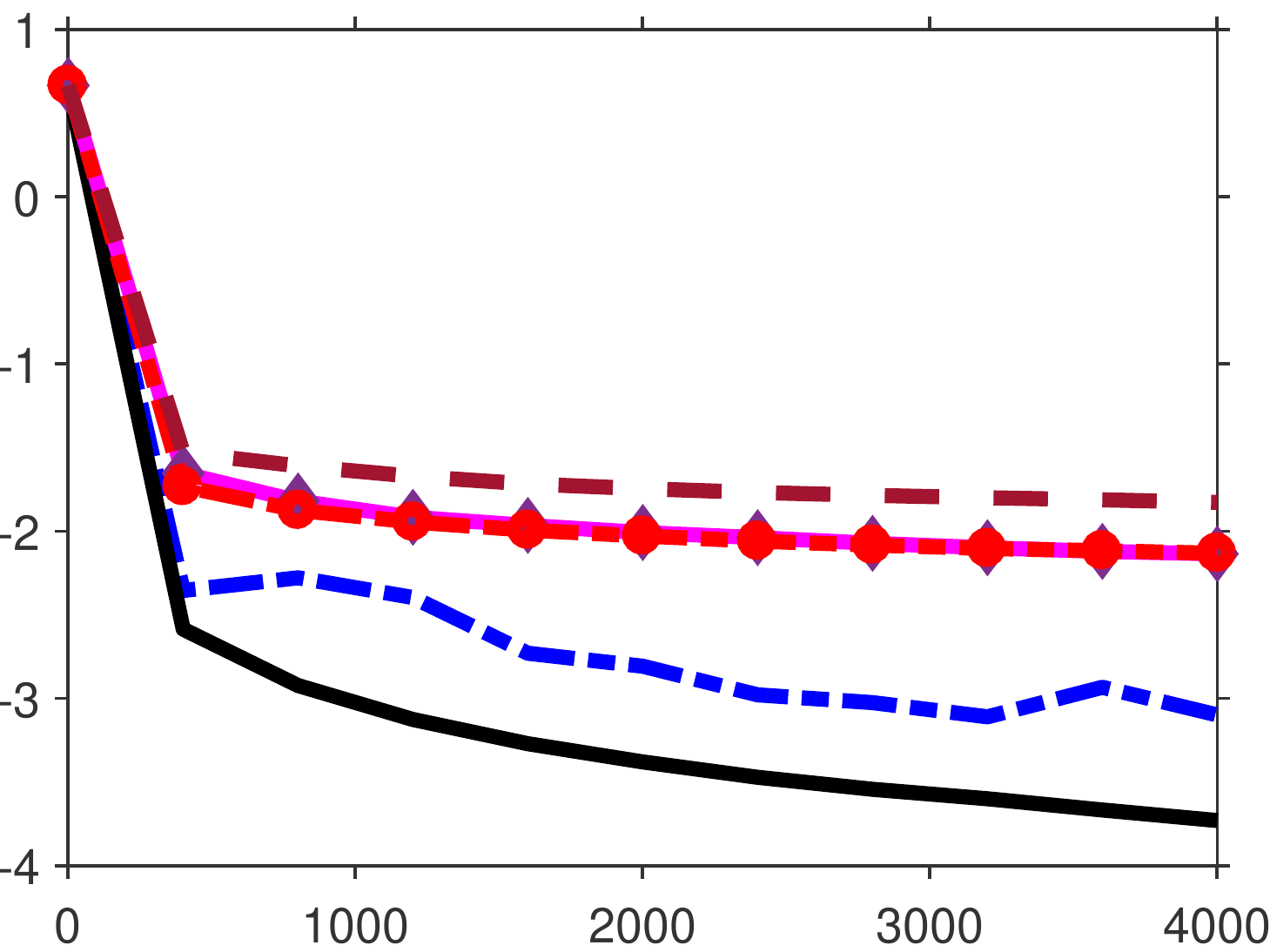}
\end{minipage}
&
\begin{minipage}{.29\textwidth}
\includegraphics[scale=.3, angle=0]{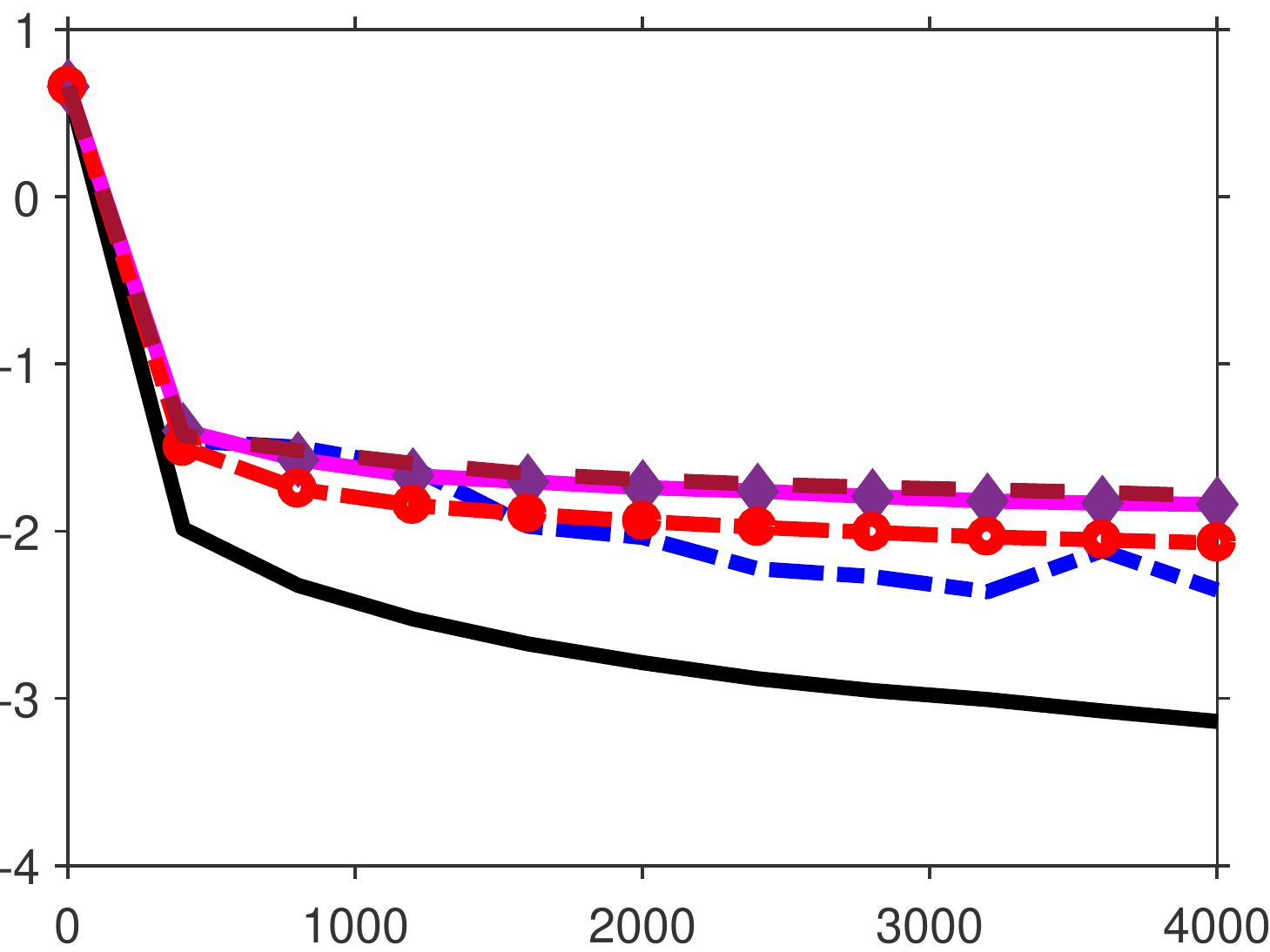}
\end{minipage}
&
\begin{minipage}{.29\textwidth}
\includegraphics[scale=.3, angle=0]{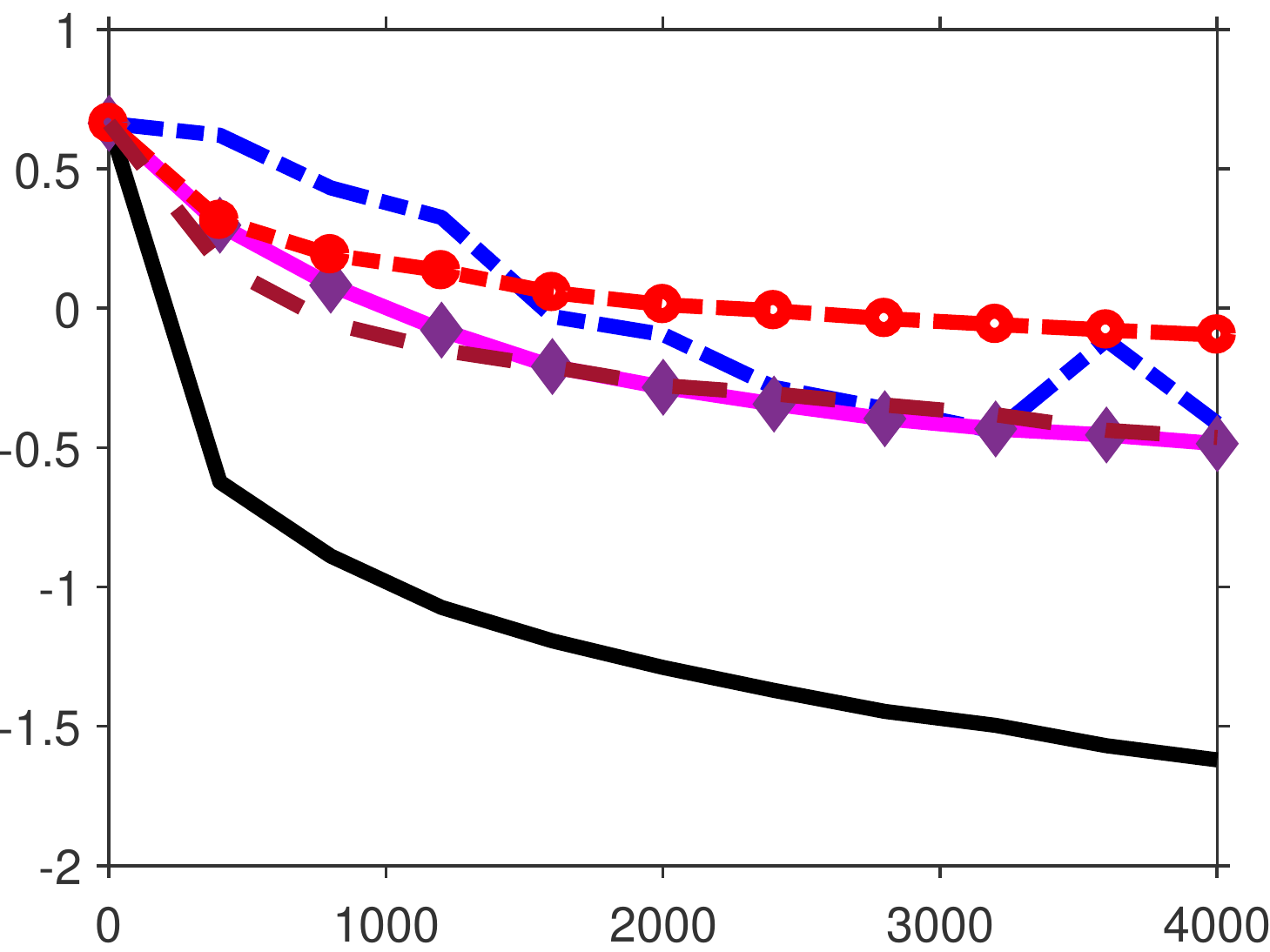}
\end{minipage}
\\
%\\
(4,4)
&
\begin{minipage}{.29\textwidth}
\includegraphics[scale=.3, angle=0]{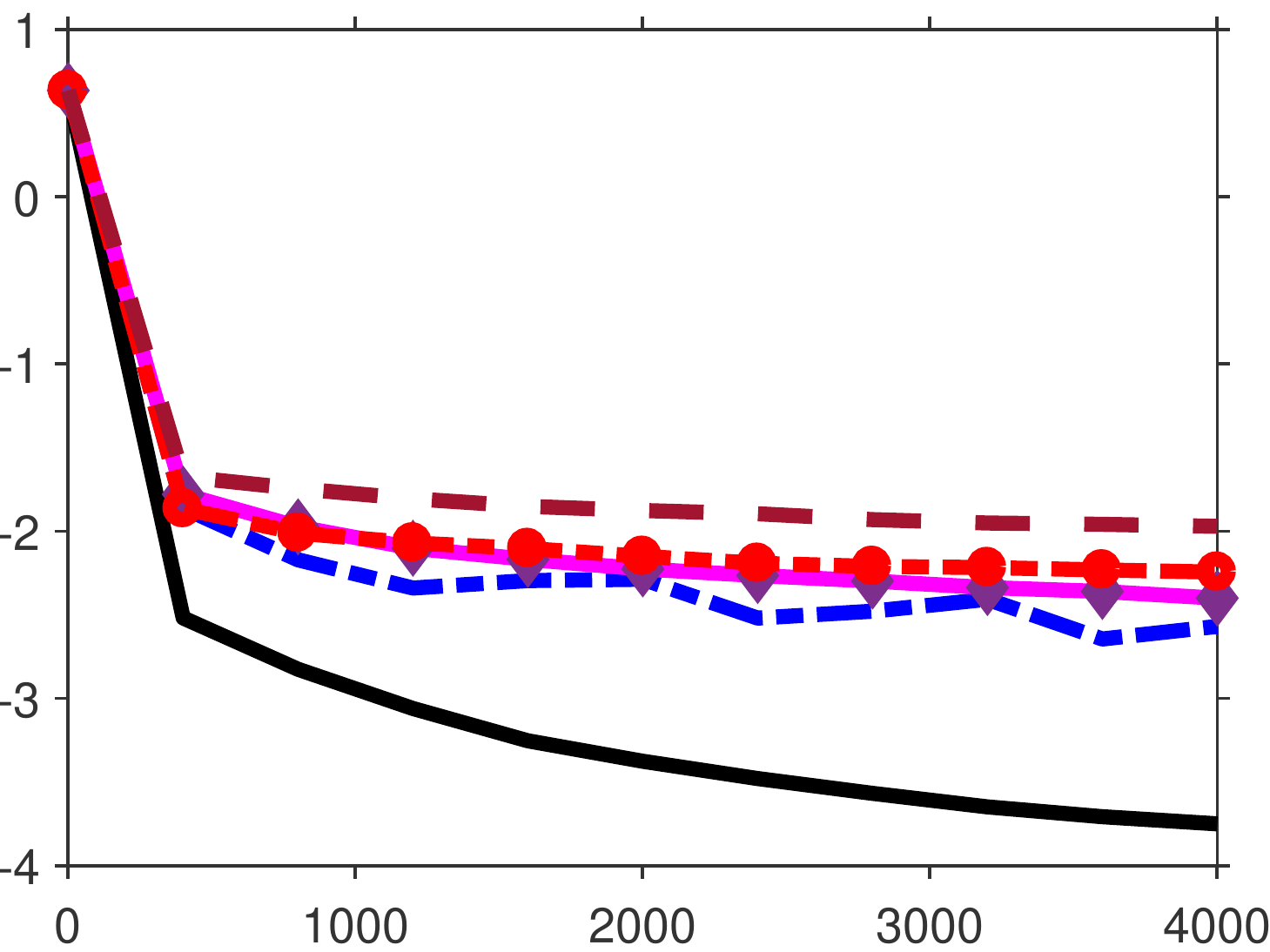}
\end{minipage}
&
\begin{minipage}{.29\textwidth}
\includegraphics[scale=.3, angle=0]{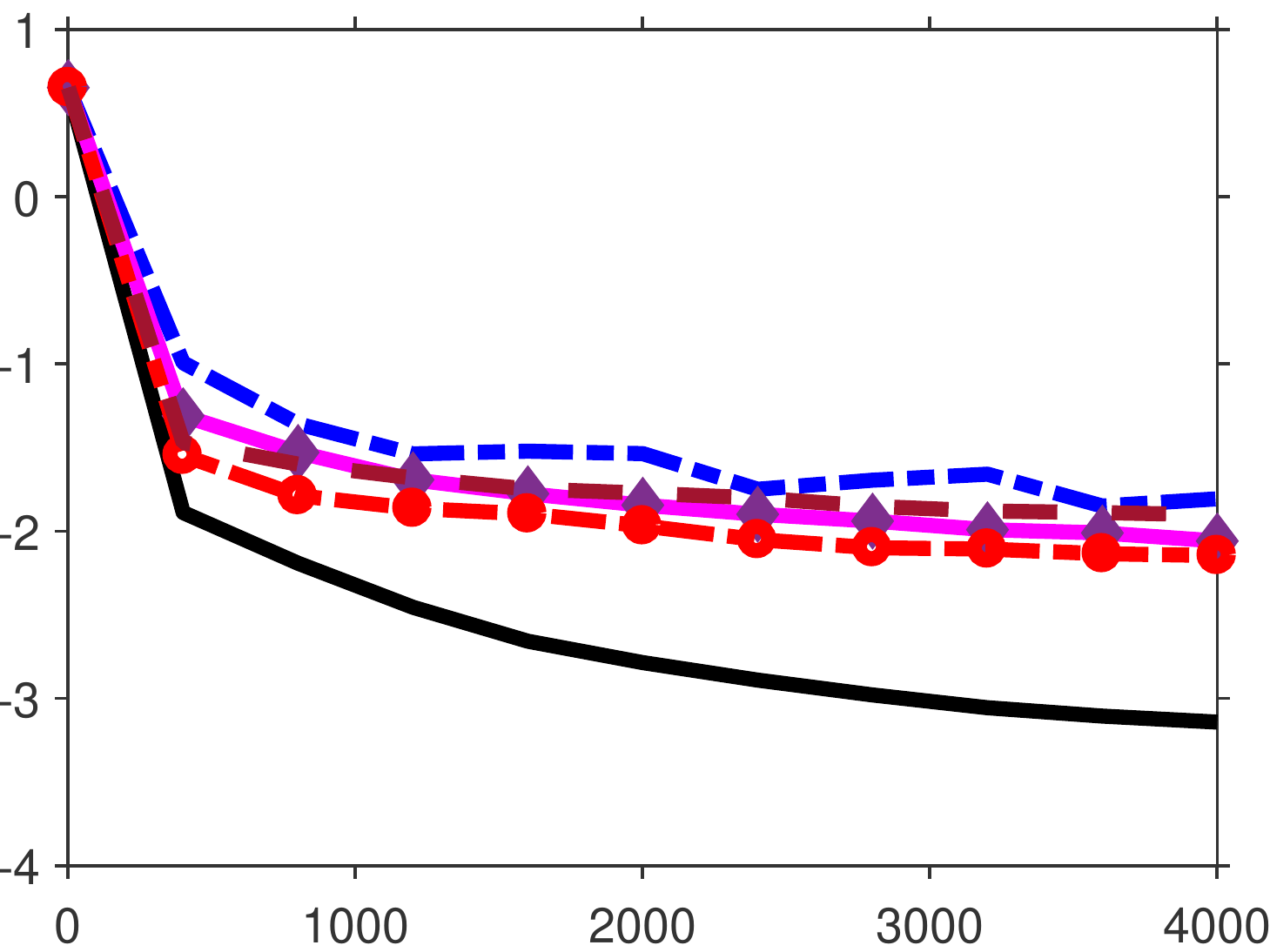}
\end{minipage}
&
\begin{minipage}{.29\textwidth}
\includegraphics[scale=.3, angle=0]{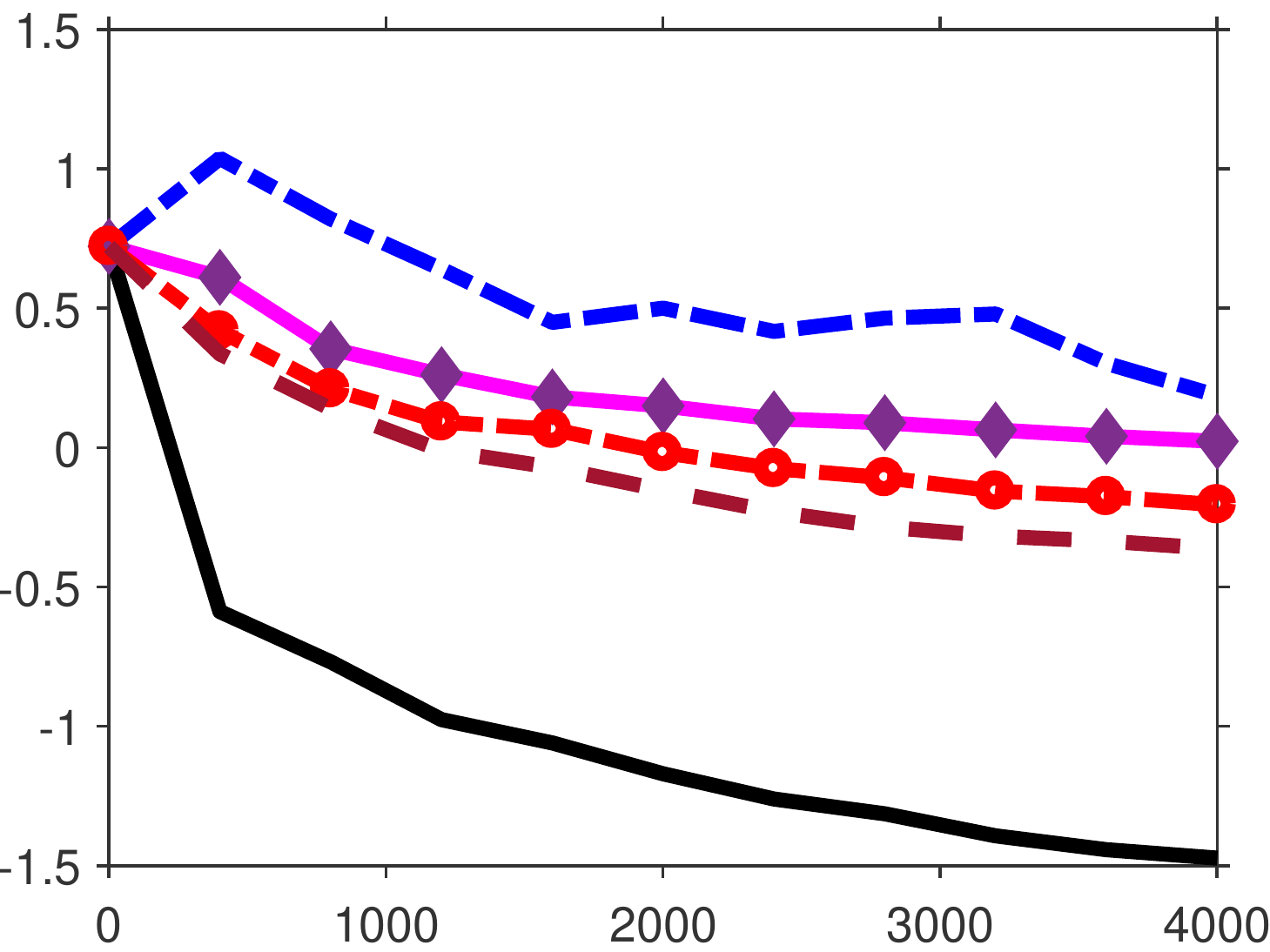}
\end{minipage}
\end{tabular}
\vspace{.1in}
\captionof{figure}{Comparison of M-SMD, A-M-SMD and MEL w.r.t. problem size ($n,m$), uncertainty ($\sigma$), and regularization parameter ($\lambda$) for 4000 iterations}
\label{fig:fiveplots}
\end{table}

\subsection{Matrix Exponential Learning} 
\cite{mertikopoulos2017distributed} proved the convergence of matrix exponential learning (MEL) algorithm under strong stability of mapping $F$ assumption while, in practice, this assumption might not hold for the games and VIs. We proved the convergence of A-M-SMD without assuming strong stability. For comparison purposes, we need to regularize the mapping $F$ by adding the gradient of a strongly convex function to it. Doing so, we obtain a strongly stable mapping (\cite{facchinei2007finite}, Chapter 2).
Let $\Vert A \Vert_F$ denote the Frobenius norm of a matrix $ A $ which is defined as $\Vert A \Vert_F=\sqrt{\tr{A^TA}}=\sqrt{\sum_{u}\sum_{v} |[A]_{uv}|^2}$ (\cite{golub2012matrix}). In the following Lemma, we show that the function $\frac{1}{2}\Vert  A \Vert_F^2$ is strongly convex. 
\begin{lemma} \label{lem:strong2}
The function $h( A )=\frac{1}{2}\Vert  A \Vert_F^2$ is strongly convex with parameter 1, i.e.,
\begin{align}\label{eq:strong2}
	\frac{1}{2}\Vert  B \Vert_F^2\geq \frac{1}{2}\Vert  A \Vert_F^2 +\tr{\nabla^T_{ A } h( A ) ( B - A )}+\frac{1}{2}\Vert  A - B \Vert_F^2.
\end{align}
\end{lemma}
The proof of Lemma \ref{lem:strong2} can be found in Appendix.

Note that $\nabla \frac{\lambda}{2} \Vert X  \Vert_F^2=\lambda X $. Therefore, to regularize the mapping $F$, we need to add the term $\lambda X $ to it and consequently, the mapping $F'=F+\lambda X $ is different from the original $F$. It should be noted for small values of $\lambda$, the algorithm converges very slowly. On the other hand, the solution which is obtained by using large values of $\lambda$ may be far from the solution to the original problem. Hence, we need to find a reasonable value of $\lambda$. For this reason, we tried three different values including $0.1, 0.5, 1$. Note that the difference between MEL and M-SMD algorithm is adding the term $\lambda X $ to the mapping $F$.

For each experiment, the algorithm is run for $T=4000$ iterations. We apply the well-known harmonic stepsize $\eta_t=\frac{1}{\sqrt{t}}$ for A-M-SMD and M-SMD, and harmonic stepsize $\eta_t=\frac{1}{t}$ for MEL. 
Figure \ref{fig:fiveplots} demonstrates the performance of A-M-SMD, M-SMD and MEL algorithms in terms of logarithm of expected value of gap function \eqref{gap3}. The expectation is taken over $ Z _t$, we repeat the algorithm for $10$ sample paths and obtain the average of gap function. In these plots, the blue (dash-dot) and black (solid) curves correspond to the M-SMD and A-M-SMD algorithms, respectively, the magenta (solid diamond), red (circle dashed) and brown (dashed) curves display MEL algorithm with $\lambda=0.1, 0.5$ and $1$.
As can be seen in Figure \ref{fig:fiveplots}, A-M-SMD algorithm outperforms the M-SMD and MEL algorithms in all experiments. It is evident that MEL algorithm converge slowly but faster than M-SMD. Comparing three versions of MEL algorithm which apply large, moderate or small value of regularization parameter $\lambda$, it can be seen that MEL is not robust w.r.t this parameter.
\section{Concluding Remarks} 
\label{sec:conclusion}
We consider multi-agent optimization problems on semidefinite matrix spaces. We develop mirror descent methods where we choose the distance generating function to be defined as the quantum entropy. These first-order single-loop methods include a mirror descent incremental subgradient (M-MDIS) method for minimizing a convex function that consists of sum of component functions and an averaging matrix stochastic mirror descent (A-M-SMD) method for solving Cartesian stochastic variational inequality problems under monotonicity assumption of the mapping. 
We show that the iterate generated by M-MDIS algorithm converges asymptotically to the optimal solution and derive a non-asymptotic
convergence rate. We also prove that A-M-SMD method converges to a weak solution of the CSVI with rate of ${\cal O}(1/\sqrt{t})$. Our numerical experiments performed on a wireless communication network display that the A-M-SMD method is significantly robust w.r.t. the problem size and uncertainty. 
\newpage
\bibliographystyle{ijocv081}
\bibliography{referenceIEEE} 

\begin{thebibliography}{59}
\expandafter\ifx\csname natexlab\endcsname\relax\def\natexlab#1{#1}\fi
\expandafter\ifx\csname url\endcsname\relax
  \def\url#1{{\tt #1}}\fi
\expandafter\ifx\csname urlprefix\endcsname\relax\def\urlprefix{URL }\fi
\expandafter\ifx\csname urlstyle\endcsname\relax
  \expandafter\ifx\csname doi\endcsname\relax
  \def\doi#1{doi:\discretionary{}{}{}#1}\fi \else
  \expandafter\ifx\csname doi\endcsname\relax
  \def\doi{doi:\discretionary{}{}{}\begingroup \urlstyle{rm}\Url}\fi \fi

\bibitem[{Athans and Schweppe(1965)}]{athans1965gradient}
Athans, Michael, Fred~C Schweppe. 1965.
\newblock Gradient matrices and matrix calculations.
\newblock Tech. rep., Massachusetts Inst of Tech Lexington Lab.

\bibitem[{Beck(2017)}]{Beck17}
Beck, A. 2017.
\newblock {\it First-{O}rder {M}ethods in {O}ptimization\/}.
\newblock Series: MOS-SIAM Series on Optimization, Philadelphia, PA.

\bibitem[{Bertsekas(2011)}]{bertsekas2011incremental}
Bertsekas, Dimitri~P. 2011.
\newblock Incremental proximal methods for large scale convex optimization.
\newblock {\it Mathematical programming\/} {\bf 129} 163.

\bibitem[{Bertsekas(2015)}]{bertsekas2015incremental}
Bertsekas, Dimitri~P. 2015.
\newblock Incremental aggregated proximal and augmented {L}agrangian
  algorithms.
\newblock {\it arXiv preprint arXiv:1509.09257\/} .

\bibitem[{Bien and Tibshirani(2011)}]{bien2011sparse}
Bien, Jacob, Robert~J Tibshirani. 2011.
\newblock Sparse estimation of a covariance matrix.
\newblock {\it Biometrika\/} {\bf 98} 807--820.

\bibitem[{Bo{\c{t}} and B{\"o}hm(2018)}]{boct2018incremental}
Bo{\c{t}}, Radu~Ioan, Axel B{\"o}hm. 2018.
\newblock An incremental mirror descent subgradient algorithm with random
  sweeping and proximal step.
\newblock {\it Optimization\/}  1--18.

\bibitem[{Boyd and Vandenberghe(2004)}]{boyd2004convex}
Boyd, Stephen, Lieven Vandenberghe. 2004.
\newblock {\it Convex optimization\/}.
\newblock Cambridge university press.

\bibitem[{Carlen(2010)}]{carlen2010trace}
Carlen, Eric. 2010.
\newblock Trace inequalities and quantum entropy: an introductory course.
\newblock {\it Entropy and the Quantum\/} {\bf 529} 73--140.

\bibitem[{Chang et~al.(2015)Chang, Hong, and Wang}]{chang2015multi}
Chang, Tsung-Hui, Mingyi Hong, Xiangfeng Wang. 2015.
\newblock Multi-agent distributed optimization via inexact consensus admm.
\newblock {\it IEEE Trans. Signal Processing\/} {\bf 63} 482--497.

\bibitem[{Chen et~al.(2017)Chen, Lan, and Ouyang}]{chen2017accelerated}
Chen, Yunmei, Guanghui Lan, Yuyuan Ouyang. 2017.
\newblock Accelerated schemes for a class of variational inequalities.
\newblock {\it Mathematical Programming\/} {\bf 165} 113--149.

\bibitem[{de~Pillis(1967)}]{de1967linear}
de~Pillis, John. 1967.
\newblock Linear transformations which preserve hermitian and positive
  semidefinite operators.
\newblock {\it Pacific Journal of Mathematics\/} {\bf 23} 129--137.

\bibitem[{Defazio et~al.(2014)Defazio, Bach, and
  Lacoste-Julien}]{defazio2014saga}
Defazio, Aaron, Francis Bach, Simon Lacoste-Julien. 2014.
\newblock Saga: A fast incremental gradient method with support for
  non-strongly convex composite objectives.
\newblock {\it Advances in neural information processing systems\/}.
  1646--1654.

\bibitem[{Durham et~al.(2012)Durham, Franchi, and
  Bullo}]{durham2012distributed}
Durham, Joseph~W, Antonio Franchi, Francesco Bullo. 2012.
\newblock Distributed pursuit-evasion without mapping or global localization
  via local frontiers.
\newblock {\it Autonomous Robots\/} {\bf 32} 81--95.

\bibitem[{Facchinei and Pang(2003)}]{facchinei02finite}
Facchinei, Francisco, Jong-Shi Pang. 2003.
\newblock {\it Finite-dimensional variational inequalities and complementarity
  problems. {V}ols. {I,II}\/}.
\newblock Springer Series in Operations Research, Springer-Verlag, New York.

\bibitem[{Facchinei and Pang(2007)}]{facchinei2007finite}
Facchinei, Francisco, Jong-Shi Pang. 2007.
\newblock {\it Finite-dimensional variational inequalities and complementarity
  problems\/}.
\newblock Springer Science \& Business Media.

\bibitem[{Fazel et~al.(2001)Fazel, Hindi, and Boyd}]{fazel2001rank}
Fazel, Maryam, Haitham Hindi, Stephen~P Boyd. 2001.
\newblock A rank minimization heuristic with application to minimum order
  system approximation.
\newblock {\it Proceedings of the American Control Conference\/}, vol.~6. IEEE,
  4734--4739.

\bibitem[{Golub and Van~Loan(2012)}]{golub2012matrix}
Golub, Gene~H, Charles~F Van~Loan. 2012.
\newblock {\it Matrix computations\/}, vol.~3.
\newblock JHU Press.

\bibitem[{Gurbuzbalaban et~al.(2017)Gurbuzbalaban, Ozdaglar, and
  Parrilo}]{gurbuzbalaban2017convergence}
Gurbuzbalaban, Mert, Asuman Ozdaglar, Pablo~A Parrilo. 2017.
\newblock On the convergence rate of incremental aggregated gradient
  algorithms.
\newblock {\it SIAM Journal on Optimization\/} {\bf 27} 1035--1048.

\bibitem[{Hiai and Petz(2014)}]{hiai2014introduction}
Hiai, Fumio, D{\'e}nes Petz. 2014.
\newblock {\it Introduction to matrix analysis and applications\/}.
\newblock Springer Science \& Business Media.

\bibitem[{Hsieh et~al.(2013)Hsieh, Sustik, Dhillon, Ravikumar, and
  Poldrack}]{hsieh2013big}
Hsieh, Cho-Jui, M{\'a}ty{\'a}s~A Sustik, Inderjit~S Dhillon, Pradeep~K
  Ravikumar, Russell Poldrack. 2013.
\newblock {BIG} \& {QUIC}: Sparse inverse covariance estimation for a million
  variables.
\newblock {\it Advances in neural information processing systems\/}.
  3165--3173.

\bibitem[{Jiang and Xu(2008)}]{jiang2008stochastic}
Jiang, Houyuan, Huifu Xu. 2008.
\newblock Stochastic approximation approaches to the stochastic variational
  inequality problem.
\newblock {\it IEEE Transactions on Automatic Control\/} {\bf 53} 1462--1475.

\bibitem[{Johnson and Zhang(2013)}]{johnson2013accelerating}
Johnson, Rie, Tong Zhang. 2013.
\newblock Accelerating stochastic gradient descent using predictive variance
  reduction.
\newblock {\it Advances in neural information processing systems\/}. 315--323.

\bibitem[{Juditsky et~al.(2011)Juditsky, Nemirovski, and
  Tauvel}]{juditsky2011solving}
Juditsky, Anatoli, Arkadi Nemirovski, Claire Tauvel. 2011.
\newblock Solving variational inequalities with stochastic mirror-prox
  algorithm.
\newblock {\it Stochastic Systems\/} {\bf 1} 17--58.

\bibitem[{Kakade et~al.(2009)Kakade, Shalev-Shwartz, and
  Tewari}]{kakade2009duality}
Kakade, Sham, Shai Shalev-Shwartz, Ambuj Tewari. 2009.
\newblock On the duality of strong convexity and strong smoothness: Learning
  applications and matrix regularization.
\newblock {\it Unpublished Manuscript, http://ttic. uchicago.
  edu/shai/papers/KakadeShalevTewari09. pdf\/} .

\bibitem[{Koshal et~al.(2013)Koshal, Nedi{\'c}, and
  Shanbhag}]{koshal2013regularized}
Koshal, Jayash, Angelia Nedi{\'c}, Uday~V. Shanbhag. 2013.
\newblock Regularized iterative stochastic approximation methods for stochastic
  variational inequality problems.
\newblock {\it IEEE Transactions on Automatic Control\/} {\bf 58} 594--609.

\bibitem[{Kwong(1989)}]{kwong1989some}
Kwong, Man~Kam. 1989.
\newblock Some results on matrix monotone functions.
\newblock {\it Linear Algebra and Its Applications\/} {\bf 118} 129--153.

\bibitem[{Lan et~al.(2011)Lan, Lu, and Monteiro}]{lan2011primal}
Lan, Guanghui, Zhaosong Lu, Renato~DC Monteiro. 2011.
\newblock Primal-dual first-order methods with ${O}(1/\epsilon)$
  iteration-complexity for cone programming.
\newblock {\it Mathematical Programming\/} {\bf 126} 1--29.

\bibitem[{Lobel and Ozdaglar(2011)}]{lobel2011distributed}
Lobel, Ilan, Asuman Ozdaglar. 2011.
\newblock Distributed subgradient methods for convex optimization over random
  networks.
\newblock {\it IEEE Transactions on Automatic Control\/} {\bf 56} 1291.

\bibitem[{Lu(2010)}]{lu2010adaptive}
Lu, Zhaosong. 2010.
\newblock Adaptive first-order methods for general sparse inverse covariance
  selection.
\newblock {\it SIAM Journal on Matrix Analysis and Applications\/} {\bf 31}
  2000--2016.

\bibitem[{Majlesinasab et~al.(2019{\natexlab{a}})Majlesinasab, Yousefian, and
  Feizollahi}]{majlesinasab2018first}
Majlesinasab, Nahidsadat, Farzad Yousefian, Mohammad~Javad Feizollahi.
  2019{\natexlab{a}}.
\newblock A first-order method for monotone stochastic variational inequalities
  on semidefinite matrix spaces.
\newblock {\it accepted for publication in Proceedings of the American Control
  Conference\/} .

\bibitem[{Majlesinasab et~al.(2019{\natexlab{b}})Majlesinasab, Yousefian, and
  Pourhabib}]{majlesinasab2017optimal}
Majlesinasab, Nahidsadat, Farzad Yousefian, Arash Pourhabib.
  2019{\natexlab{b}}.
\newblock Self-tuned mirror descent schemes for smooth and nonsmooth
  high-dimensional stochastic optimization.
\newblock {\it accepted for publication in IEEE Transactions on Automatic
  Control\/} .

\bibitem[{Makhdoumi and Ozdaglar(2017)}]{makhdoumi2017convergence}
Makhdoumi, Ali, Asuman Ozdaglar. 2017.
\newblock Convergence rate of distributed admm over networks.
\newblock {\it IEEE Transactions on Automatic Control\/} {\bf 62} 5082--5095.

\bibitem[{Mertikopoulos et~al.(2012)Mertikopoulos, Belmega, and
  Moustakas}]{mertikopoulos2012matrix}
Mertikopoulos, Panayotis, E~Veronica Belmega, Aris~L Moustakas. 2012.
\newblock Matrix exponential learning: Distributed optimization in {MIMO}
  systems.
\newblock {\it Information Theory Proceedings (ISIT), 2012 IEEE International
  Symposium on\/}. IEEE, 3028--3032.

\bibitem[{Mertikopoulos et~al.(2017)Mertikopoulos, Belmega, Negrel, and
  Sanguinetti}]{mertikopoulos2017distributed}
Mertikopoulos, Panayotis, E~Veronica Belmega, Romain Negrel, Luca Sanguinetti.
  2017.
\newblock Distributed stochastic optimization via matrix exponential learning.
\newblock {\it IEEE Transactions on Signal Processing\/} {\bf 65} 2277--2290.

\bibitem[{Mertikopoulos and Moustakas(2016)}]{mertikopoulos2016learning}
Mertikopoulos, Panayotis, Aris~L Moustakas. 2016.
\newblock Learning in an uncertain world: {MIMO} covariance matrix optimization
  with imperfect feedback.
\newblock {\it IEEE Transactions on Signal Processing\/} {\bf 64} 5--18.

\bibitem[{Mertikopoulos and Sandholm(2016)}]{mertikopoulos2016learning2}
Mertikopoulos, Panayotis, William~H Sandholm. 2016.
\newblock Learning in games via reinforcement and regularization.
\newblock {\it Mathematics of Operations Research\/} {\bf 41} 1297--1324.

\bibitem[{Necoara et~al.(2017)Necoara, Patrascu, and
  Glineur}]{necoara2017complexity}
Necoara, Ion, Andrei Patrascu, Francois Glineur. 2017.
\newblock Complexity of first-order inexact {L}agrangian and penalty methods
  for conic convex programming.
\newblock {\it Optimization Methods and Software\/}  1--31.

\bibitem[{Nedi{\'c}(2011)}]{nedic2011asynchronous}
Nedi{\'c}, Angelia. 2011.
\newblock Asynchronous broadcast-based convex optimization over a network.
\newblock {\it IEEE Transactions on Automatic Control\/} {\bf 56} 1337--1351.

\bibitem[{Nedi{\'c} and Olshevsky(2015)}]{nedic2015distributed}
Nedi{\'c}, Angelia, Alex Olshevsky. 2015.
\newblock Distributed optimization over time-varying directed graphs.
\newblock {\it IEEE Transactions on Automatic Control\/} {\bf 60} 601--615.

\bibitem[{Nedi{\'c} et~al.(2017)Nedi{\'c}, Olshevsky, and
  Uribe}]{nedic2017distributed}
Nedi{\'c}, Angelia, Alex Olshevsky, C{\'e}sar~A Uribe. 2017.
\newblock Distributed learning for cooperative inference.
\newblock {\it arXiv preprint arXiv:1704.02718\/} .

\bibitem[{Nedi{\'c} and Ozdaglar(2009)}]{nedic2009distributed}
Nedi{\'c}, Angelia, Asuman Ozdaglar. 2009.
\newblock Distributed subgradient methods for multi-agent optimization.
\newblock {\it IEEE Transactions on Automatic Control\/} {\bf 54} 48--61.

\bibitem[{Nemirovski et~al.(2009)Nemirovski, Juditsky, Lan, and
  Shapiro}]{nemirovski2009robust}
Nemirovski, Arkadi, Anatoli Juditsky, Guanghui Lan, Alexander Shapiro. 2009.
\newblock Robust stochastic approximation approach to stochastic programming.
\newblock {\it SIAM Journal on optimization\/} {\bf 19} 1574--1609.

\bibitem[{Polyak and Juditsky(1992)}]{Polyak92}
Polyak, Boris~T, Anatoli~B Juditsky. 1992.
\newblock Acceleration of stochastic approximation by averaging.
\newblock {\it SIAM Journal on Control and Optimization\/} {\bf 30} 838--855.

\bibitem[{Ram et~al.(2009)Ram, Veeravalli, and Nedi{\'c}}]{ram2009distributed}
Ram, Sundhar~Srinivasan, Venugopal~V Veeravalli, Angelia Nedi{\'c}. 2009.
\newblock Distributed non-autonomous power control through distributed convex
  optimization.
\newblock {\it INFOCOM 2009, IEEE\/}. IEEE, 3001--3005.

\bibitem[{Robbins and Monro(1951)}]{robbins1951stochastic}
Robbins, Herbert, Sutton Monro. 1951.
\newblock A stochastic approximation method.
\newblock {\it The Annals of Mathematical Statistics\/}  400--407.

\bibitem[{Rockafellar(1970)}]{rockafellar1970convex}
Rockafellar, Ralph~Tyrell. 1970.
\newblock {\it Convex analysis\/}.
\newblock Princeton university press.

\bibitem[{Scutari et~al.(2009)Scutari, Palomar, and
  Barbarossa}]{scutari2009mimo}
Scutari, Gesualdo, Daniel~P Palomar, Sergio Barbarossa. 2009.
\newblock The {MIMO} iterative waterfilling algorithm.
\newblock {\it IEEE Transactions on Signal Processing\/} {\bf 57} 1917--1935.

\bibitem[{Scutari et~al.(2010)Scutari, Palomar, Facchinei, and
  Pang}]{scutari2010convex}
Scutari, Gesualdo, Daniel~P Palomar, Francisco Facchinei, Jong-shi Pang. 2010.
\newblock Convex optimization, game theory, and variational inequality theory.
\newblock {\it IEEE Signal Processing Magazine\/} {\bf 27} 35--49.

\bibitem[{Shi et~al.(2015)Shi, Ling, Wu, and Yin}]{shi2015extra}
Shi, Wei, Qing Ling, Gang Wu, Wotao Yin. 2015.
\newblock Extra: An exact first-order algorithm for decentralized consensus
  optimization.
\newblock {\it SIAM Journal on Optimization\/} {\bf 25} 944--966.

\bibitem[{Telatar(1999)}]{telatar1999capacity}
Telatar, Emre. 1999.
\newblock Capacity of multi-antenna {G}aussian channels.
\newblock {\it Transactions on Emerging Telecommunications Technologies\/} {\bf
  10} 585--595.

\bibitem[{Tsuda et~al.(2005)Tsuda, R{\"a}tsch, and Warmuth}]{tsuda2005matrix}
Tsuda, Koji, Gunnar R{\"a}tsch, Manfred~K Warmuth. 2005.
\newblock Matrix exponentiated gradient updates for on-line learning and
  {B}regman projection.
\newblock {\it Journal of Machine Learning Research\/} {\bf 6} 995--1018.

\bibitem[{Vandenberghe et~al.(1998)Vandenberghe, Boyd, and
  Wu}]{vandenberghe1998determinant}
Vandenberghe, Lieven, Stephen Boyd, Shao-Po Wu. 1998.
\newblock Determinant maximization with linear matrix inequality constraints.
\newblock {\it SIAM journal on matrix analysis and applications\/} {\bf 19}
  499--533.

\bibitem[{Vedral(2002)}]{vedral2002role}
Vedral, Vlatko. 2002.
\newblock The role of relative entropy in quantum information theory.
\newblock {\it Reviews of Modern Physics\/} {\bf 74} 197.

\bibitem[{Watkins(1974)}]{watkins1974convex}
Watkins, William. 1974.
\newblock Convex matrix functions.
\newblock {\it Proceedings of the American Mathematical Society\/}  31--34.

\bibitem[{Xi et~al.(2014)Xi, Wu, and Khan}]{xi2014distributed}
Xi, Chenguang, Qiong Wu, Usman~A Khan. 2014.
\newblock Distributed mirror descent over directed graphs.
\newblock {\it arXiv preprint arXiv:1412.5526\/} .

\bibitem[{Xiao and Boyd(2006)}]{xiao2006optimal}
Xiao, Lin, Stephen Boyd. 2006.
\newblock Optimal scaling of a gradient method for distributed resource
  allocation.
\newblock {\it Journal of optimization theory and applications\/} {\bf 129}
  469--488.

\bibitem[{Yousefian et~al.(2017)Yousefian, Nedi{\'c}, and
  Shanbhag}]{yousefian2017smoothing}
Yousefian, Farzad, Angelia Nedi{\'c}, Uday~V. Shanbhag. 2017.
\newblock On smoothing, regularization, and averaging in stochastic
  approximation methods for stochastic variational inequality problems.
\newblock {\it Mathematical Programming\/} {\bf 165} 391--431.

\bibitem[{Yousefian et~al.(2018)Yousefian, Nedi{\'c}, and
  Shanbhag}]{yousefian2016stochastic}
Yousefian, Farzad, Angelia Nedi{\'c}, Uday~V. Shanbhag. 2018.
\newblock On stochastic mirror-prox algorithms for stochastic {C}artesian
  variational inequalities: randomized block coordinate and optimal averaging
  schemes.
\newblock {\it Set-Valued and Variational Analysis\/} {\bf 26} 789--819.

\bibitem[{Yu(2013)}]{yu2013strong}
Yu, Yao-Liang. 2013.
\newblock The strong convexity of von {N}eumann's entropy .

\end{thebibliography}
\section{Appendix}
\ks{We make use of the following lemma in some proofs.} %Note that $ \mathbb R^{n \times n}$ is a vector space with dimension $n^2$ \cite{axler1997linear}. 
\begin{lemma}\label{lemma:change}
%[]
Let $[X]_{uv}$ denotes the elements of matrix $X$. If we rewrite matrices $X$, $Z$ and $\nabla_{{X}}f({X})$ as vectors $x=\left([X]_{11},\ldots,[X]_{nn}\right)^T$, $z=\left([z]_{11},\ldots,[z]_{nn}\right)^T$, and $\nabla f({x})= \left([\nabla_{{  X}}f({  X})]_{11}, \ldots, [\nabla_{{X}}f({X})]_{nn}\right)^T$ respectively, it is trivial that
\begin{align*}
	(z-{x})^T\nabla f({x})= \sum_{u}\sum_{v}[(Z-X)]_{uv}[\nabla_{X} f(X)]_{uv}=\tr{(Z-{X})^T\nabla_{{X}}f({X})},
\end{align*}
where the last inequality follows by relation $\tr{  A^T   B}=\sum_{u}\sum_{v}[{  A} ]_{uv}[  B]_{uv}$.
\end{lemma}

\textit{Proof of Lemma \ref{lemma:optimality}:}

($\Rightarrow$) Assume $\widetilde{  X}^*$ is optimal to Problem \eqref{eqn:problem2}. Assume by contradiction, there exists some $\hat{  Z} \in \mathcal B$ such that $\tr{(\hat{  Z}-\widetilde {  X}^*)^T\nabla_{\widetilde{  X}}f(\widetilde{  X}^*)} < 0$. Since $f$ is continuously differentiable, by the first-order Taylor expansion, for all sufficiently small $0<\alpha<1$, we have 
\begin{align*}
f(\widetilde{  X}^*+\alpha (\hat{  Z}-\widetilde {  X}^*))= f({  X}^*) +\tr{(\hat{  Z}-\widetilde {  X}^*)^T \nabla_{\widetilde{  X}}f(\widetilde{  X}^*)}+ o(\alpha) < f({  X}^*),
\end{align*}
following the hypothesis $\tr{(\hat{  Z}-\widetilde {  X}^*)^T\nabla_{\widetilde{  X}}f(\widetilde{  X}^*)} < 0$. Since $\mathcal B$ is convex and $  X^*,~\hat{  Z} \in \mathcal B$, we have $\widetilde{  X}^*+\alpha (\hat{  Z}-\widetilde {  X}^*) \in \mathcal B$ with smaller objective function value than the optimal matrix $\widetilde{  X}^*$. This is a contradiction. Therefore, we must have $\tr{(  Z-\widetilde {  X}^*)^T \nabla_{\widetilde{  X}} f(\widetilde{  X}^*)} \geq 0$ for all $  Z\in \mathcal B$. 
\\($\Leftarrow$) Now suppose that  $\widetilde{  X}^* \in \mathcal B$ and $\tr{(  Z-\widetilde {  X}^*)^T \nabla_{\widetilde{  X}} f(\widetilde{  X}^*)} \geq 0$ for all $   Z\in \mathcal B$. Since $f$ is convex and by Lemma \ref{lemma:change}, we have 
\begin{align*}
	f(\widetilde {  X}^*) +\tr{(  Z-\widetilde {  X}^*)^T \nabla_{\widetilde{  X}} f(\widetilde{  X}^*)} \leq f(  Z) ,\quad \text{for all} \quad   Z\in \mathcal B,
\end{align*}
which implies for all $Z\in \mathcal B$, 
\begin{align*}
 f(  Z)-f(\widetilde {  X}^*) \geq 	\tr{(  Z-\widetilde {  X}^*)^T \nabla_{\widetilde{  X}} f(\widetilde{  X}^*)} \geq 0,
\end{align*}
where the last inequality follows by the hypothesis. Since $\widetilde {  X}^* \in \mathcal B$, it follows that $\widetilde {  X}^*$ is optimal. 

\textit{Proof of Lemma \ref{lemma:average}:}

We use induction to prove \eqref{eq:avg}. It is trivial that it holds for $t=0$, since $\overline{  X}_{i,0}={  X}_{i,0}$. Assume \eqref{eq:avg} holds for $t$. From \eqref{eq:gamma}, $\Gamma_t=\sum_{k'=0}^{t}\eta_{k'}$ which results in $\overline{  X}_{i,t}=\frac{\sum_{k=0}^{t} \eta_k  X_{i,k}}{  \Gamma_t}$. From \eqref{eq:gamma}, we have 
\begin{align*}
	&\overline{  X}_{i,t+1}:=\frac{\Gamma_t\overline{  X}_{i,t}+\eta_{t+1}{  X}_{i,t+1}}{ \Gamma_{t+1}}=\frac{\sum_{k=0}^{t} \eta_k  X_{i,k}+\eta_{t+1}{  X}_{i,t+1}}{\Gamma_{t+1}}=\frac{\sum_{k=0}^{t+1} \eta_k  X_{i,k}}{\sum_{k'=0}^{t+1} \eta_k'}. 
\end{align*}

\textit{Proof of Lemma \ref{pre:smoothstrong}:}

Using the Fenchel coupling definition,
\begin{align}
\label{eq:pre2-2}
	H(  {X},  {Y}+  {Z}) = \omega(  {X})+\omega^*(  {Y}+  Z)-\tr{  {X^T}(  {Y}+  Z)}.
\end{align}
By strong convexity of $\omega$ w.r.t. trace norm (Lemma \ref{lm:strong}) and using duality between strong convexity and strong smoothness \cite{kakade2009duality}, %we conclude that the conjugate 
$\omega^*$ is 1-strongly smooth w.r.t. the spectral norm, i.e., $\omega^*(  {Y}+  Z) \leq \omega^*(  {Y})+\tr{  Z^T \nabla \omega^*(  {Y})}+\Vert   Z\Vert^2_2.$
By plugging this inequality into \eqref{eq:pre2-2} we have
	\begin{align*}
H(  {X},  {Y}+  {Z}) &\leq \omega(  {X}) + \omega^*(  {Y})+ \tr{  Z^T \nabla \omega^*(  {Y})}+\Vert   Z\Vert^2_2-\tr{  {X^T}{Y}}-\tr{  {X^T}{Z}}\\&=H(  {X},  {Y})+\tr{  Z^T(\nabla \omega^*(  {Y})-  {X})}+\Vert   Z\Vert^2_2,
	\end{align*}
where in the last relation, we used \eqref{eq:fenchel}. 

\textit{Proof of Lemma \ref{lem:concaveconvex}:}

We use the following definitions in the proof. 
\begin{defn} [Matrix convex function]
\label{def:matrix} Let $\mathbb C^n$ be the complex vector space.
\item [(a)] An arbitrary matrix $ A  \in \mathbb H_m$ is nonnegative if $ (A y)^\dagger y \geq 0$ for all $y \in\mathbb C^n$.  
\item [(b)] For $ A ,  B  \in \mathbb H_m$ we write $ A \geq  B $ if $ A - B $ is nonnegative. 
\item [(c)] A function $f:\mathbb H_m\rightarrow \mathbb H_n$ is convex if $f(\lambda  A + (1-\lambda)  B )\leq \lambda f( A )+(1-\lambda)f( B )$, for all $0\leq \lambda \leq 1$. 
\item [(d)] A function $f:\mathbb H_m\rightarrow \mathbb H_n$ is called matrix monotone increasing if $ A \geq  B $ implies $f( A )\geq f( B )$ (\cite{watkins1974convex}).
 \item [(e)] A function $f:\mathbb H_m\rightarrow \mathbb R$ is called matrix monotone increasing if $ A \geq  B $ implies $f( A )\geq f( B )$ (\cite{kwong1989some}).
\end{defn}
\begin{proof}
Assume that $X ,  Z  \in \mathbb H_m$, and $0 \leq \lambda\leq 1$. By convexity of $\mathbb H_m$, we have $\lambda X  + (1 - \lambda) Z  \in \mathbb H_m$, and from concavity of $g$, we have
\begin{align}
	g(\lambda X + (1-\lambda)  Z )\geq \lambda g(X )+(1-\lambda)g( Z ).
\end{align}
 Since $h$ is matrix monotone increasing and by Definition \ref{def:matrix}(e), we get
\begin{align}
		h\left(g(\lambda X + (1-\lambda)  Z )\right)\geq h\left(\lambda g(X )+(1-\lambda)g( Z )\right) \geq \lambda h( g(X ))+(1-\lambda) h (g( Z )),
\end{align}
where the last inequality follows from concavity of $h$. Therefore,
\begin{align}
	h\left(g(\lambda X + (1-\lambda)  Z )\right) \geq \lambda h( g(X ))+(1-\lambda) h (g( Z )),
\end{align}
and we conclude that $f$ is a concave function. 
\end{proof}

\textit{Proof of Lemma \ref{lem:gradientmonotone}:}

By convexity of $f$ and by Lemma \ref{lemma:change}, we have for arbitrary $X , Z \in \mathcal X$
\begin{align*}
	f( { Z }) +\tr{(X - { Z })^T\nabla_{{ Z }} f({ Z })} \leq f(X ) .
\end{align*}
By choosing the points in reverse, we also have 
\begin{align*}
f( {X }) +\tr{( Z - {X })^T \nabla_{{X }}f({X })} \leq f( Z ).
\end{align*}
Summing the above inequalities, we get
\begin{align*}
	f( Z )+f({X })+\tr{(X -{ Z })^T\nabla_{{ Z }} f({ Z })}+\tr{( Z - {X })^T\nabla_{{X }} f({X })}\leq f(X )+f( Z ),
\end{align*}
and using the fact that $\tr{ A + B }=\tr{ A }+\tr{ B }$, we get the desired result.

\textit{Proof of Lemma \ref{lm:gap properties}:}

$(i)$ For an arbitrary $X  \in \mathcal X$, we have
\begin{align*}
		Gap( X )= \underset { Z  \in \mathcal X} {\sup}\ \tr{( X - Z )^TF(X )} \geq \tr{( X - A )^T F(X )}, \quad \text{for all}~  A  \in  \mathcal X.
\end{align*}
For $ A = X $, the above inequality suggests that $Gap( X )\geq \tr{( X - X )^TF(X )}=0$ implying that the function $Gap( X )$ is nonnegative for all $ X  \in  \mathcal X$. \\
$(ii)$ Assume $X ^*$ is a strong solution. By definition of VI($\mathbf{\mathcal X}, F$) and relation \eqref{eq:VI}, we have
\begin{align*}
	\tr{( {X^*}- X )^T F(X ^*)} \leq 0, \quad \text{for all}~ X  \in \mathbf{\mathcal X}
\end{align*}
which implies 
\begin{align*}
Gap( {X^*})= \underset {X  \in \mathcal X} {\sup}\ \tr{( {X^*}- X )^T F(X ^*)}	 \leq 0, \quad \text{for all}~ X  \in \mathbf{\mathcal X}.
\end{align*}
On the other hand, from Lemma \ref{lm:gap properties}$(i)$, we get $Gap( {X^*}) \geq 0$. We conclude that for any strong solution $ {X^*}$, we have $Gap( {X^*}) = 0$. 
Conversely, assume that there exist an $ X $ such that $Gap( X ) = 0$. Therefore, $\underset { Z  \in \mathcal X} {\sup}\ \tr{( X - Z )^T F(X )}=0$ which implies $\tr{( X - Z )^T F(X )} \leq 0$ for all $ Z  \in \mathcal X$. Equivalently, we get $\tr{( Z - X )^T F(X )} \geq 0$ for all $ Z  \in \mathcal X$ implying $ X $ is a strong solution.

\textit{Proof of Lemma \ref{lem:strong2}:}

For an arbitrary matrix $ A $, we have $\nabla_ A   \tr{ A ^T A }= A $ (\cite{athans1965gradient}, page 32). That being said and using the definition of Frobenius norm, we have  
\begin{align*}
&\frac{1}{2}\Vert  A \Vert_F^2 +\tr{\nabla^T_{ A } h( A ) ( B - A )}+\frac{1}{2}\Vert  A - B \Vert_F^2=
\\
	 &\frac{1}{2}\Vert  A \Vert_F^2 +\tr{{ A }^T ( B - A )}+\frac{1}{2}\tr{( A - B )^T( A - B )}=
		\\
		& \frac{1}{2}\Vert  A \Vert_F^2 +\tr{{ A }^T ( B - A )}+\frac{1}{2}\tr{ A ^T  A - B ^T A - A ^T B + B ^T B }=
		\\
		&\frac{1}{2}\Vert  A \Vert_F^2 +\tr{ A ^T  B - A ^T A }+\frac{1}{2}\tr{ A ^T  A - B ^T A - A ^T B + B ^T B }=
		\\
	& \frac{1}{2}\Vert  A \Vert_F^2 +\frac{1}{2}\tr{ A ^T  B }-\frac{1}{2}\tr{ A ^T A }-\frac{1}{2}\tr{ B ^T A }+\frac{1}{2}\tr{ B ^T B }=
	\\
	&\frac{1}{2}\Vert  A \Vert_F^2 +\frac{1}{2}\tr{ A ^T  B }-\frac{1}{2}\Vert  A \Vert_F^2-\frac{1}{2}\tr{ A ^T  B } +\frac{1}{2}\Vert  B \Vert_F^2 = \frac{1}{2}\Vert  B \Vert_F^2.
\end{align*}
Therefore, the inequality \eqref{eq:strong2} holds in equality and we conclude that $h( A )$ is strongly convex with parameter 1. 
\end{document}